\documentclass[conference]{IEEEtran}

\usepackage{amsmath, amsthm,amsfonts,amssymb,mathtools,mathdots,scalerel,mathrsfs,stmaryrd,cmll}
\usepackage[utf8]{inputenc}
\usepackage[T1]{fontenc}
\usepackage{tikz-cd, quiver, adjustbox}

\usepackage{enumitem, tocloft, csquotes}
\usepackage[hidelinks]{hyperref}
\usepackage{tikz, fancyhdr, xparse, xcolor}
\usepackage[british]{babel}

\usepackage{etoolbox}
\makeatletter
\patchcmd{\ttlh@hang}{\parindent\z@}{\parindent\z@\leavevmode}{}{}
\patchcmd{\ttlh@hang}{\noindent}{}{}{}
\makeatother

\newcommand\eqdef{\coloneqq}
\newcommand\nbd{\nobreakdash-\hspace{0pt}}

\newcommand\incl{\hookrightarrow}
\newcommand\incliso{\stackrel{\sim}{\hookrightarrow}}

\newcommand\restr[2]{{#1}{\raisebox{0pt}{$|_{#2}$}}}

\newcommand\set[1]{\left\{ {#1} \right\}}
\newcommand\size[1]{\left|{#1}\right|}

\newcommand\cat[1]{\mathbf{#1}}

\newcommand\fun[1]{\mathsf{#1}}

\newcommand\order[2]{#2^{(#1)}}

\newcommand\ogpos{\cat{ogPos}}

\newcommand\poscat{\cat{Pos}}

\newcommand\hasse[1]{\mathscr{H}{#1}}
\newcommand\hasseo[1]{\mathring{\mathscr{H}}{#1}}
\newcommand\clos[1]{\mathrm{cl}\,#1}
\newcommand\clset[1]{\mathrm{cl}\set{#1}}

\newcommand\maxel[1]{\mathscr{M}\!\textit{ax}\,#1}
\newcommand\grade[2]{#2_{#1}}
\newcommand\lydim[1]{\mathrm{lydim}\, {#1}}
\newcommand\frdim[1]{\mathrm{frdim}\, {#1}}

\newcommand\bound[2]{\partial_{#1}^{#2}}
\newcommand\faces[2]{\Delta_{#1}^{#2}}
\newcommand\cofaces[2]{\nabla_{#1}^{#2}}

\newcommand\cp[1]{\,{\scriptstyle\#}_{#1}\,}

\newcommand\cpiso[2]{\,{\scriptstyle\#}^{#2}_{#1}\,}
\newcommand\celto{\Rightarrow}
\newcommand\compos[1]{\langle#1\rangle}

\newcommand\submol{\sqsubseteq}

\newcommand\subs[3]{#1[#2/#3]}

\newcommand\flow[2]{\mathscr{F}_{#1}{#2}}
\newcommand\maxflow[2]{\mathscr{M}_{#1}{#2}}

\newcommand\layerings[2]{\mathscr{L}\!\textit{ay}_{#1}{#2}}
\newcommand\orderings[2]{\mathscr{O}\!\textit{rd}_{#1}{#2}}
\newcommand\lto[2]{\fun{m}_{#1, #2}}

\newcommand\thearrow[1]{{#1}\data{arrow}}
\newcommand\prech{\preceq_0}
\newcommand\precv{\preceq_2}

\newcommand\edges[2]{\grade{#1}{\mathscr{E}}#2}

\newcommand\data[1]{\mathsf{#1}}

\newtheoremstyle{ittheorem}
  {\topsep}   
  {\topsep}   
  {\itshape}  
  {0pt}       
  {\itshape \bfseries} 
  {.}         
  {5pt plus 1pt minus 1pt} 
  {}          

\newtheoremstyle{itdfn}
  {\topsep}   
  {\topsep}   
  {}  
  {0pt}       
  {\bfseries} 
  {.}         
  {5pt plus 1pt minus 1pt} 
  {}          

\makeatletter
  \renewcommand\@upn{\textit}
\makeatother

\theoremstyle{ittheorem}
\newtheorem{thm}{Theorem}
\newtheorem{prop}[thm]{Proposition}
\newtheorem{cor}[thm]{Corollary}
\newtheorem{lem}[thm]{Lemma}

\theoremstyle{itdfn}
\newtheorem{dfn}[thm]{}
\theoremstyle{remark}
\newtheorem{rmk}[thm]{Remark}

\newtheorem{exm}[thm]{Example}

\begin{document}


\title{Higher-Dimensional Subdiagram Matching}

\author{
\IEEEauthorblockN{Amar Hadzihasanovic}
\IEEEauthorblockA{$^1$ Department of Software Science\\
Tallinn University of Technology\\
$^2$ Quantinuum, 17 Beaumont Street\\
Oxford OX1 2NA, United Kingdom}
\and
\IEEEauthorblockN{Diana Kessler}
\IEEEauthorblockA{Department of Software Science\\
Tallinn University of Technology}
}

\IEEEoverridecommandlockouts
\IEEEpubid{\makebox[\columnwidth]{Accepted for LICS 2023 \hfill} \hspace{\columnsep}\makebox[\columnwidth]{ }}

\maketitle 

\begin{abstract}
Higher-dimensional rewriting is founded on a duality of rewrite systems and cell complexes, connecting computational mathematics to higher categories and homotopy theory: the two sides of a rewrite rule are two halves of the boundary of an $(n+1)$\nbd cell, which are diagrams of $n$\nbd cells.
We study higher-dimensional diagram rewriting as a mechanism of computation, focussing on the matching problem for rewritable subdiagrams within the combinatorial framework of diagrammatic sets.
We provide an algorithm for subdiagram matching in arbitrary dimensions, based on new results on layerings of diagrams, and derive upper bounds on its time complexity.
We show that these superpolynomial bounds can be improved to polynomial bounds under certain acyclicity conditions, and that these conditions hold in general for diagrams up to dimension 3.
We discuss the challenges that arise in dimension 4.
\end{abstract}


\section*{Introduction}

Higher-dimensional rewriting \cite{burroni1993higher, guiraud2019rewriting} is founded, as a field, on the observation that different varieties of rewrite systems are  instances of \emph{directed cell complexes} in different dimensions.
These, in turn, can be seen as presentations of higher or monoidal categories and groupoids, situating objects traditionally associated with syntactic and quantitative aspects of computation in the same universe as objects traditionally associated with semantic and logical aspects.

Informally, a directed cell complex is an object assembled from ``directed $n$\nbd cells'', models of topological $n$\nbd balls whose boundary, an $(n-1)$\nbd sphere, is subdivided into an \emph{input} and an \emph{output} half, modelling the left and right-hand sides of a rewrite rule.
A 1\nbd dimensional directed cell complex is a directed graph, which is the same as an abstract rewrite system up to interpretive nuances.
A 2\nbd dimensional directed cell complex is the presentation with ``oriented equations'' of a category; when the category has a single object, it is the presentation of a \emph{monoid}, also known as a string rewrite system.
One dimension higher, we find presentations with oriented equations of monoidal categories, which subsume, \emph{via} functorial semantics \cite{lawvere1963functorial}, term rewrite systems such as presentations of algebraic theories.

The duality of rewrite systems and higher structures has, so far, been leveraged mostly on the side of higher algebra, for example in the study of homotopically coherent presentations of algebras \cite{gaussent2015coherent} or the development of proof assistants for homotopy theory and higher category theory \cite{reutter2019high}.
We would like to argue that, also \emph{as a mechanism of computation}, higher-dimensional rewriting has some unique characteristics which may be significant in fundamental computer science, yet have remained underexplored.

One of these is a kind of \emph{uniformity of data and computations}.
In most models, computations or executions are objects of different nature from the data that they manipulate: for example, sequences of configurations as opposed to terms or strings of characters.
There is no \emph{internal} way to manipulate computations as data, that does not require encoding in some external ``reference'' machine, to which the same consideration applies.
In higher-dimensional rewriting, on the other hand, data is given in the form of a \emph{diagram} (sometimes called a \emph{pasting diagram}) of $n$\nbd cells, and a computation is then embodied by \emph{a diagram of $(n+1)$\nbd cells}, which is naturally the data of a computation one dimension above.
This suggests that higher-dimensional rewriting may, for example, be a natural framework for the formal study of simulations of machines by other machines.

Another point of interest is that a higher-dimensional rewrite system presents not only the admissible data and computations, but also \emph{the space in which computations happen}.
The possibility of parallelism, or whether the data is accessed as a stack or in free order, become internal topological features rather than externally imposed constraints.
This is particularly interesting in relation to quantum models where the topological features themselves drive the computation \cite{freedman2003topological}.

This article is an effort to establish some fundamental facts about higher-dimensional rewriting as a mechanism of computation.
In particular, we try to answer the following question: \emph{is a machine that operates by higher-dimensional rewriting a ``reasonable'' machine, according to the standards of computational complexity theory?}
More precisely: is the obvious cost model that attributes constant cost to each rewrite step a ``reasonable'' cost model?

The basic computational step of any such machine may be described as follows.
The machine has a list $(r_i)_{i=1}^m$ of rewrite rules, which are $(n+1)$\nbd cells, and whose input boundaries $(\bound{}{-}r_i)_{i=1}^m$ and output boundaries $(\bound{}{+}r_i)_{i=1}^m$ are $n$\nbd dimensional diagrams.
Given an $n$\nbd dimensional diagram $t$ as input, the machine tries to match one of the input boundaries with a rewritable \emph{subdiagram} of $t$.
If it finds a match with $\bound{}{-}r_i$ for some $i \in \set{1, \ldots, m}$, it substitutes $\bound{}{+}r_i$ for the match in $t$; otherwise it stops.

Evidently, our question is answered in the affirmative for such a machine if and only if the \emph{subdiagram matching problem} is feasible in dimension $n$, which we read as: \emph{admits a (preferably low-degree) polynomial-time algorithm with respect to a reasonable size measure for diagrams}.
Since a cognate problem such as subgraph matching is notoriously $\fun{NP}$-complete, it is not at all obvious that this should be true.

With this in mind, in this article we study the higher-dimensional subdiagram matching problem.
As a first step, we must fix a particular model of higher-dimensional diagrams.
We adopt the \emph{diagrammatic set} model defined by the first author \cite{hadzihasanovic2020diagrammatic} after a combinatorial approach to diagrams started by Steiner \cite{steiner1993algebra}.
This model combines expressiveness with the property of \emph{topological soundness}: diagrams admit a functorial interpretation as homotopies in CW complexes.
Furthermore, it uses inductive data structures that are suitable for computation; we began a study of their algorithms and complexity in \cite{hadzihasanovic2022data}.

Our main contribution is an algorithm for subdiagram matching in arbitrary dimension, relying on new combinatorial results on \emph{layerings} of diagrams (decompositions into layers containing each a single cell of highest dimension).
Only one stage in this algorithm takes superpolynomial time in the worst 
case.
We show that this can be avoided under certain acyclicity conditions on diagrams, which hold in general up to dimension 3.
We derive that subdiagram matching is feasible up to dimension 3, which should cover most current applications of higher-dimensional rewriting.
We then discuss the case of dimension 4 and higher, showing through a counterexample why there is no obvious patch to the algorithm.
We leave the question of feasibility in higher dimensions open.

\subsection*{Related work}

Complexity-theoretic aspects of higher-dimensional rewriting in the proper sense were considered by Bonfante and Guiraud in \cite{bonfante2009polygraphic}, but this work focussed only on a particular class of 2\nbd dimensional rewrite systems.

Other works focus on \emph{string diagram rewriting}, which is connected to higher-dimensional rewriting in that both of them have semantics in higher and monoidal categories, but the flavour is altogether different: models of string diagram rewrite systems are typically 2\nbd dimensional but have extra structure which makes the diagrams ``graph-like'', and it can be convenient to reflect that in the data structures.
Works in this vein include the series by Bonchi, Gadducci, Kissinger, Sobocinski, and Zanasi \cite{bonchi2022stringI, bonchi2022string, bonchi2021stringIII} with a focus on rewriting-theoretic questions, and articles by Delpeuch and Vicary \cite{delpeuch2021word, vicary2022normalization} with a more complexity-theoretic focus.

\subsection*{Structure of the article}

Section \ref{sec:structures} presents the combinatorial framework together with the data structures used to represent diagrams.
It also provides an improved complexity upper bound for the diagram isomorphism problem.
Section \ref{sec:matching} defines rewritable subdiagrams and their matching problem, which can be split into subproblems.
It then deals with the first subproblem, which is to match the shape of a diagram in another diagram, irrespective of whether it is a subdiagram.
Section \ref{sec:rewritable} deals with the second subproblem, which is to recognise which matches are, in fact, rewritable subdiagrams.
It presents a theory of \emph{layerings} and \emph{orderings} of diagrams as a way to an algorithm solving this problem in any dimension.
This algorithm has only a superpolynomial upper bound, but it is shown that it can be improved to a polynomial bound under certain acyclicity conditions.
Section \ref{sec:lowdim} shows that these acyclicity conditions hold automatically up to dimension 3, and the algorithms can be further simplified in low dimensions.
It then discusses an example in dimension 4 which highlights why the strategies that work in low dimensions do not have obvious extensions to higher dimensions.

Proofs of combinatorial results are attached in the appendix.
A full development will be presented in a forthcoming technical monograph \cite{hadzihasanovicpasting}.
\section{The data structures} \label{sec:structures}

\begin{dfn}
This section is for the largest part a recap of \cite{hadzihasanovic2022data}.
We refer the reader there for more details.
\end{dfn}

\begin{dfn}
In the framework of diagrammatic sets, a diagram $t$ is specified by the data of
\begin{enumerate}
    \item its \emph{shape} $U$,
    \item a labelling $t\colon U \to \mathbb{V}$ in a set of variables.
\end{enumerate}
The shape of a diagram records its cells, together with the information of which $(n-1)$\nbd dimensional cells are located in the input or output half of the boundary of an $n$\nbd dimensional cell.
This is similar to the data of an abstract polytope or polytopal complex, but comes with additional orientation data.
We present these data in the form of an \emph{oriented graded poset}.
\end{dfn}

\begin{dfn}[Covering relation]
Let $P$ be a finite poset with order relation $\leq$.
Given elements $x, y \in P$, we say that $y$ \emph{covers} $x$ if $x < y$ and, for all $y' \in P$, if $x < y' \leq y$ then $y' = y$. 
\end{dfn}

\begin{dfn}[Hasse diagram]
Let $P$ be a finite poset.
The \emph{Hasse diagram} of $P$ is the directed acyclic graph $\hasse{P}$ whose \begin{itemize}
    \item set of vertices is the underlying set of $P$, and 
    \item for all vertices $x, y$, there is an edge from $y$ to $x$ if and only if $y$ covers $x$ in $P$.
\end{itemize}
\end{dfn}

\begin{dfn}[Graded poset]
Let $P$ be a finite poset.
We say that $P$ is \emph{graded} if, for all $x \in P$, all maximal paths starting from $x$ in $\hasse{P}$ have the same length.
\end{dfn}

\begin{dfn}[Dimension of an element]
Let $P$ be a graded poset and $x \in P$.
The \emph{dimension} of $x$ is the length $\dim{x}$ of a maximal path starting from $x$ in $\hasse{P}$.
For each $U \subseteq P$ and $n \in \mathbb{N}$, we write $\grade{n}{U} \eqdef \set{ x \in U \mid \dim x = n }$.
\end{dfn}

\begin{dfn}[Oriented graded poset]
Let $P$ be a finite poset.
An \emph{orientation} on $P$ is an edge-labelling of $\hasse{P}$ with values in $\set{ +, - }$.
An \emph{oriented graded poset} is a graded poset $P$ together with an orientation on $P$.
\end{dfn}

\begin{dfn}[Faces and cofaces]
Let $P$ be an oriented graded poset, $x \in P$, $\alpha \in \set{ +, - }$.
The set of \emph{input} ($\alpha = -$) or \emph{output} ($\alpha = +$) \emph{faces} of $x$ is
\begin{equation*}
    \faces{}{\alpha} x \eqdef \set{ y \in P \mid \text{$x$ covers $y$ with orientation $\alpha$}  }.
\end{equation*}
The set of \emph{input} ($\alpha = -$) or \emph{output} ($\alpha = +$) \emph{cofaces} of $x$ is
\begin{equation*}
    \cofaces{}{\alpha} x \eqdef \set{ y \in P \mid \text{$y$ covers $x$ with orientation $\alpha$}  }.
\end{equation*}
We let $\faces{}{} x \eqdef \faces{}{-} x \cup \faces{}{+} x$ and $\cofaces{}{} x \eqdef \cofaces{}{-} x \cup \cofaces{}{+} x$.
\end{dfn}

\begin{dfn}[Oriented Hasse diagram]
Let $P$ be an oriented graded poset.
The \emph{oriented Hasse diagram of $P$} is the directed graph $\hasseo{P}$ whose 
\begin{itemize}
    \item set of vertices is the underlying set of $P$, and 
    \item for all vertices $x, y$, there is an edge from $y$ to $x$ if and only if $y \in \faces{}{-}x$ or $x \in \faces{}{+}y$.
\end{itemize}
\end{dfn}

\begin{dfn} \label{data:ogposet}
To represent an oriented graded poset $P$, we linearly order $\grade{n}{P}$ in each dimension $n$, so that each element $x \in P$ is uniquely identified by a pair of integers $(n, k)$, where $n \eqdef \dim{x}$ and $k$ is the position of $x$ in the linear order on $\grade{n}{P}$.
Then we represent $P$ as a pair $(\data{face\_data}, \data{coface\_data})$ of arrays of arrays of pairs of sets of integers, where
\begin{enumerate}
	\item $j \in \data{face\_data}[n][k][i]$ iff $(n-1, j) \in \faces{}{\alpha(i)}(n, k)$,
	\item $j \in \data{coface\_data}[n][k][i]$ iff $(n+1, j) \in \cofaces{}{\alpha(i)}(n, k)$;
\end{enumerate}
here the index $i \in \set{0, 1}$ is used to encode pairs, $\alpha(0) \eqdef -$ and $\alpha(1) \eqdef +$.
Sets of integers may be implemented as any data type supporting binary search in logarithmic time.
Note that this representation is redundant: $\data{face\_data}$ and $\data{coface\_data}$ can be reconstructed from each other.

This is essentially an adjacency list representation of $\hasse{P}$, with vertices separated according to their dimension, and incoming and outgoing edges separated according to their label.
It is not unique: any permutation of the dimension-wise linear orders produces an equivalent representation.
\end{dfn}

\begin{exm} \label{exm:whisker}
The following are representations of the same 2\nbd dimensional diagram shape as
\begin{itemize}
    \item a typical drawing of a pasting diagram;
    \item an oriented Hasse diagram, with input edges in pink, and dimension increasing from bottom to top;
    \item the pair of $\data{face\_data}$ and $\data{coface\_data}$ (rows are outer array indices, increasing from top to bottom, and columns inner array indices, increasing from left to right).
\end{itemize}
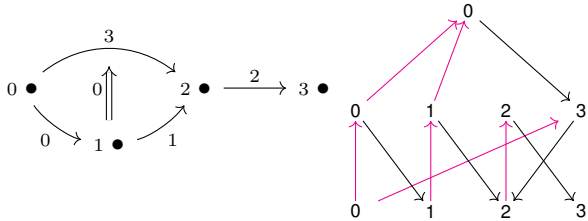
\[
\begin{tikzcd}[sep=small]
	{{\scriptstyle 0}\;\bullet} && {{\scriptstyle 2}\;\bullet} && {{\scriptstyle 3}\;\bullet} \\
	& {{\scriptstyle 1}\;\bullet}
	\arrow[""{name=0, anchor=center, inner sep=0}, "3", curve={height=-18pt}, from=1-1, to=1-3]
	\arrow["2", from=1-3, to=1-5]
	\arrow["0"', curve={height=6pt}, from=1-1, to=2-2]
	\arrow["1"', curve={height=6pt}, from=2-2, to=1-3]
	\arrow["0", shorten <=3pt, shorten >=6pt, Rightarrow, from=2-2, to=0]
\end{tikzcd}
\begin{tikzpicture}[xscale=4, yscale=4, baseline={([yshift=-.5ex]current bounding box.center)}]
\draw[->, draw=magenta] (0.125, 0.19999999999999998) -- (0.125, 0.4666666666666667);
\draw[->, draw=magenta] (0.19999999999999998, 0.19999999999999998) -- (0.8, 0.4666666666666667);
\draw[->, draw=magenta] (0.375, 0.19999999999999998) -- (0.375, 0.4666666666666667);
\draw[->, draw=magenta] (0.625, 0.19999999999999998) -- (0.625, 0.4666666666666667);
\draw[->, draw=black] (0.15, 0.4666666666666667) -- (0.35, 0.19999999999999998);
\draw[->, draw=magenta] (0.16249999999999998, 0.5333333333333333) -- (0.4625, 0.7999999999999999);
\draw[->, draw=black] (0.4, 0.4666666666666667) -- (0.6, 0.19999999999999998);
\draw[->, draw=magenta] (0.3875, 0.5333333333333333) -- (0.4875, 0.7999999999999999);
\draw[->, draw=black] (0.65, 0.4666666666666667) -- (0.85, 0.19999999999999998);
\draw[->, draw=black] (0.85, 0.4666666666666667) -- (0.65, 0.19999999999999998);
\draw[->, draw=black] (0.5375, 0.7999999999999999) -- (0.8375, 0.5333333333333333);
\node[text=black, font={\scriptsize \sffamily}, xshift=0pt, yshift=0pt] at (0.125, 0.16666666666666666) {0};
\node[text=black, font={\scriptsize \sffamily}, xshift=0pt, yshift=0pt] at (0.375, 0.16666666666666666) {1};
\node[text=black, font={\scriptsize \sffamily}, xshift=0pt, yshift=0pt] at (0.625, 0.16666666666666666) {2};
\node[text=black, font={\scriptsize \sffamily}, xshift=0pt, yshift=0pt] at (0.875, 0.16666666666666666) {3};
\node[text=black, font={\scriptsize \sffamily}, xshift=0pt, yshift=0pt] at (0.125, 0.5) {0};
\node[text=black, font={\scriptsize \sffamily}, xshift=0pt, yshift=0pt] at (0.375, 0.5) {1};
\node[text=black, font={\scriptsize \sffamily}, xshift=0pt, yshift=0pt] at (0.625, 0.5) {2};
\node[text=black, font={\scriptsize \sffamily}, xshift=0pt, yshift=0pt] at (0.875, 0.5) {3};
\node[text=black, font={\scriptsize \sffamily}, xshift=0pt, yshift=0pt] at (0.5, 0.8333333333333333) {0};
\end{tikzpicture}
\]
{\small \begin{tabular}{l l l l}
		\multicolumn{4}{l}{$\data{face\_data}$:} \\
		$(\set{}, \set{})$ & $(\set{}, \set{})$ & $(\set{}, \set{})$ & $(\set{}, \set{})$ \\
		$(\set{0}, \set{1})$ & $(\set{1}, \set{2})$ & $(\set{2}, \set{3})$ & $(\set{0}, \set{2})$ \\
		$(\set{0, 1}, \set{3})$ & & & \\
		\multicolumn{4}{l}{$\data{coface\_data}$:} \\
		$(\set{0, 3}, \set{})$ & $(\set{1}, \set{0})$ & $(\set{2}, \set{1, 3})$ & $(\set{}, \set{2})$ \\
		$(\set{0}, \set{})$ & $(\set{0}, \set{})$ & $(\set{}, \set{})$ & $(\set{}, \set{0})$ \\
		$(\set{}, \set{})$ & & &	
\end{tabular}}
\end{exm}

\begin{dfn}[Closure of a subset]
Let $P$ be a poset, $U \subseteq P$.
The \emph{closure of $U$} is the subset $\clos{U} \eqdef \set{ x \in P \mid \exists y \in U \, x \leq y }$.
We say that $U$ is \emph{closed} if $U = \clos{U}$.
\end{dfn}

\begin{dfn}[Dimension of a subset]
Let $U$ be a closed subset of a graded poset.
The \emph{dimension} of $U$ is the integer 
\begin{equation*}
    \dim U \eqdef
    \begin{cases}
        \max \set{ \dim x \mid x \in U } & \text{if $U$ is inhabited,} \\
        -1 & \text{if $U$ is empty.}
    \end{cases}
\end{equation*}
\end{dfn}

\begin{dfn}[Input and output boundaries]
Let $U$ be a closed subset of an oriented graded poset.
For all $\alpha \in \set{ +, - }$ and $n \in \mathbb{N}$, let
\begin{equation*}
    \faces{n}{\alpha} U \eqdef 
    \set{ x \in U_n \mid \cofaces{}{-\alpha} x \cap U = \varnothing  }.
\end{equation*}
Note that, if $\maxel{U}$ is the set of maximal elements of $U$, then
\begin{equation*}
    \faces{n}{-}U \cap \faces{n}{+}U = \grade{n}{(\maxel{U})}.
\end{equation*}
The \emph{input} ($\alpha = -$) and \emph{output} ($\alpha = +$) \emph{$n$\nbd boundary} of $U$ is the closed subset
\begin{equation*}
    \bound{n}{\alpha} U \eqdef \clos{(\faces{n}{\alpha} U)} \cup \bigcup_{k < n} \clos{\grade{k}{(\maxel{U})}}.
\end{equation*}
We omit the subscript when $n = \dim{U} - 1$, and for $n \in \set{-1, -2}$, we let $\faces{n}{\alpha}U = \bound{n}{\alpha}U \eqdef \varnothing$.
\end{dfn}

\begin{dfn}
We use the following notations, for $x$ an element in an oriented graded poset, $U$ a closed subset, $n \in \mathbb{N}$, and $\alpha \in \set{ +, - }$: 
$\bound{n}{\alpha} x \eqdef \bound{n}{\alpha} \clset{ x }$,
$\bound{n}{} U \eqdef \bound{n}{-} U \cup \bound{n}{+} U$,
$\faces{n}{} U \eqdef \faces{n}{-} U \cup \faces{n}{+} U$.
\end{dfn}

\begin{lem} \label{lem:boundary_inclusion}
Let $U$ be a closed subset of an oriented graded poset, $n \in \mathbb{N}$, and $\alpha \in \set{ +, - }$.
Then
\begin{enumerate}
    \item $\bound{n}{\alpha} U \subseteq U$,
    \item $\bound{n}{\alpha}U = U$ if and only if $n \geq \dim{U}$.
\end{enumerate}
\end{lem}

\begin{exm}
Let $U$ be the oriented graded poset of Example \ref{exm:whisker}.
Then $\bound{2}{-}U = \bound{2}{+}U = U$, and
\begin{align*}
	\bound{1}{-}U & = \set{(0, 0), (0, 1), (0, 2), (0, 3), (1, 0), (1, 1), (1, 2)}, \\
	\bound{1}{+}U & = \set{(0, 0), (0, 2), (0, 3), (1, 2), (1, 3)}, \\
	\bound{0}{-}U & = \set{(0, 0)}, \quad \quad \bound{0}{+}U = \set{(0, 3)}.
\end{align*}
\end{exm}

\begin{dfn}[Map of oriented graded posets]
Let $P, Q$ be oriented graded posets.
A \emph{map} $f\colon P \to Q$ is a function of their underlying sets that satisfies
\begin{equation*}
    f(\bound{n}{\alpha} x) = \bound{n}{\alpha}f(x)
\end{equation*}
for all $x \in P$, $n \in \mathbb{N}$, and $\alpha \in \set{ +, - }$.
There is a category $\ogpos$ whose objects are oriented graded posets and morphisms are maps.
\end{dfn}

\begin{dfn}[Inclusion of oriented graded posets]
An \emph{inclusion} is an injective map of oriented graded posets.
\end{dfn}

\begin{lem} \label{lem:properties_of_inclusions}
Let $\imath\colon P \incl Q$ be an inclusion of oriented graded posets.
Then
\begin{enumerate}
    \item $\imath$ is both order-preserving and order-reflecting, that is, $\imath(x) \leq \imath(y)$ if and only if $x \leq y$;
    \item $\imath$ preserves dimensions, that is, $\dim{\imath(x)} = \dim{x}$ for all $x \in P$;
    \item $\imath$ preserves the covering relation and orientations, that is, if $y$ covers $x$ in $P$ with orientation $\alpha$, then $\imath(y)$ covers $\imath(x)$ in $Q$ with orientation $\alpha$;
    \item for all closed $U \subseteq P$, $n \in \mathbb{N}$ and $\alpha \in \set{ +, - }$, $\imath$ maps $\bound{n}{\alpha}U$ isomorphically onto $\bound{n}{\alpha}\imath(U)$.
\end{enumerate}
\end{lem}

\begin{rmk}
Every closed subset of an oriented graded poset inherits an orientation by restriction, and its inclusion as a subset is an inclusion of oriented graded posets.
The fact that inclusions preserve all boundaries justifies us in identifying an oriented graded poset with its image through an inclusion, which we will often do implicitly.
\end{rmk}

\begin{prop} \label{prop:ogpos_limits_and_colimits}
The category $\ogpos$ has
\begin{enumerate}
    \item a terminal object $1$,
    \item an initial object $\varnothing$,
    \item pushouts of inclusions.
\end{enumerate}
\end{prop}

\begin{dfn}
Not all oriented graded posets describe well-formed diagram shapes.
In our framework, the well-formed shapes form an inductive subclass, the \emph{regular molecules}, generated by the following constructions.
\end{dfn}

\begin{dfn}[Pasting construction]
Let $U, V$ be oriented graded posets, $k \in \mathbb{N}$, and let $\varphi\colon \bound{k}{+}U \incliso \bound{k}{-}V$ be an isomorphism.
The \emph{pasting of $U$ and $V$ at the $k$\nbd boundary along $\varphi$} is the oriented graded poset $U \cpiso{k}{\varphi} V$ obtained in $\ogpos$ as the pushout
\[\begin{tikzcd}
	\bound{k}{+}U & \bound{k}{-}V & V \\
	U && U \cpiso{k}{\varphi} V.
	\arrow["\varphi", hook, from=1-1, to=1-2]
	\arrow[hook, from=1-2, to=1-3]
	\arrow[hook', from=1-1, to=2-1]
	\arrow[hook, from=2-1, to=2-3]
	\arrow[hook', from=1-3, to=2-3]
	\arrow["\lrcorner"{anchor=center, pos=0.125, rotate=180}, draw=none, from=2-3, to=1-1]
\end{tikzcd}\]
\end{dfn}

\begin{dfn}[Rewrite construction]
Let $U, V$ be oriented graded posets of the same dimension $n$, and suppose $\varphi\colon \bound{}{}U \incliso \bound{}{}V$ is an isomorphism restricting to isomorphisms $\varphi^\alpha\colon \bound{}{\alpha} U \incliso \bound{}{\alpha} V$ for each $\alpha \in \set{ +, - }$.
Construct the pushout
\[\begin{tikzcd}
	\bound{}{}U & \bound{}{}V & V \\
	U && \bound{}{}(U \celto^\varphi V)
	\arrow["\varphi", hook, from=1-1, to=1-2]
	\arrow[hook, from=1-2, to=1-3]
	\arrow[hook', from=1-1, to=2-1]
	\arrow[hook, from=2-1, to=2-3]
	\arrow[hook', from=1-3, to=2-3]
	\arrow["\lrcorner"{anchor=center, pos=0.125, rotate=180}, draw=none, from=2-3, to=1-1]
\end{tikzcd}\]
in $\ogpos$. 
The \emph{rewrite of $U$ into $V$ along $\varphi$} is the oriented graded poset $U \celto^\varphi V$ obtained by adjoining a single $(n+1)$\nbd dimensional element $\top$ to $\bound{}{}(U \celto^\varphi V)$ such that $\faces{}{-}\top \eqdef U_n$ and $\faces{}{+}\top \eqdef V_n$.
\end{dfn}

\begin{lem} \label{lem:boundaries_of_rewrite}
Let $U, V$ be oriented graded posets and suppose $U \celto^\varphi V$ is defined.
Then
\begin{enumerate}
    \item $\bound{}{-}(U \celto^\varphi V)$ is isomorphic to $U$,
    \item $\bound{}{+}(U \celto^\varphi V)$ is isomorphic to $V$.
\end{enumerate}
\end{lem}

\begin{rmk} \label{rmk:boundary_of_rewrite_notation}
The notation $\bound{}{}(U \celto^\varphi V)$ for the pushout in the rewrite construction is \emph{a posteriori} justified, that is, the pushout indeed constructs the boundary of $U \celto^\varphi V$.
\end{rmk}

\begin{dfn}[Roundness] \index{oriented graded poset!round}
Let $U$ be an oriented graded poset.
We say that $U$ is \emph{round} if, for all $n < \dim{U}$, $\bound{n}{-}U \cap \bound{n}{+}U = \bound{n-1}{}U$.
\end{dfn}

\begin{exm} \label{exm:round}
The shape of a pasting diagram is round when, intuitively, the diagram is shaped as a topological ball of the appropriate dimension.
For example, the oriented graded poset of Example \ref{exm:whisker} is not round, since
\begin{align*}
	\bound{0}{}U & = \set{(0, 0), (0, 3)} \\
    \subsetneq \bound{1}{+}U \cap \bound{1}{-}U & = \set{(0, 0), (0, 2), (1, 2), (0, 3)},
\end{align*}
and, indeed, the pasting diagram is shaped as the wedge of a 2\nbd ball (disc) with a 1\nbd ball (interval).
However, the following pasting diagram is round:
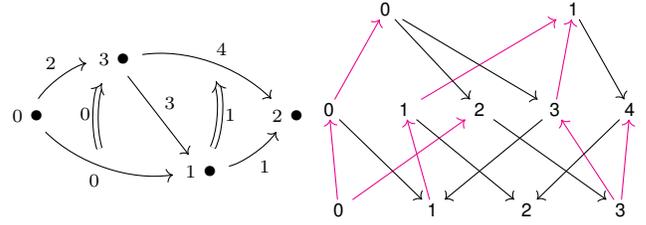
\[
    \begin{tikzcd}[sep=small]
	& {{\scriptstyle 3}\;\bullet} \\
	{{\scriptstyle 0}\;\bullet} &&& {{\scriptstyle 2}\;\bullet} \\
	&& {{\scriptstyle 1}\;\bullet}
	\arrow["2", curve={height=-6pt}, from=2-1, to=1-2]
	\arrow[""{name=0, anchor=center, inner sep=0}, "4", curve={height=-12pt}, from=1-2, to=2-4]
	\arrow[""{name=1, anchor=center, inner sep=0}, "0"', curve={height=12pt}, from=2-1, to=3-3]
	\arrow["1"', curve={height=6pt}, from=3-3, to=2-4]
	\arrow["3", from=1-2, to=3-3]
	\arrow["0", curve={height=-6pt}, shorten <=7pt, shorten >=3pt, Rightarrow, from=1, to=1-2]
	\arrow["1"', curve={height=6pt}, shorten <=3pt, shorten >=7pt, Rightarrow, from=3-3, to=0]
    \end{tikzcd}
\begin{tikzpicture}[xscale=5, yscale=4, baseline={([yshift=-.5ex]current bounding box.center)}]
\draw[->, draw=magenta] (0.1225, 0.19999999999999998) -- (0.10250000000000001, 0.4666666666666667);
\draw[->, draw=magenta] (0.16249999999999998, 0.19999999999999998) -- (0.4625, 0.4666666666666667);
\draw[->, draw=magenta] (0.3675, 0.19999999999999998) -- (0.30750000000000005, 0.4666666666666667);
\draw[->, draw=magenta] (0.8575, 0.19999999999999998) -- (0.7175, 0.4666666666666667);
\draw[->, draw=magenta] (0.8775, 0.19999999999999998) -- (0.8975000000000001, 0.4666666666666667);
\draw[->, draw=black] (0.1275, 0.4666666666666667) -- (0.34750000000000003, 0.19999999999999998);
\draw[->, draw=magenta] (0.115, 0.5333333333333333) -- (0.23500000000000001, 0.7999999999999999);
\draw[->, draw=black] (0.3325, 0.4666666666666667) -- (0.5925, 0.19999999999999998);
\draw[->, draw=magenta] (0.34500000000000003, 0.5333333333333333) -- (0.7050000000000001, 0.7999999999999999);
\draw[->, draw=black] (0.5375, 0.4666666666666667) -- (0.8375, 0.19999999999999998);
\draw[->, draw=black] (0.6675000000000001, 0.4666666666666667) -- (0.4075, 0.19999999999999998);
\draw[->, draw=magenta] (0.7050000000000001, 0.5333333333333333) -- (0.745, 0.7999999999999999);
\draw[->, draw=black] (0.8725, 0.4666666666666667) -- (0.6525, 0.19999999999999998);
\draw[->, draw=black] (0.275, 0.7999999999999999) -- (0.475, 0.5333333333333333);
\draw[->, draw=black] (0.295, 0.7999999999999999) -- (0.655, 0.5333333333333333);
\draw[->, draw=black] (0.765, 0.7999999999999999) -- (0.885, 0.5333333333333333);
\node[text=black, font={\scriptsize \sffamily}, xshift=0pt, yshift=0pt] at (0.125, 0.16666666666666666) {0};
\node[text=black, font={\scriptsize \sffamily}, xshift=0pt, yshift=0pt] at (0.375, 0.16666666666666666) {1};
\node[text=black, font={\scriptsize \sffamily}, xshift=0pt, yshift=0pt] at (0.625, 0.16666666666666666) {2};
\node[text=black, font={\scriptsize \sffamily}, xshift=0pt, yshift=0pt] at (0.875, 0.16666666666666666) {3};
\node[text=black, font={\scriptsize \sffamily}, xshift=0pt, yshift=0pt] at (0.1, 0.5) {0};
\node[text=black, font={\scriptsize \sffamily}, xshift=0pt, yshift=0pt] at (0.30000000000000004, 0.5) {1};
\node[text=black, font={\scriptsize \sffamily}, xshift=0pt, yshift=0pt] at (0.5, 0.5) {2};
\node[text=black, font={\scriptsize \sffamily}, xshift=0pt, yshift=0pt] at (0.7000000000000001, 0.5) {3};
\node[text=black, font={\scriptsize \sffamily}, xshift=0pt, yshift=0pt] at (0.9, 0.5) {4};
\node[text=black, font={\scriptsize \sffamily}, xshift=0pt, yshift=0pt] at (0.25, 0.8333333333333333) {0};
\node[text=black, font={\scriptsize \sffamily}, xshift=0pt, yshift=0pt] at (0.75, 0.8333333333333333) {1};
\end{tikzpicture}
\]
\end{exm}

\begin{dfn}[Pure subset]
Let $U$ be a closed subset of a graded poset, $n \eqdef \dim{U}$.
We say that $U$ is \emph{pure} if all the maximal elements of $U$ have dimension $n$, that is, $\maxel{U} = \grade{n}{U}$.
\end{dfn}

\begin{lem} \label{lem:round_is_pure}
If $U$ is round, then it is pure.
\end{lem}

\begin{dfn}[Regular molecule]
The class of \emph{regular molecules} is the inductive subclass of oriented graded posets closed under isomorphisms and generated by the following clauses.
\begin{enumerate}
    \item (\textit{Point}). The terminal oriented graded poset $1$ is a regular molecule.
    \item (\textit{Paste}). If $U$, $V$ are regular molecules, $\varphi\colon \bound{k}{+}U \incliso \bound{k}{-}V$ is an isomorphism with $k < \min \set{ \dim U, \dim V }$, then $U \cpiso{k}{\varphi} V$ is a regular molecule.
    \item (\textit{Atom}). If $U$, $V$ are \emph{round} regular molecules of the same dimension and $\varphi\colon \bound{}{}U \incliso \bound{}{}V$ is an isomorphism restricting to $\varphi^\alpha\colon \bound{}{\alpha} U \incliso \bound{}{\alpha} V$ for each $\alpha \in \set{ +, - }$, then $U \celto^\varphi V$ is a regular molecule.
\end{enumerate}
\end{dfn}

\begin{dfn}
We summarise the essential properties of regular molecules.
The second point of Proposition \ref{prop:molecule_properties} allows us to write $U \cp{k} V$ and $U \celto V$ instead of $U \cpiso{k}{\varphi} V$ and $U \celto^\varphi V$ when the latter are defined and $U, V$ are regular molecules.
\end{dfn}

\begin{prop} \label{prop:molecule_properties}
Let $U, V$ be regular molecules, $k \in \mathbb{N}$.
Then
\begin{enumerate}
    \item if $U$ and $V$ are isomorphic, they are isomorphic in a unique way;
    \item if $U \cpiso{k}{\varphi} V$ or $U \celto^\varphi V$ is defined, it is defined for a unique $\varphi$;
    \item for all $n \in \mathbb{N}$, $\alpha \in \set{+, -}$, $\bound{n}{\alpha}U$ is a regular molecule;
    \item $U$ is globular, that is, for all $k, n \in \mathbb{N}$ and $\alpha, \beta \in \set{ +, - }$, if $k < n$ then $\bound{k}{\alpha}(\bound{n}{\beta}U) = \bound{k}{\alpha}U$;
    \item if $U$ is round, for all $n \in \mathbb{N}$, $\alpha \in \set{+, -}$, $\bound{n}{\alpha}U$ is round.
\end{enumerate}
\end{prop}

\begin{exm} \label{exm:generation}
Let
\begin{align*}
    \data{arrow} & \eqdef 1 \celto 1, \\
    \data{binary} & \eqdef (\data{arrow} \cp{0} \data{arrow}) \celto \data{arrow}, \\
    \data{cobinary} & \eqdef \data{arrow} \celto (\data{arrow} \cp{0} \data{arrow}).
\end{align*}
Then the shape of Example \ref{exm:whisker} is a regular molecule constructed as    $\data{binary} \cp{0} \data{arrow}$, while the shape of Example \ref{exm:round} is a regular molecule constructed as 
\begin{equation*}
    (\data{cobinary} \cp{0} \data{arrow}) \cp{1} (\data{arrow} \cp{0} \data{binary}).
\end{equation*}
\end{exm}

\begin{dfn}
The following results imply, together, that pasting of regular molecules satisfies the equations of strict $\omega$\nbd categories up to unique isomorphism.
In particular, the associativity result allows us to write multiple pastings in the same dimension without bracketing.
\end{dfn}

\begin{prop} \label{prop:associativity_of_pasting}
Let $U, V, W$ be regular molecules and let $k \in \mathbb{N}$ such that $U \cp{k} V$ and $V \cp{k} W$ are both defined.
Then $(U \cp{k} V) \cp{k} W$ and $U \cp{k} (V \cp{k} W)$ are both defined and uniquely isomorphic.
\end{prop}

\begin{prop} \label{prop:unitality_of_pasting}
Let $U$ be a regular molecule and $k \in \mathbb{N}$.
Then $U \cp{k} \bound{k}{+}U$ and $\bound{k}{-}U \cp{k} U$ are both defined and uniquely isomorphic to $U$.
\end{prop}

\begin{prop} \label{prop:interchange_of_pasting}
Let $U, U', V, V'$ be regular molecules and $k < n \in \mathbb{N}$ such that $(U \cp{n} U') \cp{k} (V \cp{n} V')$ is defined.
Then $(U \cp{k} V) \cp{n} (U' \cp{k} V')$ is defined and uniquely isomorphic to $(U \cp{n} U') \cp{k} (V \cp{n} V')$.
\end{prop}

\begin{dfn}[Atom]
An \emph{atom} is a regular molecule with a greatest element.
\end{dfn}

\begin{dfn}[Merger of a round regular molecule]
Let $U$ be a round regular molecule of dimension $> 0$.
The \emph{merger} of $U$ is the atom $\compos{U} \eqdef \bound{}{-}U \celto \bound{}{+}U$.
\end{dfn}

\begin{exm} \label{exm:merger}
The following pair of pasting diagrams depicts a round regular molecule $U$ and its merger $\compos{U}$, an atom.
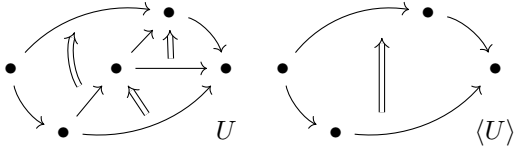
\[\begin{tikzcd}[column sep=tiny,row sep=small]
	&&& \bullet &&&& {} & \bullet \\
	\bullet && \bullet && \bullet & \bullet &&&& \bullet \\
	& \bullet &&& U && \bullet & {} && {\compos{U}}
	\arrow[curve={height=6pt}, from=2-1, to=3-2]
	\arrow[""{name=0, anchor=center, inner sep=0}, from=3-2, to=2-3]
	\arrow[""{name=1, anchor=center, inner sep=0}, from=2-3, to=2-5]
	\arrow[""{name=2, anchor=center, inner sep=0}, curve={height=12pt}, from=3-2, to=2-5]
	\arrow[""{name=3, anchor=center, inner sep=0}, curve={height=-12pt}, from=2-1, to=1-4]
	\arrow[curve={height=-6pt}, from=1-4, to=2-5]
	\arrow[from=2-3, to=1-4]
	\arrow[curve={height=-12pt}, from=2-6, to=1-9]
	\arrow[curve={height=-6pt}, from=1-9, to=2-10]
	\arrow[curve={height=6pt}, from=2-6, to=3-7]
	\arrow[curve={height=12pt}, from=3-7, to=2-10]
	\arrow[shorten <=6pt, shorten >=3pt, Rightarrow, from=3-8, to=1-8]
	\arrow[shorten <=4pt, Rightarrow, from=2, to=2-3]
	\arrow[shorten <=3pt, Rightarrow, from=1, to=1-4]
	\arrow[curve={height=-6pt}, shorten <=6pt, shorten >=6pt, Rightarrow, from=0, to=3]
\end{tikzcd}\]
\end{exm}

\begin{prop} \label{prop:atom_properties}
Let $U$ be a regular molecule.
Then
\begin{enumerate}
    \item $U$ is an atom if and only if it was produced by (\textit{Point}) or (\textit{Atom});
    \item for all $x \in U$, $\clset{x}$ is an atom;
    \item if $U$ is an atom, then $U$ is round, and if $\dim{U} > 0$ then $U$ is isomorphic to $\compos{U}$.
\end{enumerate}
\end{prop}

\begin{dfn}
For us, a \emph{cell} is a diagram whose shape is an atom.
\end{dfn}

\begin{dfn}[Diagram isomorphism problem]
The \emph{diagram isomorphism problem} is the following decision problem: given diagrams $t\colon U \to \mathbb{V}$ and $t'\colon U' \to \mathbb{V}$, does there exist an isomorphism $\varphi\colon U \incliso U'$ of their shapes such that $t = \varphi;t'$?

By Proposition \ref{prop:molecule_properties}, if $\varphi$ exists it is unique, and if $\varphi$ is found the labellings $t$ and $\varphi; t'$ can be compared in linear time, so this problem reduces to the isomorphism problem for regular molecules.
One of the main results of \cite{hadzihasanovic2022data} is a polynomial-time solution to this problem, relying on a deterministic \emph{traversal algorithm} for regular molecules:
\begin{enumerate}
    \item all elements of a regular molecule can be traversed in polynomial time in such a way that the traversal order is invariant under isomorphism;
    \item consequently, $U$ and $U'$ are isomorphic if and only if their representations using the traversal order dimension-wise are identical.
\end{enumerate}
The traversal order also gives a \emph{canonical representation} of regular molecules.
The constructors for regular molecules can be implemented in such a way that the elements are rearranged in traversal order after each step, so that regular molecules can be checked for equality rather than isomorphism.

The complexity upper bound given in \cite[Theorem 2.19]{hadzihasanovic2022data} is in fact the result of overcounting.
We describe the traversal algorithm, whose correctness is proved in \cite[Theorem 2.17]{hadzihasanovic2022data}, and give an improved upper bound.
\end{dfn}

\begin{dfn}[Traversal algorithm]
The procedure takes as input a regular molecule $U$ and returns a list of its elements in the order in which they are \emph{marked}.
It uses an auxiliary stack of regular molecules $V \subseteq U$.

At the beginning, only $U$ is on the stack and all elements are unmarked.
We iterate the \emph{main loop} until the stack is empty, at which point the procedure terminates.

At each iteration, suppose $V$ is on top of the stack.
If all elements of $V$ are marked, then we pop $V$ from the stack and iterate.
Else, if any elements of $\bound{}{-}V$ are unmarked, we push $\bound{}{-}V$ to the top of the stack and iterate.
Else, if $V = \clset{x}$ for some $x \in U$, we 
\begin{enumerate}
    \item mark $x$ and pop $V$ from the stack, 
    \item if any elements of $\bound{}{+}V$ are unmarked, we push $\bound{}{+}V$ to the top of the stack, and
    \item we iterate.
\end{enumerate}
Else, we let $y$ be the earliest marked element such that $\dim{y} = \dim{V} - 1$ and there is an unmarked $x \in \cofaces{}{-}y \cap V$.
Such a $y$ always exists, and then $\cofaces{}{-}y \cap V = \set{x}$.
We push $\clset{x}$ to the top of the stack and iterate.
\end{dfn}

\begin{dfn} \label{dfn:parameters}
For a regular molecule $U$, and all $k \in \mathbb{N}$, we let
\begin{align*}
    \edges{k}{U} & \eqdef \coprod_{\mathclap{x \in \grade{k}{U}}} \faces{}{}x = \coprod_{\mathclap{y \in \grade{k-1}{U}}} \cofaces{}{}y, \\
    \size{\grade{\lor}{U}} & \eqdef \max \set{\size{\grade{i}{U}}}_{i\in\mathbb{N}},
    \\
    \size{\edges{\lor}{U}} & \eqdef \max (\set{\size{\edges{i}{U}}}_{i\in\mathbb{N}} \cup \set{1}).
\end{align*}
Note that $\edges{k}{U}$ is the set of edges between $k$ and $(k-1)$\nbd dimensional elements in $\hasse{U}$.
We have $\size{\grade{k}{U}} \leq \size{\edges{k}{U}}$ for all $k > 0$, while $\size{\edges{0}{U}} = 0$.
Since the maximum of the $\size{\edges{k}{U}}$ is 0 only when $U$ is 0\nbd dimensional, in which case $\size{\grade{k}{U}} = 1$, with our definition we always have $\size{\grade{\lor}{U}} \leq \size{\edges{\lor}{U}}$.
\end{dfn}

\begin{thm} \label{thm:traversal_new_bound}
The traversal algorithm admits an implementation running in time $O(\size{U}\,\size{\edges{\lor}{U}}\,\log \size{\edges{\lor}{U}})$.
\end{thm}
\begin{proof}
The upper bound of $O(\size{\edges{\lor}{U}}\,\log \size{\edges{\lor}{U}})$ on each loop iteration is derived as in \cite[Theorem 2.19]{hadzihasanovic2022data}, simplified by our modified definition of $\size{\edges{\lor}{U}}$, so we only need to show that there are $O(\size{U})$ loop iterations.

We let $k \leq \dim{U}$ and we count the number of loop iterations where a $k$\nbd dimensional subset $V$ is on top of the stack.
This can happen in two ways:
\begin{itemize}
    \item $V$ is either $U$ or $\bound{}{\alpha}W$ for some $W$ with $\dim{W} > k$, where $W$ was earlier (and may still be) on the stack,
    \item $V$ is $\clset{x}$ for some $x \in W$, where $\dim{W} = k$ and $W$ is below $V$ on the stack.
\end{itemize}
Let $(\order{i}{V})_{i=1}^m$ be the sequence of all $k$\nbd dimensional subsets appearing on the stack \emph{in the first way} during the run, in the order in which they appear.
For all $j < i \in \set{1, \ldots, m}$, by \cite[Lemma 2.14]{hadzihasanovic2022data} $\order{j}{V}$ must be fully marked before $\order{i}{V}$ can appear on the stack. Moreover, $\order{i}{V}$ can be on top at most
\begin{enumerate}
    \item once to push $\bound{}{-}\order{i}{V}$ to the top,
    \item once every time we push $\clset{x}$ to the top for an unmarked $x \in \grade{k}{(\order{i}{V})}$,
    \item once to pop $\order{i}{V}$ from the stack.
\end{enumerate}
Any $k$\nbd dimensional $\clset{x}$ appearing \emph{in the second way} appears while a unique $\order{i}{V}$ is on the stack, and at most
\begin{enumerate}
    \item once to push $\bound{}{-}x$ to the top,
    \item once to mark $x$ and pop $\clset{x}$ from the stack.
\end{enumerate}
Let $\order{i}{\grade{k}{U}} \eqdef \grade{k}{(\order{i}{V})} \setminus \bigcup_{j < i} \grade{k}{(\order{j}{V})}$.
Then $\order{i}{\grade{k}{U}}$ is precisely the set of unmarked $k$\nbd dimensional elements of $\order{i}{V}$ when $\order{i}{V}$ first appears on the stack.
It follows that the number of loop iterations with a $k$\nbd dimensional subset on top of the stack \emph{while $\order{i}{V}$ is on the stack} is at most $2 + 3\size{\order{i}{\grade{k}{U}}}$.

Since at the end of the procedure all $k$\nbd dimensional elements of $U$ are marked, the $(\order{i}{\grade{k}{U}})_{i=1}^m$ form a partition of $\grade{k}{U}$.
Thus, the total number of loop iterations where a $k$\nbd dimensional subset is on top of the stack is bounded above by
\begin{equation*}
    \sum_{i=1}^m \left(2 + 3\size{\order{i}{\grade{k}{U}}}\right) = 2m + 3\sum_{i=1}^m \size{\order{i}{\grade{k}{U}}} = 2m + 3\size{\grade{k}{U}},
\end{equation*}
which is bounded above by $5\size{\grade{k}{U}}$.
Summing over all dimensions, we get an upper bound of $5\size{U}$ iterations.
\end{proof}
\section{The subdiagram matching problem} \label{sec:matching}

\begin{dfn}[Submolecule inclusion]
The class of \emph{submolecule inclusions} is the smallest subclass of inclusions of regular molecules such that
\begin{enumerate}
    \item all isomorphisms are submolecule inclusions,
    \item for all regular molecules $U, V$ and all $k \in \mathbb{N}$ such that $U \cp{k} V$ is defined, $U \incl (U \cp{k} V)$ and $V \incl (U \cp{k} V)$ are submolecule inclusions,
    \item the composite of two submolecule inclusions is a submolecule inclusion.
\end{enumerate}
A closed subset $V \subseteq U$ is a \emph{submolecule} if its inclusion in $U$ is a submolecule inclusion.
In that case we write $V \submol U$.
\end{dfn}

\begin{dfn}
We also let $\varnothing \submol \varnothing$ to take care of some corner cases.
\end{dfn}

\begin{lem} \label{lem:submolecule_properties}
Let $U$ be a regular molecule.
Then
\begin{enumerate}
    \item for all $n \in \mathbb{N}$ and $\alpha \in \set{+,-}$, $\bound{n}{\alpha}U \submol U$;
    \item for all $x \in U$, $\clset{x} \submol U$.
\end{enumerate}
\end{lem}

\begin{dfn}[Substitution]
Let $U, V, W$ be regular molecules with $\dim U = \dim V = \dim W$, $\imath\colon V \incl U$ an inclusion, and suppose that $V \celto W$ is defined.
Consider the pushout
\begin{equation} \label{eq:rewrite_application}
\begin{tikzcd}
	V & V \celto W \\
	U & U \cup (V \celto W)
	\arrow[hook, from=1-1, to=1-2]
	\arrow["\imath", hook', from=1-1, to=2-1]
	\arrow[hook, from=2-1, to=2-2]
	\arrow[hook', from=1-2, to=2-2]
	\arrow["\lrcorner"{anchor=center, pos=0.125, rotate=180}, draw=none, from=2-2, to=1-1]
\end{tikzcd}
\end{equation}
in $\ogpos$.
The \emph{substitution of $W$ for $\imath\colon V \incl U$} is the oriented graded poset $\subs{U}{W}{\imath(V)} \eqdef \bound{}{+} (U \cup (V \celto W))$.
When $\imath$ is the inclusion of a closed subset we write simply $\subs{U}{W}{V}$.
\end{dfn}

\begin{prop} \label{prop:round_submolecule_substitution}
Let $\imath\colon V \incl U$ be an inclusion of regular molecules such that $\dim{V} = \dim{U}$ and $V$ is round.
The following are equivalent:
\begin{enumerate}[label=(\alph*)]
    \item $\imath$ is a submolecule inclusion;
    \item for all regular molecules $W$ such that $V \celto W$ is defined, $U \cup (V \celto W)$ in (\ref{eq:rewrite_application}) is a regular molecule;
    \item for all regular molecules $W$ such that $V \celto W$ is defined, $\subs{U}{W}{\imath(V)}$ is a regular molecule;
    \item $\subs{U}{\compos{V}}{\imath(V)}$ is a regular molecule.
\end{enumerate}
\end{prop}

\begin{dfn}[Rewritable submolecule]
A submolecule $V \submol U$ is \emph{rewritable} if $\dim{V} = \dim{U}$ and $V$ is round.
\end{dfn}

\begin{dfn}[Rewritable subdiagram]
Let $t\colon U \to \mathbb{V}$ be a diagram.
A \emph{rewritable subdiagram} of $t$ is the restriction of $t$ to a rewritable submolecule $V \submol U$.
\end{dfn}

\begin{dfn}
We extend boundary operations to diagrams $t\colon U \to \mathbb{V}$ by $\bound{n}{\alpha}t \eqdef \restr{t}{\bound{n}{\alpha}U}$ for all $n \in \mathbb{N}$ and $\alpha \in \set{+, -}$.
\end{dfn}

\begin{dfn}
Suppose $t$ is an $n$\nbd dimensional diagram of shape $U$ and $r$ a rewrite rule, in the form of an $(n+1)$\nbd dimensional cell of shape $V \celto W$.
If there is an inclusion $\imath\colon V \incl U$ such that $\bound{}{-}r = \imath;t$, then the application of the rewrite $r$ to $t$ is modelled by an $(n+1)$\nbd dimensional diagram $t \cup r$ of shape $U \cup (V \celto W)$ as in (\ref{eq:rewrite_application}).

By Proposition \ref{prop:round_submolecule_substitution}, $\bound{}{+}(t \cup r)$, which models the substitution of $\bound{}{+}r$ for $\bound{}{-}r$ in $t$, is guaranteed to be an $n$\nbd dimensional diagram precisely when $\imath$ is a submolecule inclusion, that is, $\bound{}{-}r$ is isomorphic to a rewritable subdiagram of $t$.
\end{dfn}

\begin{exm}
The following is a depiction of diagram (\ref{eq:rewrite_application}) when $V \celto W \eqdef \data{cobinary}$, $U \eqdef \data{arrow} \cp{0} \data{arrow}$, and $\imath$ is the inclusion of the second $\data{arrow}$ into the pasting.
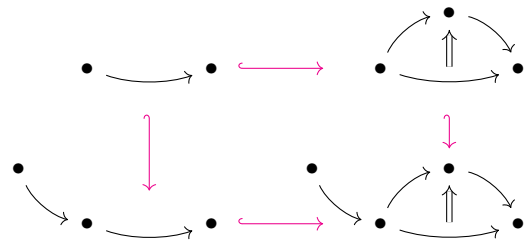
\[\begin{tikzcd}[sep=small]
	&&&&&&& \bullet \\
	& \bullet & {} & \bullet &&& \bullet & {} & \bullet \\
	&&&&&&&&& {} \\
	\bullet &&&&& \bullet && \bullet \\
	& \bullet & {} & \bullet &&& \bullet & {} & \bullet
	\arrow[curve={height=6pt}, from=5-2, to=5-4]
	\arrow[curve={height=6pt}, from=4-1, to=5-2]
	\arrow[curve={height=6pt}, from=2-2, to=2-4]
	\arrow[""{name=0, anchor=center, inner sep=0}, curve={height=6pt}, from=2-7, to=2-9]
	\arrow[curve={height=-6pt}, from=2-7, to=1-8]
	\arrow[curve={height=-6pt}, from=1-8, to=2-9]
	\arrow[curve={height=6pt}, from=4-6, to=5-7]
	\arrow[""{name=1, anchor=center, inner sep=0}, curve={height=6pt}, from=5-7, to=5-9]
	\arrow[curve={height=-6pt}, from=5-7, to=4-8]
	\arrow[curve={height=-6pt}, from=4-8, to=5-9]
	\arrow[color=magenta, shorten <=11pt, shorten >=11pt, hook', from=2-3, to=5-3]
	\arrow[color=magenta, shorten <=29pt, shorten >=43pt, hook, from=2-3, to=2-8]
	\arrow[color=magenta, shorten <=11pt, shorten >=27pt, hook', from=2-8, to=5-8]
	\arrow[color=magenta, shorten <=29pt, shorten >=43pt, hook, from=5-3, to=5-8]
	\arrow[shorten <=5pt, Rightarrow, from=0, to=1-8]
	\arrow[shorten <=5pt, Rightarrow, from=1, to=4-8]
\end{tikzcd}\]
The result of the substitution is the output boundary of the bottom right diagram, isomorphic to $\data{arrow} \cp{0} \data{arrow} \cp{0} \data{arrow}$.
\end{exm}

\begin{dfn}[Subdiagram matching problem]
The \emph{subdiagram matching problem} is the following search problem: given diagrams $t\colon U \to \mathbb{V}$ and $s\colon V \to \mathbb{V}$ such that $\dim{U} = \dim{V}$ and $V$ is round, find, if any, the submolecule inclusions $\imath\colon V \incl U$ such that $s = \imath;t$.
This can be split into three subproblems.
\begin{enumerate}
    \item (\emph{Molecule matching problem}). Find, if any, the inclusions $\imath\colon V \incl U$.
    \item (\emph{Rewritable submolecule problem}). Decide if $\imath(V) \submol U$.
    \item Decide if $s = \imath;t$.
\end{enumerate}
In this section, we will focus on the molecule matching problem, and in the next on the rewritable submolecule problem.
The third problem is trivial.
\end{dfn}

\begin{lem} \label{lem:codimension_1_elements}
Let $U$ be a regular molecule, $n \eqdef \dim{U}$, $x \in \grade{n-1}{U}$, and $\alpha \in \set{+, -}$.
Then
\begin{enumerate}
    \item $x \in \maxel{U}$ if and only if $\size{\cofaces{}{}{x}} = 0$,
    \item $x \in \faces{}{\alpha}U \setminus \faces{}{-\alpha}{U}$ if and only if $\size{\cofaces{}{\alpha}x} = 1$ and $\size{\cofaces{}{-\alpha}x} = 0$,
    \item $x \notin \faces{}{}U$ if and only if $\size{\cofaces{}{+}x} = \size{\cofaces{}{-}x} = 1$.
\end{enumerate}
\end{lem}

\begin{dfn}[Flow graph]
Let $U$ be a regular molecule, $k \geq -1$.
The \emph{$k$\nbd flow graph of $U$} is the directed graph $\flow{k}{U}$ whose
\begin{itemize}
    \item set of vertices is $\bigcup_{i > k} \grade{i}{U}$, and
    \item for all vertices $x, y$, there is an edge from $x$ to $y$ if and only if $\faces{k}{+}x \cap \faces{k}{-}y$ is non-empty.
\end{itemize}
\end{dfn}

\begin{dfn}[Induced subgraph]
Let $\mathscr{G}$ be a directed graph and let $W$ be a subset of its vertex set.
The \emph{induced subgraph} of $\mathscr{G}$ on $W$ is the directed graph $\restr{\mathscr{G}}{W}$ whose vertex set is $W$, and there is an edge from $x$ to $y$ for every edge from $x$ to $y$ in $\mathscr{G}$.
\end{dfn}

\begin{dfn}[Maximal flow graph] \index{regular molecule!flow graph!maximal}
Let $U$ be a regular molecule, $k \geq -1$.
The \emph{maximal $k$\nbd flow graph of $U$} is the induced subgraph $\maxflow{k}{U}$ of $\flow{k}{U}$ on the vertex set
\begin{equation*}
    \bigcup_{i > k} \grade{i}{(\maxel{U})} \subseteq \bigcup_{i > k} \grade{i}{U}.
\end{equation*}
Note that, if $k = \dim{U} - 1$, then $\flow{k}{U} = \maxflow{k}{U}$.
\end{dfn}

\begin{exm}
If $U$ is the regular molecule of Example \ref{exm:whisker},
\[\flow{0}{U}:
\begin{tikzcd}[sep=tiny]
	{(1,3)} \\
	{(2,0)} && {(1,2)} \\
	{(1,0)} & {(1,1)}
	\arrow[from=2-1, to=2-3]
	\arrow[curve={height=-6pt}, from=1-1, to=2-3]
	\arrow[curve={height=6pt}, from=3-2, to=2-3]
	\arrow[from=3-1, to=3-2]
\end{tikzcd}
\quad \maxflow{0}{U}:
\begin{tikzcd}[sep=tiny]
	{(2,0)} & {(1,2)}
	\arrow[from=1-1, to=1-2]
\end{tikzcd}
\]
\[
\flow{1}{U} = \maxflow{1}{U}: 
\begin{tikzcd}[sep=tiny]
	{(2,0)}
\end{tikzcd}
\]
If $U$ is the round regular molecule of Example \ref{exm:round},
\[
\flow{1}{U} = \maxflow{1}{U}:
\begin{tikzcd}[sep=tiny]
	{(2,0)} & {(2,1)}
	\arrow[from=1-1, to=1-2]
\end{tikzcd}
\]
\end{exm}

\begin{lem} \label{lem:flow_under_inclusion}
Let $\imath\colon V \incl U$ be an inclusion of regular molecules, $k \geq -1$.
Then $\flow{k}{V}$ is isomorphic to the induced subgraph of $\flow{k}{U}$ on the vertices in the image of $\imath$.
\end{lem}

\begin{prop} \label{prop:round_molecule_connected_flowgraph}
Let $U$ be a regular molecule, $n \eqdef \dim{U}$.
If $U$ is round, then $\flow{n-1}{U}$ is connected.
\end{prop}

\begin{dfn}
We describe an algorithm for the molecule matching problem.
The main idea is the following: if we succeed in matching only \emph{one} top-dimensional atom in $V$ with one top-dimensional atom in $U$, then there is only \emph{one possible matching} of all other top-dimensional atoms of $V$ to atoms in $U$.
This is because, by Proposition \ref{prop:round_molecule_connected_flowgraph}, we can try to match top-dimensional elements of $V$ in such an order that the next element to match is connected by an edge in $\flow{n-1}{V}$ to a previously matched element; that is, it shares a face $z$ with a previously matched element.
In particular, $z$ has already been matched.
By Lemma \ref{lem:codimension_1_elements}, in order for us to continue, the match of $z$ must have exactly two cofaces in $U$, one of which is the previously matched top-dimensional element.
Necessarily, the other coface is the next match.
\end{dfn}

\begin{dfn}[Molecule matching algorithm]
The procedure takes as input two regular molecules $U, V$ such that $\dim{U} = \dim{V}$ and $V$ is round, and it returns all inclusions $V \incl U$.

Let $n \eqdef \dim{U}$.
To begin, we pick an arbitrary ordering $(\order{i}{x})_{i=1}^m$, for example the traversal order, of the elements of $\grade{n}{U}$.
Moreover, we pick an ordering $(\order{j}{y})_{j=1}^p$ of the elements of $\grade{n}{V}$ with the property that, for all $k \in \set{1, \ldots, p}$, the induced subgraph of $\flow{n-1}{V}$ on $(\order{j}{y})_{j=1}^k$ is connected.
This is possible because $\flow{n-1}{V}$ is connected by Proposition \ref{prop:round_molecule_connected_flowgraph}.
For each $k \in \set{1, \ldots, p}$, we let $\order{k}{V} \eqdef \bigcup_{j \leq k} \clset{{\order{j}{y}}}$.
We have $\order{i}{V} \subseteq \order{j}{V}$ whenever $i \leq j$, and $\order{p}{V} = V$ since $V$ is pure by Lemma \ref{lem:round_is_pure}.

For each $i \in \set{1, \ldots, m}$, we attempt to construct a sequence of inclusions $(\order{i, j}{\imath}\colon \order{j}{V} \incl U)_{j=1}^p$ such that the restriction of $\order{i, j'}{\imath}$ to $\order{j}{V}$ is equal to $\order{i, j}{\imath}$ when $j \leq j'$, iterating on $k \in \set{1, \ldots, p}$.
When $k = 1$, if $\order{1}{V} = \clset{{\order{1}{y}}}$ is isomorphic to $\clset{{\order{i}{x}}}$, we let $\order{i, 1}{\imath}$ be the unique isomorphism $\order{1}{V} \incliso \clset{{\order{i}{x}}}$ followed by the inclusion $\clset{{\order{i}{x}}} \subseteq U$, and iterate on $k$.
Else, we iterate on $i$.

When $k > 1$, we let $j$ be the least value such that there exists an edge between $\order{j}{y}$ and $\order{k}{y}$ in $\flow{n-1}{V}$.
Then $j < k$ because of our connectedness assumption, and there exists $z \in \faces{}{\alpha}\order{j}{y} \cap \faces{}{-\alpha}\order{k}{y}$ for some $\alpha \in \set{+, -}$.
We pick the least such $z$ with respect to some ordering of $\grade{n-1}{V}$, for example the traversal order.
By Lemma \ref{lem:properties_of_inclusions}, $\order{i, k-1}{\imath}(\order{j}{y})$ is one coface of $\order{i, k-1}{\imath}(z)$ in $U$.
If $\order{i, k-1}{\imath}(z)$ has no other cofaces, then we iterate on $i$.
Else, by Lemma \ref{lem:codimension_1_elements}, $\order{i, k-1}{\imath}(z)$ has exactly one other coface, call it $x$; note that $x$ cannot be in the image of $\order{i, k-1}{\imath}$, since $\order{j}{y}$ and $\order{k}{y}$ are the only cofaces of $z$ in $V$.
If $\clset{{\order{k}{y}}}$ is isomorphic to $\clset{{x}}$, and the unique isomorphism $\clset{{\order{k}{y}}} \incliso \clset{{x}}$ followed by the inclusion $\clset{{x}} \subseteq U$ matches $\order{i, k-1}{\imath}$ on $\clset{{\order{k}{y}}} \cap \order{k-1}{V}$, then we let $\order{i, k}{\imath}$ be the unique extension of $\order{i, k-1}{\imath}$ that restricts to $\clset{{\order{k}{y}}} \incliso \clset{{x}} \subseteq U$.
Else, we iterate on $i$.

If we succeed to construct $\order{i, p}{\imath}$, we add it to the list of inclusions $V \incl U$, then iterate on $i$.
\end{dfn}

\begin{thm} \label{thm:molecule_matching_bound}
The molecule matching problem in dimension $n$ can be solved in time
\begin{equation*}
    O(\size{\grade{n}{U}}\, \size{\grade{n}{V}} \, \size{V} \, \size{\edges{\lor}{V}}\,\log \size{\edges{\lor}{V}}).
\end{equation*}
\end{thm}
\begin{proof}
We suppose $n > 0$ since the case $n = 0$ is trivial.
First of all, with our choice of data structures both $\grade{n}{U}$ and $\grade{n-1}{V}$ already come with a linear order when one is needed.
Moreover, we can both construct $\flow{n-1}{V}$ and order its vertices in the desired way by traversing the ``slice'' of $\hasse{V}$ on the elements of dimension $n$ and $(n-1)$.
Since $\max \set{\size{\grade{n}{V}}, \size{\grade{n-1}{V}}} \leq \size{\edges{n}{V}}$, this can be done in time $O(\size{\edges{n}{V}})$ with a standard traversal algorithm.

In the main part of the algorithm, we have exactly $\size{\grade{n}{U}}$ iterations.
At each iteration, we need to solve at most $\size{\grade{n}{V}}$ isomorphism problems for submolecules of $V$.
The time complexity of each can be bounded above by the time complexity of the isomorphism problem for $V$, which is $O(\size{V} \, \size{\edges{\lor}{V}}\,\log \size{\edges{\lor}{V}})$ by Theorem \ref{thm:traversal_new_bound}.
It is straightforward to verify that all other operations, such as checking that the isomorphisms match on intersections or finding the next match, have lower complexity.
Since $\size{\edges{n}{V}} \leq \size{\edges{\lor}{V}}$, we can ignore the $O(\size{\edges{n}{V}})$ summand, and conclude.
\end{proof}

\section{The rewritable submolecule problem} \label{sec:rewritable}

\begin{dfn}
Our solution to the rewritable submolecule problem requires us to develop new results about \emph{layerings} of diagrams, and their associated \emph{orderings}.
\end{dfn}

\begin{dfn}[Layering of a regular molecule]
Let $U$ be a regular molecule, $-1 \leq k < \dim{U}$, and
\begin{equation*}
    m \eqdef \size{\bigcup_{i > k} \grade{i}{(\maxel{U})}}.
\end{equation*}
A \emph{$k$\nbd layering of $U$} is a sequence $(\order{i}{U})_{i=1}^m$ of regular molecules such that $U$ is isomorphic to $\order{1}{U} \cp{k} \ldots \cp{k} \order{m}{U}$ and $\dim{\order{i}{U}} > k$ for all $i \in \set{ 1, \ldots, m }$.

For $k = -1$, it is implied that $m = 1$, and $U$ is an atom.
We will regularly identify the regular molecules in a layering of $U$ with their isomorphic images in $U$, which are submolecules.
\end{dfn}

\begin{exm} \label{exm:layerings}
The shape of the pasting diagram
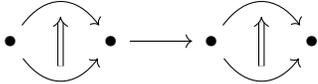
\[\begin{tikzcd}[sep=small]
	\bullet && \bullet && \bullet && \bullet
	\arrow[""{name=0, anchor=center, inner sep=0}, curve={height=18pt}, from=1-1, to=1-3]
	\arrow[""{name=1, anchor=center, inner sep=0}, curve={height=-18pt}, from=1-1, to=1-3]
	\arrow[from=1-3, to=1-5]
	\arrow[""{name=2, anchor=center, inner sep=0}, curve={height=18pt}, from=1-5, to=1-7]
	\arrow[""{name=3, anchor=center, inner sep=0}, curve={height=-18pt}, from=1-5, to=1-7]
	\arrow[shorten <=5pt, shorten >=5pt, Rightarrow, from=0, to=1]
	\arrow[shorten <=5pt, shorten >=5pt, Rightarrow, from=2, to=3]
\end{tikzcd}\]
admits no $(-1)$\nbd layerings, a single $0$\nbd layering
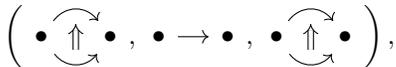
\begin{equation*}
    \left(
\begin{tikzcd}[sep=tiny]
	\bullet && \bullet
	\arrow[""{name=0, anchor=center, inner sep=0}, curve={height=12pt}, from=1-1, to=1-3]
	\arrow[""{name=1, anchor=center, inner sep=0}, curve={height=-12pt}, from=1-1, to=1-3]
	\arrow[shorten <=5pt, shorten >=5pt, Rightarrow, from=0, to=1]
\end{tikzcd},
    \begin{tikzcd}[sep=tiny]
	\bullet && \bullet
	\arrow[from=1-1, to=1-3]
\end{tikzcd},
\begin{tikzcd}[sep=tiny]
	\bullet && \bullet
	\arrow[""{name=0, anchor=center, inner sep=0}, curve={height=12pt}, from=1-1, to=1-3]
	\arrow[""{name=1, anchor=center, inner sep=0}, curve={height=-12pt}, from=1-1, to=1-3]
	\arrow[shorten <=5pt, shorten >=5pt, Rightarrow, from=0, to=1]
\end{tikzcd}
    \right),
\end{equation*}
and two $1$\nbd layerings:
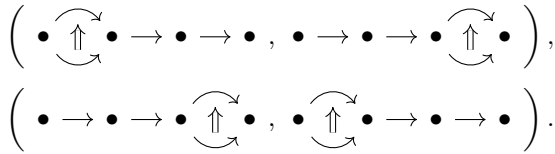
\begin{equation*}
    \left(
\begin{tikzcd}[sep=tiny]
	\bullet && \bullet && \bullet && \bullet
	\arrow[""{name=0, anchor=center, inner sep=0}, curve={height=12pt}, from=1-1, to=1-3]
	\arrow[""{name=1, anchor=center, inner sep=0}, curve={height=-12pt}, from=1-1, to=1-3]
	\arrow[from=1-3, to=1-5]
	\arrow[from=1-5, to=1-7]
	\arrow[shorten <=5pt, shorten >=5pt, Rightarrow, from=0, to=1]
\end{tikzcd},
\begin{tikzcd}[sep=tiny]
	\bullet && \bullet && \bullet && \bullet
	\arrow[from=1-3, to=1-5]
	\arrow[""{name=0, anchor=center, inner sep=0}, curve={height=12pt}, from=1-5, to=1-7]
	\arrow[""{name=1, anchor=center, inner sep=0}, curve={height=-12pt}, from=1-5, to=1-7]
	\arrow[from=1-1, to=1-3]
	\arrow[shorten <=5pt, shorten >=5pt, Rightarrow, from=0, to=1]
\end{tikzcd}
\right),
\end{equation*}
\begin{equation*}
\left(
\begin{tikzcd}[sep=tiny]
	\bullet && \bullet && \bullet && \bullet
	\arrow[from=1-3, to=1-5]
	\arrow[""{name=0, anchor=center, inner sep=0}, curve={height=12pt}, from=1-5, to=1-7]
	\arrow[""{name=1, anchor=center, inner sep=0}, curve={height=-12pt}, from=1-5, to=1-7]
	\arrow[from=1-1, to=1-3]
	\arrow[shorten <=5pt, shorten >=5pt, Rightarrow, from=0, to=1]
\end{tikzcd},
\begin{tikzcd}[sep=tiny]
	\bullet && \bullet && \bullet && \bullet
	\arrow[""{name=0, anchor=center, inner sep=0}, curve={height=12pt}, from=1-1, to=1-3]
	\arrow[""{name=1, anchor=center, inner sep=0}, curve={height=-12pt}, from=1-1, to=1-3]
	\arrow[from=1-3, to=1-5]
	\arrow[from=1-5, to=1-7]
	\arrow[shorten <=5pt, shorten >=5pt, Rightarrow, from=0, to=1]
\end{tikzcd}
\right).
\end{equation*}
\end{exm}

\begin{lem} \label{lem:layering_basic_properties}
Let $U$ be a regular molecule, $k < \dim{U}$, and suppose $U$ admits a $k$\nbd layering $(\order{i}{U})_{i=1}^m$.
Then
\begin{enumerate}
    \item for all $i \in \set{ 1, \ldots, m }$, $\order{i}{U}$ contains a single maximal element of dimension $> k$,
    \item for all $k \leq \ell < \dim{U}$, $U$ admits an $\ell$\nbd layering.
\end{enumerate}
\end{lem}

\begin{dfn}
A regular molecule does not, in general, admit a $k$\nbd layering for each $k < \dim{U}$; however, by Lemma \ref{lem:layering_basic_properties}, when it does admit a $k$\nbd layering, it also admits a layering in dimensions higher than $k$.
The next result shows that every regular molecule does admit a $k$\nbd layering for some $k$, and that the smallest such $k$ falls into a particular range.
\end{dfn}

\begin{dfn}[Layering dimension]
Let $U$ be a regular molecule.
The \emph{layering dimension} of $U$ is the integer
\begin{equation*}
    \lydim{U} \eqdef \min \set{ k \geq -1 \mid \size{\bigcup_{i > k+1} \grade{i}{(\maxel{U})}} \leq 1 }.
\end{equation*}
\end{dfn}

\begin{dfn}[Frame dimension]
Let $U$ be a regular molecule.
The \emph{frame dimension of $U$} is the integer
\begin{equation*}
    \frdim{U} \eqdef \dim{\bigcup \set{\clset{{x}} \cap \clset{{y}} \mid
    x, y \in \maxel{U}, x \neq y } }.
\end{equation*}
\end{dfn}

\begin{thm} \label{thm:molecules_admit_layerings}
Let $U$ be a regular molecule.
Then there exists $k < \dim{U}$ such that $U$ admits a $k$\nbd layering.
Moreover,
\begin{equation*}
    \frdim{U} \leq \min\set{k \mid \text{$U$ admits a $k$\nbd layering}} \leq \lydim{U}.
\end{equation*}
\end{thm}

\begin{cor} \label{cor:codimension_1_layering}
Let $U$ be a regular molecule, $n \eqdef \dim{U}$.
Then $U$ admits an $(n-1)$\nbd layering.
\end{cor}

\begin{dfn}[Ordering of a regular molecule]
Let $U$ be a regular molecule, $k \geq -1$, and suppose $\maxflow{k}{U}$ is acyclic.
A \emph{$k$\nbd ordering of $U$} is a topological sort of $\maxflow{k}{U}$.
\end{dfn}

\begin{prop} \label{prop:if_layering_then_ordering}
Let $U$ be a regular molecule, $k \geq -1$.
If $U$ admits a $k$\nbd layering, then $\maxflow{k}{U}$ is acyclic, hence $U$ admits a $k$\nbd ordering.
\end{prop}

\begin{cor} \label{cor:flow_acyclic_in_codimension_1}
Let $U$ be a regular molecule, $n \eqdef \dim{U}$.
Then $\flow{n-1}{U}$ is acyclic.
\end{cor}
\begin{proof}
Follows from Corollary \ref{cor:codimension_1_layering} and Proposition \ref{prop:if_layering_then_ordering} combined with the fact that $\flow{n-1}{U} = \maxflow{n-1}{U}$.
\end{proof}

\begin{dfn}
Let $U$ be a regular molecule, $k \geq -1$.
We let
\begin{align*}
    \layerings{k}{U} & \eqdef \set{\text{$k$\nbd layerings $(\order{i}{U})_{i=1}^m$ of $U$}}, \\
    \orderings{k}{U} & \eqdef \set{\text{$k$\nbd orderings $(\order{i}{x})_{i=1}^m$ of $U$}},
\end{align*}
where layerings are considered up to layer-wise isomorphism.
\end{dfn}

\begin{prop} \label{prop:layerings_induce_orderings}
Let $U$ be a regular molecule, $k \geq -1$.
For each $k$\nbd layering $(\order{i}{U})_{i=1}^m$ of $U$ and each $i \in \set{1, \ldots, m}$, let $\order{i}{x}$ be the only element of $\bigcup_{j > k}\grade{j}{(\maxel{U})}$ in the image of $\order{i}{U}$.
Then the assignment
\begin{equation}
    \lto{k}{U}\colon (\order{i}{U})_{i=1}^m \mapsto (\order{i}{x})_{i=1}^m
\end{equation}
determines an injective function $\layerings{k}{U} \incl \orderings{k}{U}$.
\end{prop}

\begin{dfn}
In general, the function $\lto{k}{U}$ is not surjective, that is, not every $k$\nbd ordering is induced by a $k$\nbd layering.
The following is a criterion for deciding when a $k$\nbd ordering comes from a $k$\nbd layering.
\end{dfn}

\begin{prop} \label{prop:layering_from_ordering}
Let $U$ be a regular molecule, $k \geq -1$, and let $(\order{i}{x})_{i=1}^m$ be a $k$\nbd ordering of $U$.
Let
\begin{align*}
    \order{0}{U} & \eqdef \bound{k}{-}U, \\
    \order{i}{U} & \eqdef \bound{k}{+}\order{i-1}{U} \cup \clset{{\order{i}{x}}} \quad \text{for $i \in \set{1, \ldots, m}$}.
\end{align*}
The following are equivalent:
\begin{enumerate}[label=(\alph*)]
    \item $(\order{i}{U})_{i=1}^m$ is a $k$\nbd layering of $U$;
    \item for all $i \in \set{1, \ldots, m}$, $\bound{k}{-}\order{i}{x} \submol \bound{k}{-}\order{i}{U}$. \label{cond:inputs_are_submolecules}
\end{enumerate}
Moreover, for all $i \in \set{1, \ldots, m-1}$, if $\bound{k}{-}\order{i}{x} \submol \bound{k}{-}\order{i}{U}$, then $\order{i}{U}$ and $\bound{k}{+}\order{i}{U} = \bound{k}{-}\order{i+1}{U}$ are regular molecules.
\end{prop}

\begin{dfn}[Path-induced subgraph]
Let $\mathscr{G}$ be a directed graph and $W$ a subset of its vertex set.
We say that $\restr{\mathscr{G}}{W}$ is \emph{path-induced} if, for all $x, y \in W$, every path from $x$ to $y$ in $\mathscr{G}$ is included in $\restr{\mathscr{G}}{W}$.
\end{dfn}

\begin{dfn}
Path-induced subgraphs are also called \emph{convex subgraphs}, for example in \cite{bonchi2022string}.
\end{dfn}

\begin{dfn}[Contraction of a connected subgraph]
Let $\mathscr{G}$ be a directed graph and $W$ a subset of its vertex set such that $\restr{\mathscr{G}}{W}$ is connected.
The \emph{contraction of $\restr{\mathscr{G}}{W}$ in $\mathscr{G}$} is the graph minor $\mathscr{G}/(\restr{\mathscr{G}}{W})$ obtained by contracting every edge in $\restr{\mathscr{G}}{W}$.
\end{dfn}

\begin{lem} \label{lem:connected_subgraph_conditions_path_induced}
Let $\mathscr{G}$ be a directed acyclic graph and $W$ a subset of its vertex set such that $\restr{\mathscr{G}}{W}$ is connected.
The following are equivalent:
\begin{enumerate}[label=(\alph*)]
    \item $\restr{\mathscr{G}}{W}$ is path-induced; \label{cond:path_induced}
    \item $\mathscr{G}/(\restr{\mathscr{G}}{W})$ is acyclic; \label{cond:acyclic_contraction}
    \item there is a topological sort of $\mathscr{G}$ in which vertices of $W$ are consecutive. \label{cond:consecutive_tsort}
\end{enumerate}
Moreover, under any of the equivalent conditions, there is a bijection between
\begin{itemize}
    \item topological sorts of $\mathscr{G}$ in which vertices of $W$ are consecutive,
    \item pairs of a topological sort of $\restr{\mathscr{G}}{W}$ and a topological sort of $\mathscr{G}/(\restr{\mathscr{G}}{W})$.
\end{itemize}
\end{lem}

\begin{lem} \label{lem:flow_of_substitution_is_contraction_of_flow}
Let $\imath\colon V \incl U$ be an inclusion of regular molecules such that $n \eqdef \dim{U} = \dim{V}$ and $V$ is round.
Then $\flow{n-1}{\subs{U}{\compos{V}}{\imath(V)}}$ is isomorphic to $\flow{n-1}{U}/\flow{n-1}{V}$.
\end{lem}

\begin{prop} \label{prop:round_submolecule_flow_path_induced}
Let $\imath\colon V \incl U$ be an inclusion of regular molecules such that $n \eqdef \dim{U} = \dim{V}$ and $V$ is round.
If $\imath$ is a submolecule inclusion, then $\flow{n-1}{V}$ is a path-induced subgraph of $\flow{n-1}{U}$.
\end{prop}
\begin{proof}
By Proposition \ref{prop:round_submolecule_substitution}, if $\imath$ is a submolecule inclusion then $\subs{U}{\compos{V}}{\imath(V)}$ is a regular molecule.
By Corollary \ref{cor:flow_acyclic_in_codimension_1} $\flow{n-1}{\subs{U}{\compos{V}}{\imath(V)}}$ is acyclic, and by Lemma \ref{lem:flow_of_substitution_is_contraction_of_flow} it is isomorphic to $\flow{n-1}{U}/\flow{n-1}{V}$.
It follows from Lemma \ref{lem:connected_subgraph_conditions_path_induced} that $\flow{n-1}{V}$ is a path-induced subgraph of $\flow{n-1}{U}$.
\end{proof}

\begin{dfn}
Given an inclusion $\imath\colon V \incl U$ of regular molecules such that $n \eqdef \dim{U} = \dim{V}$ and $V$ is round, the vertices of $\flow{n-1}{U}/\flow{n-1}{V}$ are either
\begin{itemize}
    \item $x \in \grade{n}{U} \setminus \imath(\grade{n}{V})$, or
    \item $x_V$, obtained from contracting all vertices in $\imath(\grade{n}{V})$.
\end{itemize}
The following results will justify our algorithm for the rewritable submolecule problem.
\end{dfn}

\begin{lem} \label{lem:round_submolecules_from_layering}
Let $\imath\colon V \incl U$ be an inclusion of regular molecules such that $n \eqdef \dim{U} = \dim{V}$ and $V$ is round, and let $(\order{i}{y})_{i=1}^p$ be an $(n-1)$\nbd ordering induced by an $(n-1)$\nbd layering of $V$.
The following are equivalent:
\begin{enumerate}[label=(\alph*)]
    \item $\imath$ is a submolecule inclusion;
    \item there exist an $(n-1)$\nbd ordering $(\order{i}{x})_{i=1}^m$ induced by an $(n-1)$\nbd layering $(\order{i}{U})_{i=1}^m$ of $U$, and $q \in \set{1, \ldots, m}$ such that
    \begin{enumerate}[label=\arabic*.]
        \item $(\order{i}{x})_{i=q}^{p+q-1} = (\imath(\order{i}{y}))_{i=1}^p$,
        \item $\imath(\bound{}{-}V) \submol \bound{}{-}\order{q}{U}$.
    \end{enumerate}
\end{enumerate}
\end{lem}

\begin{thm} \label{thm:rewritable_submolecule_criterion}
Let $\imath\colon V \incl U$ be an inclusion of regular molecules such that $n \eqdef \dim{U} = \dim{V}$ and $V$ is round, $m \eqdef \size{\grade{n}{U}}$, $p \eqdef \size{\grade{n}{V}}$.
The following are equivalent:
\begin{enumerate}[label=(\alph*)]
    \item $\imath$ is a submolecule inclusion;
    \item there is a topological sort $((\order{i}{x})_{i=1}^{q-1}, x_V, (\order{i}{x})_{i=q+1}^{m-p+1})$ of $\flow{n-1}{U}/\flow{n-1}{V}$ such that, letting
    \begin{align*}
    \order{0}{U} & \eqdef \bound{}{-}U, \\
    \order{q}{U} & \eqdef \bound{n-1}{+}\order{q-1}{U} \cup \imath(V), \\
    \order{i}{U} & \eqdef \bound{n-1}{+}\order{i-1}{U} \cup \clset{{\order{i}{x}}} \quad \text{for $i \neq q$},
    \end{align*}
    we have $\imath(\bound{}{-}V) \submol \bound{}{-}\order{q}{U}$ and $\bound{}{-}\order{i}{x} \submol \bound{}{-}\order{i}{U}$ for all $i \neq q$.
\end{enumerate}
\end{thm}
\begin{proof}[Sketch of proof]
By Proposition \ref{prop:round_submolecule_substitution}, if $\imath$ is a submolecule inclusion then $\subs{U}{\compos{V}}{\imath(V)}$ is a regular molecule, so by Corollary \ref{cor:codimension_1_layering} it admits an $(n-1)$\nbd layering inducing an $(n-1)$\nbd ordering.
By Lemma \ref{lem:flow_of_substitution_is_contraction_of_flow} this can be identified with a topological sort of $\flow{n-1}{U}/\flow{n-1}{V}$.
The properties of the $\order{i}{U}$ follow from the properties of the layering of $\subs{U}{\compos{V}}{\imath(V)}$, as established by Proposition \ref{prop:layering_from_ordering}, after we ``reverse'' the substitution of $\compos{V}$ for $\imath(V)$, producing a regular molecule isomorphic to $U$.

Conversely, if $(\order{i}{y})_{i=1}^p$ is an $(n-1)$\nbd ordering induced by an $(n-1)$\nbd layering of $V$, then $((\order{i}{x})_{i=1}^{q-1}, (\order{i}{y})_{i=1}^p, (\order{i}{x})_{i=q+1}^{m-p+1})$ is an $(n-1)$\nbd ordering of $U$, which by the criterion of Proposition \ref{prop:layering_from_ordering} is induced by an $(n-1)$\nbd layering. 
We conclude by Lemma \ref{lem:round_submolecules_from_layering}.
\end{proof}

\begin{dfn}[Rewritable submolecule decision algorithm]
The procedure takes as input an inclusion $V \subseteq U$ of regular molecules such that $n \eqdef \dim{U} = \dim{V}$ and $V$ is round, and it returns whether $V \submol U$.
We let $m \eqdef \size{\grade{n}{U}}$ and $p \eqdef \size{\grade{n}{V}}$.

We construct the graph $\mathscr{G} \eqdef \flow{n-1}{U}/\flow{n-1}{V}$.
Then we start a loop.
At each iteration, we search for a new topological sort of $\mathscr{G}$.
If we cannot find one, we return \emph{false}.
Else, let $((\order{i}{x})_{i=1}^{q-1}, x_V, (\order{i}{x})_{i=q+1}^{m-p+1})$ be the new topological sort, and let $(\order{i}{U})_{i=1}^{m-p+1}$ be as in Theorem \ref{thm:rewritable_submolecule_criterion}.

For each $i \in \set{1, \ldots, m - p + 1}$, we start a recursive call to the algorithm to decide whether $\bound{}{-}\order{i}{x} \submol \bound{}{-}\order{i}{U}$ if $i \neq q$, and $\bound{}{-}V \submol \bound{}{-}\order{q}{U}$ if $i = q$.
If this returns \emph{false}, we break the iteration on $i$ and iterate the main loop.
If this returns \emph{true}, we iterate on $i$.
At the end of the iteration on $i$, we return \emph{true}.
\end{dfn}

\begin{thm} \label{prop:correctness_decision}
The rewritable submolecule decision algorithm is correct: it always terminates, and returns \emph{true} if and only if $V \submol U$.
\end{thm}
\begin{proof}
We proceed by induction on the dimension $n$ of $U$ and $V$.
If $n = 0$, this is straightforward, so let $n > 0$.

The number of iterations of the main loop is bounded by the number of topological sorts of $\flow{n-1}{U}/\flow{n-1}{V}$, which is finite.
Consider one such iteration, producing a topological sort $((\order{i}{x})_{i=1}^{q-1}, x_V, (\order{i}{x})_{i=q+1}^{m-p+1})$.
Let us write $\order{q}{V} \eqdef V$ and $\order{i}{V} \eqdef \clset{{\order{i}{x}}}$ for $i \neq q$.
For all $i \in \set{1, \ldots, m-p+1}$, we have a call to the decision algorithm with input $\bound{}{-}\order{i}{V} \subseteq \bound{}{-}\order{i}{U}$, assuming that the calls for $j < i$ all returned \emph{true}.

Now, $\bound{}{-}\order{i}{V}$ is round by Proposition \ref{prop:molecule_properties} and Proposition \ref{prop:atom_properties}.
Moreover, $\bound{}{-}\order{1}{U} = \bound{}{-}U$, which is a regular molecule.
For $i > 1$, assuming that $\bound{}{-}\order{i-1}{V} \submol \bound{}{-}\order{i-1}{U}$, we may apply Proposition \ref{prop:round_submolecule_substitution} to derive that $\order{i-1}{U}$ and $\bound{}{+}\order{i-1}{U} = \bound{}{-}\order{i}{U}$ are regular molecules.
Thus
\begin{enumerate}
    \item the input of the first call is well-formed,
    \item for $i > 1$, assuming that the $(i-1)$-th call correctly returned \emph{true}, the input of the $i$-th call is well-formed.
\end{enumerate}
Since all of these are in dimension $(n-1)$, by the inductive hypothesis, each call terminates returning the correct answer.
By Theorem \ref{thm:rewritable_submolecule_criterion}, this proves both correctness and termination in dimension $n$.
\end{proof}

\begin{thm} \label{thm:rewritable_complexity}
The rewritable submolecule problem in dimension $n$ can be solved in time 
\begin{equation*}
    O\left(\prod_{k \leq n} \size{\grade{k}{U}}!\size{\grade{k}{U}}\right).
\end{equation*}
\end{thm}
\begin{proof}
For $n = 0$, this is obvious, so let $n > 0$.
The number of iterations of the main loop is bounded above by the number of topological sorts of $\mathscr{G} \eqdef \flow{n-1}{U}/\flow{n-1}{V}$.
This reaches its maximum when $\mathscr{G}$ is a discrete graph, in which case the number is $(\size{\grade{n}{U}} - \size{\grade{n}{V}} + 1)!$, tightly bounded above by $\size{\grade{n}{U}}!$.

At each iteration of the main loop, we have at most $\size{\grade{n}{U}} - \size{\grade{n}{V}} + 1$ calls to the algorithm on regular molecules of dimension $n-1$ contained in $U$.
By the inductive hypothesis, these take time $O(\prod_{k \leq n-1} \size{\grade{k}{U}}!\size{\grade{k}{U}})$.

All other operations have lower complexity: both finding topological sorts and computing the boundaries of the $\order{i}{U}$ take linear time in $\size{\edges{n}{U}}$, but this can be bounded above by $\size{\grade{n}{U}}\,\size{\grade{n-1}{U}}$, and we conclude.
\end{proof}

\begin{dfn}
The superpolynomial upper bound on the rewritable submolecule problem leaves it inconclusive whether subdiagram matching admits a polynomial-time algorithm in arbitrary dimension.
Nevertheless, we are at least able to prove that the problem is in $\fun{NP}$.
\end{dfn}

\begin{prop} \label{prop:subdiagram_matching_np}
For all $n \in \mathbb{N}$, the $n$\nbd dimensional subdiagram matching problem is in $\fun{NP}$.
\end{prop}
\begin{proof}
It suffices to prove by induction on $n$ that the rewritable submolecule problem in dimension $n$ is in $\fun{NP}$.
When $n = 0$, the problem is trivial.
In dimension $n > 0$, a polynomial-size certificate that $V \submol U$ is given by
\begin{enumerate}
    \item a topological sort $((\order{i}{x})_{i=1}^{q-1}, x_V, (\order{i}{x})_{i=q+1}^{m-p+1})$ of the graph $\flow{n-1}{U}/\flow{n-1}{V}$, and
    \item polynomial-size certificates that $\bound{}{-}V \submol \bound{}{-}\order{q}{U}$ and $\bound{}{-}\order{i}{x} \submol \bound{}{-}\order{i}{U}$ for all $i \neq q$,
\end{enumerate}
with the notations of Theorem \ref{thm:rewritable_submolecule_criterion}.
By the inductive hypothesis this exists and is verifiable in polynomial time.
\end{proof}

\begin{dfn}
We conclude this section by considering some improvements on the subdiagram matching algorithm conditional on \emph{acyclicity} properties.
\end{dfn}

\begin{dfn}[Frame-acyclic molecule]
Let $U$ be a regular molecule.
We say that $U$ is \emph{frame-acyclic} if for all submolecules $V \submol U$, if $r \eqdef \frdim{V}$, then $\maxflow{r}{V}$ is acyclic.
\end{dfn}

\begin{thm} \label{thm:frame_acyclicity_equivalent_conditions}
Let $U$ be a regular molecule.
The following are equivalent:
\begin{enumerate}[label=(\alph*)]
    \item $U$ is frame-acyclic; \label{cond:frame_acyclic}
    \item for all $V \submol U$ and all $\frdim{V} \leq k < \dim{V}$, $V$ admits a $k$\nbd layering; \label{cond:frdim_layerings}
    \item for all $V \submol U$ and all $\frdim{V} \leq k < \dim{V}$, the sets $\layerings{k}{V}$ and $\orderings{k}{V}$ are non-empty and equinumerous. \label{cond:frdim_bijection}
\end{enumerate}
\end{thm}

\begin{dfn}[Stably frame-acyclic molecule]
Let $U$ be a regular molecule.
We say that $U$ is \emph{stably frame-acyclic} if for all submolecules $V \submol U$ and all rewritable submolecules $W \submol V$, the regular molecule $\subs{V}{\compos{W}}{W}$ is frame-acyclic.
\end{dfn}

\begin{dfn}
Every stably frame-acyclic regular molecule is frame-acyclic: if we take $V \submol U$ to be an atom, the substitution $\subs{U}{\compos{V}}{V}$ is trivial.
Moreover, every submolecule of a (stably) frame-acyclic regular molecule is (stably) frame-acyclic.

We are not aware of examples of regular molecules that are stably frame-acyclic but not frame-acyclic (as we will see in the next section, any such example is at least 4\nbd dimensional), so we cannot exclude that the two classes coincide, but neither it seems clear that they do.
\end{dfn}

\begin{dfn}
In general, frame-acyclicity seems difficult to check.
However, it is implied by stronger acyclicity conditions that are easier to check.
We do not know any easily verifiable sufficient conditions for stable frame-acyclicity.
\end{dfn}

\begin{dfn}[Acyclic molecule]
Let $U$ be a regular molecule.
We say that $U$ is \emph{acyclic} if $\hasseo{U}$ is acyclic.
\end{dfn}

\begin{dfn}[Dimension-wise acyclic molecule]
Let $U$ be a regular molecule.
We say that $U$ is \emph{dimension-wise acyclic} if, for all $k \in \mathbb{N}$, $\flow{k}{U}$ is acyclic.
\end{dfn}

\begin{lem} \label{lem:acyclicity_implications}
Let $U$ be a regular molecule.
Then
\begin{enumerate}
    \item if $U$ is acyclic, it is dimension-wise acyclic;
    \item if $U$ is dimension-wise acyclic, it is frame-acyclic.
\end{enumerate}
\end{lem}

\begin{exm}
Both implications are strict.
The 3\nbd dimensional atom
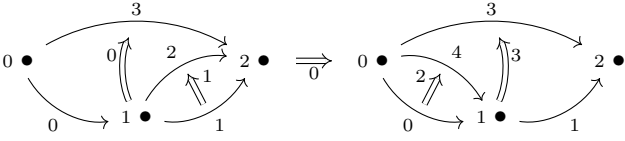
\[\begin{tikzcd}[sep=small]
	{{\scriptstyle 0}\;\bullet} &&&& {{\scriptstyle 2}\;\bullet} && {{\scriptstyle 0}\;\bullet} &&&& {{\scriptstyle 2}\;\bullet} \\
	&& {{\scriptstyle 1}\;\bullet} &&&&&& {{\scriptstyle 1}\;\bullet}
	\arrow[""{name=0, anchor=center, inner sep=0}, "3", curve={height=-18pt}, from=1-1, to=1-5]
	\arrow[""{name=1, anchor=center, inner sep=0}, "3", curve={height=-18pt}, from=1-7, to=1-11]
	\arrow[""{name=2, anchor=center, inner sep=0}, "4", curve={height=-12pt}, from=1-7, to=2-9]
	\arrow[""{name=3, anchor=center, inner sep=0}, "2", curve={height=-12pt}, from=2-3, to=1-5]
	\arrow["0"', shorten <=5pt, shorten >=5pt, Rightarrow, from=1-5, to=1-7]
	\arrow["0"', curve={height=12pt}, from=1-1, to=2-3]
	\arrow[""{name=4, anchor=center, inner sep=0}, "1"', curve={height=12pt}, from=2-3, to=1-5]
	\arrow[""{name=5, anchor=center, inner sep=0}, "0"', curve={height=12pt}, from=1-7, to=2-9]
	\arrow["1"', curve={height=12pt}, from=2-9, to=1-11]
	\arrow["1"', shorten <=3pt, shorten >=3pt, Rightarrow, from=4, to=3]
	\arrow["2", shorten <=3pt, shorten >=3pt, Rightarrow, from=5, to=2]
	\arrow["3"', curve={height=6pt}, shorten >=6pt, Rightarrow, from=2-9, to=1]
	\arrow["0", curve={height=-6pt}, shorten >=6pt, Rightarrow, from=2-3, to=0]
\end{tikzcd}\]
(based on \cite[Fig.\ 2]{steiner1993algebra}) is not acyclic, since its oriented Hasse diagram contains the cycle
\begin{equation*}
    (0, 1) \to (1, 1) \to (2, 1) \to (3, 0) \to (2, 2) \to (1, 4) \to (0, 1),
\end{equation*}
but it is dimension-wise acyclic.
The 3\nbd dimensional atom
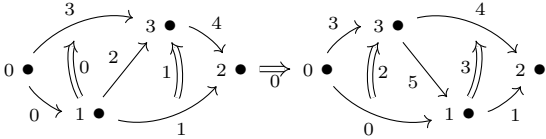
\[\begin{tikzcd}[sep=tiny]
	&& {{\scriptstyle 3}\;\bullet} &&&& {{\scriptstyle 3}\;\bullet} \\
	{{\scriptstyle 0}\;\bullet} &&& {{\scriptstyle 2}\;\bullet} && {{\scriptstyle 0}\;\bullet} &&& {{\scriptstyle 2}\;\bullet} \\
	& {{\scriptstyle 1}\;\bullet} &&&&&& {{\scriptstyle 1}\;\bullet}
	\arrow["0"', Rightarrow, from=2-4, to=2-6]
	\arrow["0"', curve={height=6pt}, from=2-1, to=3-2]
	\arrow[""{name=0, anchor=center, inner sep=0}, "1"', curve={height=12pt}, from=3-2, to=2-4]
	\arrow[""{name=1, anchor=center, inner sep=0}, "0"', curve={height=12pt}, from=2-6, to=3-8]
	\arrow["1"', curve={height=6pt}, from=3-8, to=2-9]
	\arrow[""{name=2, anchor=center, inner sep=0}, "4", curve={height=-12pt}, from=1-7, to=2-9]
	\arrow["3", curve={height=-6pt}, from=2-6, to=1-7]
	\arrow[""{name=3, anchor=center, inner sep=0}, "3", curve={height=-12pt}, from=2-1, to=1-3]
	\arrow["4", curve={height=-6pt}, from=1-3, to=2-4]
	\arrow["2", from=3-2, to=1-3]
	\arrow["5"', from=1-7, to=3-8]
	\arrow["0"', curve={height=-6pt}, shorten >=7pt, Rightarrow, from=3-2, to=3]
	\arrow["1", curve={height=6pt}, shorten <=7pt, Rightarrow, from=0, to=1-3]
	\arrow["2"', curve={height=-6pt}, shorten <=7pt, Rightarrow, from=1, to=1-7]
	\arrow["3", curve={height=6pt}, shorten >=7pt, Rightarrow, from=3-8, to=2]
\end{tikzcd}\]
(based on \cite[Fig.\ 4]{steiner1993algebra}) is not dimension-wise acyclic, since its 0\nbd flow graph contains the cycle $(1, 2) \to (1, 5) \to (1, 2)$, but it is frame-acyclic by Theorem \ref{thm:dim3_frame_acyclic}.
\end{exm}

\begin{prop} \label{prop:complexity_frame_acyclic}
If $U$ is guaranteed to be frame-acyclic, the rewritable submolecule problem in dimension $n$ can be solved in time
\begin{equation*}
    O\left(\prod_{k \leq n} \size{\grade{k}{U}}!\right).
\end{equation*}
\end{prop}
\begin{proof}
Given a topological sort $((\order{i}{x})_{i=1}^{q-1}, x_V, (\order{i}{x})_{i=q+1}^{m-p+1})$ of $\flow{n-1}{U}/\flow{n-1}{V}$, by Lemma \ref{lem:connected_subgraph_conditions_path_induced}, substituting any $(n-1)$\nbd ordering of $V$ for $x_V$ produces an $(n-1)$\nbd ordering of $U$ in which the elements of $V$ are consecutive.
By Theorem \ref{thm:frame_acyclicity_equivalent_conditions}, this is induced by an $(n-1)$\nbd layering of $U$.
By Lemma \ref{lem:round_submolecules_from_layering}, it suffices to check that $\bound{}{-}V \submol \bound{}{-}\order{q}{U}$ to conclude that $V \submol U$, so we have a single recursive call instead of $O(\size{\grade{n}{U}})$ many.
Since $\bound{}{-}\order{q}{U} \submol \order{q}{U} \submol U$, it is frame-acyclic, and we can proceed inductively.
\end{proof}

\begin{prop} \label{prop:complexity_stably_frame_acyclic}
If $U$ is guaranteed to be stably frame-acyclic, the rewritable submolecule problem can be solved in linear time in the size of $\hasse{U}$.
\end{prop}
\begin{proof}
If $V \submol U$, by assumption $\subs{U}{\compos{V}}{V}$ is frame-acyclic.
By Theorem \ref{thm:frame_acyclicity_equivalent_conditions}, \emph{all} its $(n-1)$\nbd orderings are induced by $(n-1)$\nbd layerings, and by Lemma \ref{lem:flow_of_substitution_is_contraction_of_flow} they are in bijection with topological sorts of $\flow{n-1}{U}/\flow{n-1}{V}$.
It follows that, if any topological sort of $\flow{n-1}{U}/\flow{n-1}{V}$ fails to satisfy the conditions of Theorem \ref{thm:rewritable_submolecule_criterion}, then $V$ is not a submolecule of $U$, so in the decision algorithm we can stop after the first iteration of the main loop.

This involves finding a single topological sort and computing $\bound{}{-}\order{q}{U}$, both of which take time $O(\size{\edges{n}{U}})$; then, as in Proposition \ref{prop:complexity_frame_acyclic}, we make a single call to the decision algorithm for $\bound{}{-}V \submol \bound{}{-}\order{q}{U}$.
Since $\bound{}{-}\order{q}{U} \submol \order{q}{U} \submol U$, it is stably frame-acyclic, and we can proceed inductively.
\end{proof}

\begin{cor}
The subdiagram matching problem restricted to diagrams with stably frame-acyclic shape is in $\fun{P}$.
\end{cor}

\begin{dfn}
Proposition \ref{prop:complexity_stably_frame_acyclic} begs the question: why not just develop a higher-dimensional rewriting theory around stably frame-acyclic shapes of diagrams?
The reason is that, in general, acyclicity properties are \emph{global} properties of diagrams, that are not stable under local substitutions, essential to rewriting theory.
Indeed, the inductive definition of regular molecules makes them ``minimal'' for a class of shapes closed under pasting and rewrites of round diagrams, with roundness seemingly the natural condition ensuring both topological soundness and a good combinatorial account of substitution.
Any further restriction would almost certainly be impractical from a rewriting-theoretic perspective.

Nevertheless, the following section will show that \emph{up to dimension 3} there is no restriction at all: all regular molecules are stably frame-acyclic, and in fact we can further simplify our algorithms.
\end{dfn}
\section{In low dimensions} \label{sec:lowdim}

\begin{dfn}
Some results in this section come from \cite[Section 3]{hadzihasanovic2021smash}.
\end{dfn}

\begin{lem} \label{lem:only_1_molecules}
Let $U$ be a 1\nbd dimensional regular molecule, $m \eqdef \size{\grade{1}{U}}$.
Then $\hasseo{U}$ is a linear graph with $(2m+1)$ vertices, and $\flow{0}{U}$ is a linear graph with $m$ vertices.
\end{lem}

\begin{prop} \label{prop:in_dim1_all_inclusions_are_submolecule}
Let $\imath\colon V \incl U$ be an inclusion of regular molecules, $\dim{V} = \dim{U} = 1$.
Then $\imath$ is a submolecule inclusion.
\end{prop}
\begin{proof}
By Lemma \ref{lem:flow_under_inclusion} $\flow{0}{V}$ is an induced subgraph of $\flow{0}{U}$.
By Lemma \ref{lem:only_1_molecules} both of them are linear graphs, and an induced subgraph of a linear graph is a linear graph if and only if its vertices are consecutive in the ambient graph.
All other conditions of Lemma \ref{lem:round_submolecules_from_layering} are trivially satisfied.
\end{proof}

\begin{prop} \label{prop:dim2_acyclic}
Let $U$ be a regular molecule, $\dim{U} \leq 2$.
Then $U$ is acyclic.
\end{prop}

\begin{lem} \label{lem:dim2_inclusion_is_path_induced}
Let $\imath\colon V \incl U$ be an inclusion of regular molecules, $\dim{U} \leq 2$.
Then $\flow{1}{V}$ is a path-induced subgraph of $\flow{1}{U}$.
\end{lem}

\begin{thm} \label{thm:round_submolecule_dim2}
Let $\imath\colon V \incl U$ be an inclusion of regular molecules such that $\dim{U} = \dim{V} = 2$ and $V$ is round.
Then $\imath$ is a submolecule inclusion.
\end{thm}
\begin{proof}
By Lemma \ref{lem:dim2_inclusion_is_path_induced} combined with Lemma \ref{lem:connected_subgraph_conditions_path_induced}, there exists a 1\nbd ordering $(\order{i}{x})_{i=1}^m$ of $U$ in which the elements of $\imath(V)$ are consecutive, that is, $\order{i}{x} \in \imath(V)$ if and only if $p \leq i \leq q$ for some $p, q \in \set{1, \ldots, m}$.

By Proposition \ref{prop:dim2_acyclic} $U$ is acyclic, so by Lemma \ref{lem:acyclicity_implications} it is frame-acyclic, and by Theorem \ref{thm:frame_acyclicity_equivalent_conditions} the 1\nbd ordering comes from a 1\nbd layering $(\order{i}{U})_{i=1}^m$ such that $\imath(\bound{}{-}V) \subseteq \bound{}{-}\order{p}{U}$.
Since both are 1\nbd dimensional regular molecules, by Proposition \ref{prop:in_dim1_all_inclusions_are_submolecule} $\imath(\bound{}{-}V) \submol \bound{}{-}\order{p}{U}$, and Lemma \ref{lem:round_submolecules_from_layering} allows us to conclude.
\end{proof}

\begin{cor}
The rewritable submolecule problem in dimension $\leq 2$ has a trivial constant-time solution.
\end{cor}

\begin{thm} \label{thm:dim3_frame_acyclic}
Let $U$ be a regular molecule, $\dim{U} \leq 3$.
Then $U$ is stably frame-acyclic.
\end{thm}

\begin{thm} \label{thm:round_submolecule_dim3}
Let $\imath\colon V \incl U$ be an inclusion of regular molecules such that $\dim{U} = \dim{V} = 3$ and $V$ is round.
The following are equivalent:
\begin{enumerate}[label=(\alph*)]
    \item $\imath$ is a submolecule inclusion;
    \item $\flow{2}{V}$ is a path-induced subgraph of $\flow{2}{U}$.
\end{enumerate}
\end{thm}
\begin{proof}
One implication is Proposition \ref{prop:round_submolecule_flow_path_induced}, so we only need to prove the converse.
By Lemma \ref{lem:connected_subgraph_conditions_path_induced}, if $\flow{2}{V}$ is path-induced, then there exists a 2\nbd ordering $(\order{i}{x})_{i=1}^m$ of $U$ in which the elements of $\imath(V)$ are consecutive, that is, $\order{i}{x} \in \imath(V)$ if and only if $p \leq i \leq q$ for some $p, q \in \set{1, \ldots, m}$.
By Theorem \ref{thm:dim3_frame_acyclic}, $U$ is frame-acyclic, so by Theorem \ref{thm:frame_acyclicity_equivalent_conditions} the 2\nbd ordering comes from a 2\nbd layering $(\order{i}{U})_{i=1}^m$ such that $\imath(\bound{}{-}V) \subseteq \bound{}{-}\order{p}{U}$.
Since both are 2\nbd dimensional regular molecules and $\bound{}{-}V$ is round, by Theorem \ref{thm:round_submolecule_dim2} $\imath(\bound{}{-}V) \submol \bound{}{-}\order{p}{U}$, and we conclude by Lemma \ref{lem:round_submolecules_from_layering}.
\end{proof}

\begin{thm} \label{thm:3dim_rewritable_submolecule}
The rewritable submolecule problem in dimension 3 can be solved in time $O(\size{\edges{3}{U}})$.
\end{thm}
\begin{proof}
By Theorem \ref{thm:round_submolecule_dim3} combined with Lemma \ref{lem:connected_subgraph_conditions_path_induced}, it suffices to construct $\flow{2}{U}/\flow{2}{V}$ and check if it is acyclic.
The first can be done while traversing the induced subgraph of $\hasse{U}$ on $U_3 \cup U_2$, which takes time $O(\size{\edges{3}{U}})$.
Both the number of vertices and the number of edges of $\flow{2}{U}/\flow{2}{V}$ is also $O(\size{\edges{3}{U}})$, and we conclude.
\end{proof}

\begin{dfn}
To match a diagram $s\colon V \to \mathbb{V}$ in $t\colon U \to \mathbb{V}$ in dimension 3, according to our results we need
\begin{itemize}
    \item $O(\size{\grade{3}{U}}\, \size{\grade{3}{V}} \, \size{V} \, \size{\edges{\lor}{V}}\,\log \size{\edges{\lor}{V}})$ time to find all inclusions $V \incl U$, of which there are $O(\size{\grade{3}{U}})$,
    \item $O(\size{\edges{3}{U}})$ time to check whether each of them is a submolecule inclusion,
    \item $O(\size{V})$ time to compare labellings on each, assuming labels can be compared in constant time,
\end{itemize}
leading to an overall 
\begin{equation*}
    O(\size{\grade{3}{U}} (\size{\edges{3}{U}} + \size{\grade{3}{V}} \, \size{V} \, \size{\edges{\lor}{V}}\,\log \size{\edges{\lor}{V}}))
\end{equation*}
upper bound.
Here we used the bound on molecule matching in generic dimension; it is possible that this can be improved by using strategies tailored to dimension 3, as it is certainly the case in dimension $\leq 2$.

If we consider a machine operating by rewriting 3\nbd dimensional diagrams, which has a fixed finite list of rewrite rules, the variables linked to $V$ can be considered as constant parameters of the machine.
Our results then imply that such a machine can be simulated with $O(\size{\grade{3}{U}} \, \size{\edges{3}{U}})$ overhead in a standard model of computation.
\end{dfn}

\begin{dfn}
We leave the existence of a polynomial algorithm for subdiagram matching in dimension 4 or higher as an open problem.
The main obstacle to overcome is the expensive iteration on topological sorts, motivated by the fact that, from dimension 4 onwards, not all of them arise from layerings.
One may hope, perhaps, that this is due to flow graphs ``missing'' some relations, and that it should be possible to supplement them with extra information, in such a way as to restore the bijective correspondence between layerings and topological sorts.
Unfortunately, this cannot be the case in general, as the following counterexample shows.
\end{dfn}

\begin{exm} \label{exm:no_topological_sorts}
\begin{figure}[!t]
\begin{equation*}
	\input{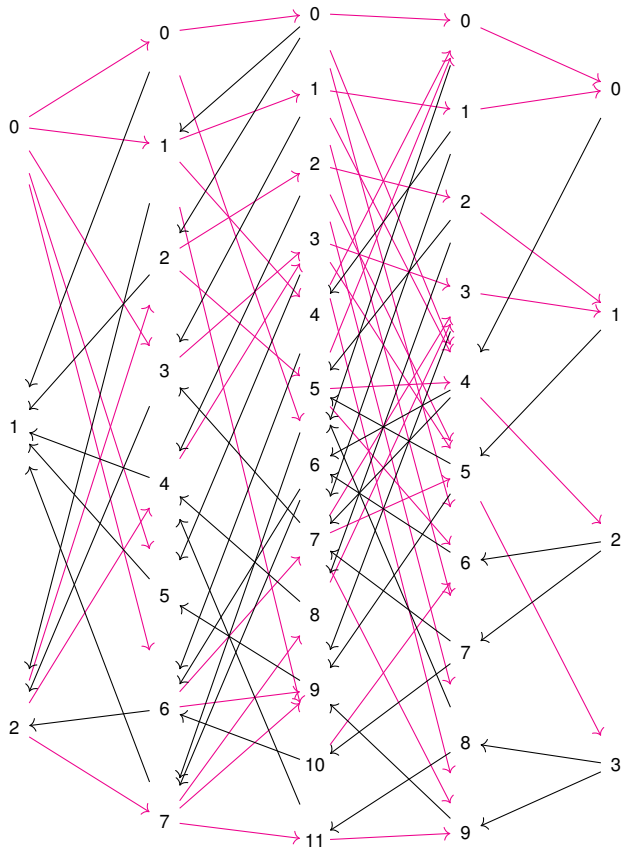}
\end{equation*}
\caption{Oriented Hasse diagram of Example \ref{exm:no_topological_sorts}.}
\label{fig:framecycle}
\end{figure}
This example is a 4\nbd dimensional regular molecule $U$ which is not frame-acyclic.
Its Hasse diagram is given in Figure \ref{fig:framecycle}, with dimensions increasing from left to right.
    
To understand this example, we can picture the 4\nbd dimensional elements of $U$ as rewrites of planar projections of 3\nbd dimensional diagrams, portrayed as string diagrams.
Then $U$ has one 3\nbd layering inducing the 3\nbd ordering
\begin{equation*}
    ((4, 1), (4, 3), (4, 0), (4, 2))
\end{equation*}
corresponding to the sequence of rewrite steps
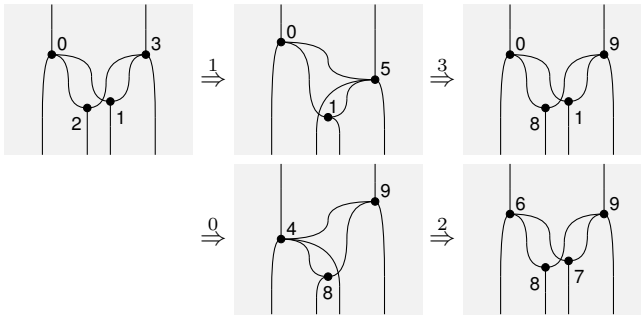
\begin{align*}
    \begin{tikzpicture}[xscale=2.5, yscale=2, baseline={([yshift=-.5ex]current bounding box.center)}]
\path[fill=gray!10] (0, 0) rectangle (1, 1);
\draw[color=black, opacity=1] (0.438595933733141, 0.31169898098285276) .. controls (0.5052626003998076, 0.31169898098285276) and (0.538595933733141, 0.36725453653840834) .. (0.538595933733141, 0.47836564764951944);
\draw[color=black, opacity=1] (0.75, 0.6666666666666666) .. controls (0.609063955822094, 0.6666666666666666) and (0.538595933733141, 0.6038996603276175) .. (0.538595933733141, 0.47836564764951944);
\draw[color=black, opacity=1] (0.438595933733141, 0.31169898098285276) .. controls (0.37192926706647433, 0.31169898098285276) and (0.338595933733141, 0.36725453653840834) .. (0.338595933733141, 0.47836564764951944);
\draw[color=black, opacity=1] (0.25, 0.6666666666666666) .. controls (0.309063955822094, 0.6666666666666666) and (0.338595933733141, 0.6038996603276175) .. (0.338595933733141, 0.47836564764951944);
\draw[color=black, opacity=1] (0.438595933733141, 0.31169898098285276) .. controls (0.438595933733141, 0.31169898098285276) and (0.438595933733141, 0.2561434254272972) .. (0.438595933733141, 0.14503231431618607);
\draw[color=black, opacity=1] (0.438595933733141, 0) .. controls (0.438595933733141, 0.0) and (0.438595933733141, 0.048344104772062024) .. (0.438595933733141, 0.14503231431618607);
\draw[color=black, opacity=1] (0.75, 0.6666666666666666) .. controls (0.75, 0.6666666666666666) and (0.75, 0.7222222222222222) .. (0.75, 0.8333333333333334);
\draw[color=black, opacity=1] (0.75, 1) .. controls (0.75, 1.0) and (0.75, 0.9444444444444444) .. (0.75, 0.8333333333333334);
\draw[color=black, opacity=1] (0.561404066266859, 0.35496768568381387) .. controls (0.6280707329335257, 0.35496768568381387) and (0.661404066266859, 0.41052324123936945) .. (0.661404066266859, 0.5216343523504806);
\draw[color=black, opacity=1] (0.75, 0.6666666666666666) .. controls (0.690936044177906, 0.6666666666666666) and (0.661404066266859, 0.6183225618946047) .. (0.661404066266859, 0.5216343523504806);
\draw[color=black, opacity=1] (0.25, 0.6666666666666666) .. controls (0.25, 0.6666666666666666) and (0.25, 0.7222222222222222) .. (0.25, 0.8333333333333334);
\draw[color=black, opacity=1] (0.25, 1) .. controls (0.25, 1.0) and (0.25, 0.9444444444444444) .. (0.25, 0.8333333333333334);
\draw[color=black, opacity=1] (0.561404066266859, 0.35496768568381387) .. controls (0.49473739960019236, 0.35496768568381387) and (0.46140406626685904, 0.41052324123936945) .. (0.46140406626685904, 0.5216343523504806);
\draw[color=black, opacity=1] (0.25, 0.6666666666666666) .. controls (0.390936044177906, 0.6666666666666666) and (0.46140406626685904, 0.6183225618946047) .. (0.46140406626685904, 0.5216343523504806);
\draw[color=black, opacity=1] (0.75, 0.6666666666666666) .. controls (0.7833333333333333, 0.6666666666666666) and (0.8, 0.5277777777777778) .. (0.8, 0.25);
\draw[color=black, opacity=1] (0.8, 0) .. controls (0.8, 0.0) and (0.8, 0.08333333333333333) .. (0.8, 0.25);
\draw[color=black, opacity=1] (0.561404066266859, 0.35496768568381387) .. controls (0.561404066266859, 0.35496768568381387) and (0.561404066266859, 0.2994121301282583) .. (0.561404066266859, 0.18830101901714724);
\draw[color=black, opacity=1] (0.561404066266859, 0) .. controls (0.561404066266859, 0.0) and (0.561404066266859, 0.06276700633904908) .. (0.561404066266859, 0.18830101901714724);
\draw[color=black, opacity=1] (0.25, 0.6666666666666666) .. controls (0.21666666666666667, 0.6666666666666666) and (0.2, 0.5277777777777778) .. (0.2, 0.25);
\draw[color=black, opacity=1] (0.2, 0) .. controls (0.2, 0.0) and (0.2, 0.08333333333333333) .. (0.2, 0.25);
\node[circle, fill=black, draw=black, inner sep=1pt] at (0.25, 0.6666666666666666) {};
\node[circle, fill=black, draw=black, inner sep=1pt] at (0.561404066266859, 0.35496768568381387) {};
\node[circle, fill=black, draw=black, inner sep=1pt] at (0.438595933733141, 0.31169898098285276) {};
\node[circle, fill=black, draw=black, inner sep=1pt] at (0.75, 0.6666666666666666) {};
\node[text=black, font={\scriptsize \sffamily}, xshift=4pt, yshift=4pt] at (0.25, 0.6666666666666666) {0};
\node[text=black, font={\scriptsize \sffamily}, xshift=4pt, yshift=-6pt] at (0.561404066266859, 0.35496768568381387) {1};
\node[text=black, font={\scriptsize \sffamily}, xshift=-4pt, yshift=-6pt] at (0.438595933733141, 0.31169898098285276) {2};
\node[text=black, font={\scriptsize \sffamily}, xshift=4pt, yshift=4pt] at (0.75, 0.6666666666666666) {3};
\end{tikzpicture} & \stackrel{1}{\celto} \begin{tikzpicture}[xscale=2.5, yscale=2, baseline={([yshift=-.5ex]current bounding box.center)}]
\path[fill=gray!10] (0, 0) rectangle (1, 1);
\draw[color=black, opacity=1] (0.75, 0.5) .. controls (0.5833333333333333, 0.5) and (0.5, 0.5416666666666666) .. (0.5, 0.625);
\draw[color=black, opacity=1] (0.25, 0.75) .. controls (0.41666666666666663, 0.75) and (0.5, 0.7083333333333334) .. (0.5, 0.625);
\draw[color=black, opacity=1] (0.75, 0.5) .. controls (0.5423972891554274, 0.5) and (0.438595933733141, 0.3834803008012687) .. (0.438595933733141, 0.15044090240380623);
\draw[color=black, opacity=1] (0.438595933733141, 0) .. controls (0.438595933733141, 0.0) and (0.438595933733141, 0.050146967467935406) .. (0.438595933733141, 0.15044090240380623);
\draw[color=black, opacity=1] (0.75, 0.5) .. controls (0.75, 0.5) and (0.75, 0.6111111111111112) .. (0.75, 0.8333333333333334);
\draw[color=black, opacity=1] (0.75, 1) .. controls (0.75, 1.0) and (0.75, 0.9444444444444444) .. (0.75, 0.8333333333333334);
\draw[color=black, opacity=1] (0.5, 0.25) .. controls (0.5666666666666667, 0.25) and (0.6, 0.29166666666666663) .. (0.6, 0.375);
\draw[color=black, opacity=1] (0.75, 0.5) .. controls (0.6499999999999999, 0.5) and (0.6, 0.4583333333333333) .. (0.6, 0.375);
\draw[color=black, opacity=1] (0.25, 0.75) .. controls (0.25, 0.75) and (0.25, 0.7916666666666667) .. (0.25, 0.875);
\draw[color=black, opacity=1] (0.25, 1) .. controls (0.25, 1.0) and (0.25, 0.9583333333333333) .. (0.25, 0.875);
\draw[color=black, opacity=1] (0.5, 0.25) .. controls (0.43333333333333335, 0.25) and (0.4, 0.3333333333333333) .. (0.4, 0.5);
\draw[color=black, opacity=1] (0.25, 0.75) .. controls (0.35, 0.75) and (0.4, 0.6666666666666666) .. (0.4, 0.5);
\draw[color=black, opacity=1] (0.75, 0.5) .. controls (0.7833333333333333, 0.5) and (0.8, 0.38888888888888884) .. (0.8, 0.16666666666666666);
\draw[color=black, opacity=1] (0.8, 0) .. controls (0.8, 0.0) and (0.8, 0.05555555555555555) .. (0.8, 0.16666666666666666);
\draw[color=black, opacity=1] (0.5, 0.25) .. controls (0.540936044177906, 0.25) and (0.561404066266859, 0.21374192142095347) .. (0.561404066266859, 0.14122576426286043);
\draw[color=black, opacity=1] (0.561404066266859, 0) .. controls (0.561404066266859, 0.0) and (0.561404066266859, 0.04707525475428681) .. (0.561404066266859, 0.14122576426286043);
\draw[color=black, opacity=1] (0.25, 0.75) .. controls (0.21666666666666667, 0.75) and (0.2, 0.5833333333333334) .. (0.2, 0.25);
\draw[color=black, opacity=1] (0.2, 0) .. controls (0.2, 0.0) and (0.2, 0.08333333333333333) .. (0.2, 0.25);
\node[circle, fill=black, draw=black, inner sep=1pt] at (0.25, 0.75) {};
\node[circle, fill=black, draw=black, inner sep=1pt] at (0.5, 0.25) {};
\node[circle, fill=black, draw=black, inner sep=1pt] at (0.75, 0.5) {};
\node[text=black, font={\scriptsize \sffamily}, xshift=4pt, yshift=4pt] at (0.25, 0.75) {0};
\node[text=black, font={\scriptsize \sffamily}, xshift=2pt, yshift=4pt] at (0.5, 0.25) {1};
\node[text=black, font={\scriptsize \sffamily}, xshift=4pt, yshift=4pt] at (0.75, 0.5) {5};
\end{tikzpicture} \stackrel{3}{\celto} \begin{tikzpicture}[xscale=2.5, yscale=2, baseline={([yshift=-.5ex]current bounding box.center)}]
\path[fill=gray!10] (0, 0) rectangle (1, 1);
\draw[color=black, opacity=1] (0.438595933733141, 0.31169898098285276) .. controls (0.5052626003998076, 0.31169898098285276) and (0.538595933733141, 0.36725453653840834) .. (0.538595933733141, 0.47836564764951944);
\draw[color=black, opacity=1] (0.75, 0.6666666666666666) .. controls (0.609063955822094, 0.6666666666666666) and (0.538595933733141, 0.6038996603276175) .. (0.538595933733141, 0.47836564764951944);
\draw[color=black, opacity=1] (0.438595933733141, 0.31169898098285276) .. controls (0.37192926706647433, 0.31169898098285276) and (0.338595933733141, 0.36725453653840834) .. (0.338595933733141, 0.47836564764951944);
\draw[color=black, opacity=1] (0.25, 0.6666666666666666) .. controls (0.309063955822094, 0.6666666666666666) and (0.338595933733141, 0.6038996603276175) .. (0.338595933733141, 0.47836564764951944);
\draw[color=black, opacity=1] (0.438595933733141, 0.31169898098285276) .. controls (0.438595933733141, 0.31169898098285276) and (0.438595933733141, 0.2561434254272972) .. (0.438595933733141, 0.14503231431618607);
\draw[color=black, opacity=1] (0.438595933733141, 0) .. controls (0.438595933733141, 0.0) and (0.438595933733141, 0.048344104772062024) .. (0.438595933733141, 0.14503231431618607);
\draw[color=black, opacity=1] (0.75, 0.6666666666666666) .. controls (0.75, 0.6666666666666666) and (0.75, 0.7222222222222222) .. (0.75, 0.8333333333333334);
\draw[color=black, opacity=1] (0.75, 1) .. controls (0.75, 1.0) and (0.75, 0.9444444444444444) .. (0.75, 0.8333333333333334);
\draw[color=black, opacity=1] (0.561404066266859, 0.35496768568381387) .. controls (0.6280707329335257, 0.35496768568381387) and (0.661404066266859, 0.41052324123936945) .. (0.661404066266859, 0.5216343523504806);
\draw[color=black, opacity=1] (0.75, 0.6666666666666666) .. controls (0.690936044177906, 0.6666666666666666) and (0.661404066266859, 0.6183225618946047) .. (0.661404066266859, 0.5216343523504806);
\draw[color=black, opacity=1] (0.25, 0.6666666666666666) .. controls (0.25, 0.6666666666666666) and (0.25, 0.7222222222222222) .. (0.25, 0.8333333333333334);
\draw[color=black, opacity=1] (0.25, 1) .. controls (0.25, 1.0) and (0.25, 0.9444444444444444) .. (0.25, 0.8333333333333334);
\draw[color=black, opacity=1] (0.561404066266859, 0.35496768568381387) .. controls (0.49473739960019236, 0.35496768568381387) and (0.46140406626685904, 0.41052324123936945) .. (0.46140406626685904, 0.5216343523504806);
\draw[color=black, opacity=1] (0.25, 0.6666666666666666) .. controls (0.390936044177906, 0.6666666666666666) and (0.46140406626685904, 0.6183225618946047) .. (0.46140406626685904, 0.5216343523504806);
\draw[color=black, opacity=1] (0.75, 0.6666666666666666) .. controls (0.7833333333333333, 0.6666666666666666) and (0.8, 0.5277777777777778) .. (0.8, 0.25);
\draw[color=black, opacity=1] (0.8, 0) .. controls (0.8, 0.0) and (0.8, 0.08333333333333333) .. (0.8, 0.25);
\draw[color=black, opacity=1] (0.561404066266859, 0.35496768568381387) .. controls (0.561404066266859, 0.35496768568381387) and (0.561404066266859, 0.2994121301282583) .. (0.561404066266859, 0.18830101901714724);
\draw[color=black, opacity=1] (0.561404066266859, 0) .. controls (0.561404066266859, 0.0) and (0.561404066266859, 0.06276700633904908) .. (0.561404066266859, 0.18830101901714724);
\draw[color=black, opacity=1] (0.25, 0.6666666666666666) .. controls (0.21666666666666667, 0.6666666666666666) and (0.2, 0.5277777777777778) .. (0.2, 0.25);
\draw[color=black, opacity=1] (0.2, 0) .. controls (0.2, 0.0) and (0.2, 0.08333333333333333) .. (0.2, 0.25);
\node[circle, fill=black, draw=black, inner sep=1pt] at (0.25, 0.6666666666666666) {};
\node[circle, fill=black, draw=black, inner sep=1pt] at (0.561404066266859, 0.35496768568381387) {};
\node[circle, fill=black, draw=black, inner sep=1pt] at (0.438595933733141, 0.31169898098285276) {};
\node[circle, fill=black, draw=black, inner sep=1pt] at (0.75, 0.6666666666666666) {};
\node[text=black, font={\scriptsize \sffamily}, xshift=4pt, yshift=4pt] at (0.25, 0.6666666666666666) {0};
\node[text=black, font={\scriptsize \sffamily}, xshift=4pt, yshift=-6pt] at (0.561404066266859, 0.35496768568381387) {1};
\node[text=black, font={\scriptsize \sffamily}, xshift=-4pt, yshift=-6pt] at (0.438595933733141, 0.31169898098285276) {8};
\node[text=black, font={\scriptsize \sffamily}, xshift=4pt, yshift=4pt] at (0.75, 0.6666666666666666) {9};
\end{tikzpicture} \\
    & \stackrel{0}{\celto} \begin{tikzpicture}[xscale=2.5, yscale=2, baseline={([yshift=-.5ex]current bounding box.center)}]
\path[fill=gray!10] (0, 0) rectangle (1, 1);
\draw[color=black, opacity=1] (0.5, 0.25) .. controls (0.5666666666666667, 0.25) and (0.6, 0.3333333333333333) .. (0.6, 0.5);
\draw[color=black, opacity=1] (0.75, 0.75) .. controls (0.6499999999999999, 0.75) and (0.6, 0.6666666666666666) .. (0.6, 0.5);
\draw[color=black, opacity=1] (0.5, 0.25) .. controls (0.43333333333333335, 0.25) and (0.4, 0.29166666666666663) .. (0.4, 0.375);
\draw[color=black, opacity=1] (0.25, 0.5) .. controls (0.35, 0.5) and (0.4, 0.4583333333333333) .. (0.4, 0.375);
\draw[color=black, opacity=1] (0.5, 0.25) .. controls (0.459063955822094, 0.25) and (0.438595933733141, 0.2029247452457132) .. (0.438595933733141, 0.10877423573713957);
\draw[color=black, opacity=1] (0.438595933733141, 0) .. controls (0.438595933733141, 0.0) and (0.438595933733141, 0.036258078579046525) .. (0.438595933733141, 0.10877423573713957);
\draw[color=black, opacity=1] (0.75, 0.75) .. controls (0.75, 0.75) and (0.75, 0.7916666666666667) .. (0.75, 0.875);
\draw[color=black, opacity=1] (0.75, 1) .. controls (0.75, 1.0) and (0.75, 0.9583333333333333) .. (0.75, 0.875);
\draw[color=black, opacity=1] (0.25, 0.5) .. controls (0.41666666666666663, 0.5) and (0.5, 0.5416666666666666) .. (0.5, 0.625);
\draw[color=black, opacity=1] (0.75, 0.75) .. controls (0.5833333333333333, 0.75) and (0.5, 0.7083333333333334) .. (0.5, 0.625);
\draw[color=black, opacity=1] (0.25, 0.5) .. controls (0.25, 0.5) and (0.25, 0.6111111111111112) .. (0.25, 0.8333333333333334);
\draw[color=black, opacity=1] (0.25, 1) .. controls (0.25, 1.0) and (0.25, 0.9444444444444444) .. (0.25, 0.8333333333333334);
\draw[color=black, opacity=1] (0.75, 0.75) .. controls (0.7833333333333333, 0.75) and (0.8, 0.5833333333333334) .. (0.8, 0.25);
\draw[color=black, opacity=1] (0.8, 0) .. controls (0.8, 0.0) and (0.8, 0.08333333333333333) .. (0.8, 0.25);
\draw[color=black, opacity=1] (0.25, 0.5) .. controls (0.45760271084457266, 0.5) and (0.561404066266859, 0.394297476976509) .. (0.561404066266859, 0.1828924309295271);
\draw[color=black, opacity=1] (0.561404066266859, 0) .. controls (0.561404066266859, 0.0) and (0.561404066266859, 0.0609641436431757) .. (0.561404066266859, 0.1828924309295271);
\draw[color=black, opacity=1] (0.25, 0.5) .. controls (0.21666666666666667, 0.5) and (0.2, 0.38888888888888884) .. (0.2, 0.16666666666666666);
\draw[color=black, opacity=1] (0.2, 0) .. controls (0.2, 0.0) and (0.2, 0.05555555555555555) .. (0.2, 0.16666666666666666);
\node[circle, fill=black, draw=black, inner sep=1pt] at (0.25, 0.5) {};
\node[circle, fill=black, draw=black, inner sep=1pt] at (0.5, 0.25) {};
\node[circle, fill=black, draw=black, inner sep=1pt] at (0.75, 0.75) {};
\node[text=black, font={\scriptsize \sffamily}, xshift=4pt, yshift=4pt] at (0.25, 0.5) {4};
\node[text=black, font={\scriptsize \sffamily}, xshift=0pt, yshift=-6pt] at (0.5, 0.25) {8};
\node[text=black, font={\scriptsize \sffamily}, xshift=4pt, yshift=4pt] at (0.75, 0.75) {9};
\end{tikzpicture} \stackrel{2}{\celto} \begin{tikzpicture}[xscale=2.5, yscale=2, baseline={([yshift=-.5ex]current bounding box.center)}]
\path[fill=gray!10] (0, 0) rectangle (1, 1);
\draw[color=black, opacity=1] (0.438595933733141, 0.31169898098285276) .. controls (0.5052626003998076, 0.31169898098285276) and (0.538595933733141, 0.36725453653840834) .. (0.538595933733141, 0.47836564764951944);
\draw[color=black, opacity=1] (0.75, 0.6666666666666666) .. controls (0.609063955822094, 0.6666666666666666) and (0.538595933733141, 0.6038996603276175) .. (0.538595933733141, 0.47836564764951944);
\draw[color=black, opacity=1] (0.438595933733141, 0.31169898098285276) .. controls (0.37192926706647433, 0.31169898098285276) and (0.338595933733141, 0.36725453653840834) .. (0.338595933733141, 0.47836564764951944);
\draw[color=black, opacity=1] (0.25, 0.6666666666666666) .. controls (0.309063955822094, 0.6666666666666666) and (0.338595933733141, 0.6038996603276175) .. (0.338595933733141, 0.47836564764951944);
\draw[color=black, opacity=1] (0.438595933733141, 0.31169898098285276) .. controls (0.438595933733141, 0.31169898098285276) and (0.438595933733141, 0.2561434254272972) .. (0.438595933733141, 0.14503231431618607);
\draw[color=black, opacity=1] (0.438595933733141, 0) .. controls (0.438595933733141, 0.0) and (0.438595933733141, 0.048344104772062024) .. (0.438595933733141, 0.14503231431618607);
\draw[color=black, opacity=1] (0.75, 0.6666666666666666) .. controls (0.75, 0.6666666666666666) and (0.75, 0.7222222222222222) .. (0.75, 0.8333333333333334);
\draw[color=black, opacity=1] (0.75, 1) .. controls (0.75, 1.0) and (0.75, 0.9444444444444444) .. (0.75, 0.8333333333333334);
\draw[color=black, opacity=1] (0.561404066266859, 0.35496768568381387) .. controls (0.6280707329335257, 0.35496768568381387) and (0.661404066266859, 0.41052324123936945) .. (0.661404066266859, 0.5216343523504806);
\draw[color=black, opacity=1] (0.75, 0.6666666666666666) .. controls (0.690936044177906, 0.6666666666666666) and (0.661404066266859, 0.6183225618946047) .. (0.661404066266859, 0.5216343523504806);
\draw[color=black, opacity=1] (0.25, 0.6666666666666666) .. controls (0.25, 0.6666666666666666) and (0.25, 0.7222222222222222) .. (0.25, 0.8333333333333334);
\draw[color=black, opacity=1] (0.25, 1) .. controls (0.25, 1.0) and (0.25, 0.9444444444444444) .. (0.25, 0.8333333333333334);
\draw[color=black, opacity=1] (0.561404066266859, 0.35496768568381387) .. controls (0.49473739960019236, 0.35496768568381387) and (0.46140406626685904, 0.41052324123936945) .. (0.46140406626685904, 0.5216343523504806);
\draw[color=black, opacity=1] (0.25, 0.6666666666666666) .. controls (0.390936044177906, 0.6666666666666666) and (0.46140406626685904, 0.6183225618946047) .. (0.46140406626685904, 0.5216343523504806);
\draw[color=black, opacity=1] (0.75, 0.6666666666666666) .. controls (0.7833333333333333, 0.6666666666666666) and (0.8, 0.5277777777777778) .. (0.8, 0.25);
\draw[color=black, opacity=1] (0.8, 0) .. controls (0.8, 0.0) and (0.8, 0.08333333333333333) .. (0.8, 0.25);
\draw[color=black, opacity=1] (0.561404066266859, 0.35496768568381387) .. controls (0.561404066266859, 0.35496768568381387) and (0.561404066266859, 0.2994121301282583) .. (0.561404066266859, 0.18830101901714724);
\draw[color=black, opacity=1] (0.561404066266859, 0) .. controls (0.561404066266859, 0.0) and (0.561404066266859, 0.06276700633904908) .. (0.561404066266859, 0.18830101901714724);
\draw[color=black, opacity=1] (0.25, 0.6666666666666666) .. controls (0.21666666666666667, 0.6666666666666666) and (0.2, 0.5277777777777778) .. (0.2, 0.25);
\draw[color=black, opacity=1] (0.2, 0) .. controls (0.2, 0.0) and (0.2, 0.08333333333333333) .. (0.2, 0.25);
\node[circle, fill=black, draw=black, inner sep=1pt] at (0.25, 0.6666666666666666) {};
\node[circle, fill=black, draw=black, inner sep=1pt] at (0.561404066266859, 0.35496768568381387) {};
\node[circle, fill=black, draw=black, inner sep=1pt] at (0.438595933733141, 0.31169898098285276) {};
\node[circle, fill=black, draw=black, inner sep=1pt] at (0.75, 0.6666666666666666) {};
\node[text=black, font={\scriptsize \sffamily}, xshift=4pt, yshift=4pt] at (0.25, 0.6666666666666666) {6};
\node[text=black, font={\scriptsize \sffamily}, xshift=4pt, yshift=-6pt] at (0.561404066266859, 0.35496768568381387) {7};
\node[text=black, font={\scriptsize \sffamily}, xshift=-4pt, yshift=-6pt] at (0.438595933733141, 0.31169898098285276) {8};
\node[text=black, font={\scriptsize \sffamily}, xshift=4pt, yshift=4pt] at (0.75, 0.6666666666666666) {9};
\end{tikzpicture}
\end{align*}
where nodes represent 3\nbd dimensional elements and incoming and outgoing wires represent their input and output faces. It has one other 3\nbd layering inducing the 3\nbd ordering
\begin{equation*}
    ((4, 0), (4, 2), (4, 1), (4, 3)).
\end{equation*}
These are the only two 3\nbd layerings of $U$.
Indeed, the application of the rewrite $(4, 1)$ creates a ``non-convexity'' in the input boundary of $(4, 0)$, in the form of a path $(3, 1) \to (3, 5) \to (3, 0)$ in $\flow{2}{U}$, which can also be spotted as the upward path $1 \to 5 \to 0$ in the second string diagram.
Before $(4, 0)$ can be applied, this needs to be resolved by the application of $(4, 3)$, which removes the non-convexity.
Similarly, if $(4, 0)$ is applied first, it creates a non-convexity in the input boundary of $(4, 1)$, a path $(3, 2) \to (3, 4) \to (3, 3)$ in $\flow{2}{U}$, which needs to be resolved by the application of $(4, 2)$.

Now we have examples of all the following.
\begin{enumerate}
    \item \emph{A 3\nbd ordering that is not induced by a 3\nbd layering.}

The graph $\flow{3}{U}$ is simply
\[\begin{tikzcd}[sep=tiny]
	{(4, 0)} && {(4, 2)} \\
	{(4, 1)} && {(4, 3)}
	\arrow[from=1-1, to=1-3]
	\arrow[from=2-1, to=2-3]
\end{tikzcd}\]
so $U$ admits four other 3\nbd orderings which do not determine 3\nbd layerings; for example, $((4, 0), (4, 1), (4, 2), (4, 3))$.
What is more, there is no extension of $\flow{3}{U}$, and more in general no graph whose vertex set is $\grade{4}{U}$, whose topological sorts correspond to the 3\nbd layerings of $U$.

\item \emph{A regular molecule that is not frame-acyclic.}

We can deduce that $U$ is not frame-acyclic using Theorem \ref{thm:frame_acyclicity_equivalent_conditions}.
More directly, $V \eqdef \clset{{(4, 0), (4, 3)}}$ is a submolecule of $U$ such that $\frdim{V} = 2$, but $\maxflow{2}{V}$ contains a cycle $(4, 3) \to (4, 0) \to (4, 3)$, since 
\begin{align*}
    (2, 5) & \in \faces{2}{+}(4, 3) \cap \faces{2}{-}(4, 0), \\
    (2, 7) & \in \faces{2}{+}(4, 0) \cap \faces{2}{-}(4, 3).
\end{align*}

\item \emph{An inclusion of a round regular molecule that is not a submolecule inclusion.}

We have that $\bound{}{-}(4, 0) = \clset{(3, 0), (3, 1)}$ is a round 3\nbd dimensional regular molecule, included in the 3\nbd dimensional regular molecule $W \eqdef \clset{(3, 0), (3, 1), (3, 5)}$.
However, it is \emph{not} a submolecule of $W$, due to the presence of the path $(3, 1) \to (3, 5) \to (3, 0)$ in $\maxflow{2}{W}$.
\end{enumerate}
\end{exm}
\section*{Conclusions and outlook}

We have taken the first steps into the study of machines based on higher-dimensional rewriting in all dimensions.
We have presented algorithms by which they could be simulated in standard models of computation, and shown that this requires only low-degree polynomial time overhead in dimension $\leq 3$.

Feasibility in dimension $> 3$ is the most obvious open question.
We hope that a deeper understanding of cases like Example \ref{exm:no_topological_sorts} will lead either to an improved algorithm, or to a proof of $\fun{NP}$\nbd completeness.
The way in which 4\nbd dimensional rewrites can introduce obstructions to ``disjoint'' rewrites, in a non-local way, may be a hint that the latter is more likely.
In either case, we are actively working on the problem.

Beyond the more immediate questions, we hope to have laid the groundwork for an approach to complexity theory based on higher-dimensional rewriting, that leverages its unique characteristics as described in the introduction.
For instance, we believe that the coexistence of higher algebraic structures and rewrite systems within the same category, as made possible by the theory of diagrammatic sets or their variants, may lead to a unified and compositional understanding of interpretations of rewrite systems, such as polynomial and matrix interpretations, which are one of the key techniques in implicit computational complexity.
We plan to develop various aspects of this programme in future work.

\subsection*{Acknowledgements}
This work was supported by the ESF funded Estonian IT Academy research measure (project 2014-2020.4.05.19-0001) and by the Estonian Research Council grant PSG764.
We thank the anonymous referees for their helpful comments on an earlier draft.


\bibliographystyle{IEEEtran}
\bibliography{IEEEabrv,main}

\clearpage 

\appendix[Additional proofs]

\subsection{Proofs for Section \ref{sec:structures}}

\begin{dfn}
Most of the proofs are taken from a forthcoming monograph on the combinatorics of pasting diagrams \cite{hadzihasanovicpasting}.
To avoid making this appendix longer than it already is, we cite results of \cite{steiner1993algebra, hadzihasanovic2020diagrammatic, hadzihasanovic2021smash} whenever possible, even though they use slightly different definitions; all the results used have been independently reproved.
\end{dfn}

\begin{proof}[Proof of Lemma \ref{lem:properties_of_inclusions}]
This is a combination of \cite[Lemma 1.9 and Lemma 1.11]{hadzihasanovic2020diagrammatic}.
\end{proof}

\begin{proof}[Proof of Proposition \ref{prop:ogpos_limits_and_colimits}]
By \cite[Lemma 1.9]{hadzihasanovic2020diagrammatic}, there is a forgetful functor $\fun{U}$ from $\ogpos$ to the category $\poscat$ of posets and order-preserving maps.
All these limits and colimits exist in $\poscat$, so it suffices to prove that they can be lifted to $\ogpos$.

The terminal poset is the poset with a single element, and the initial poset is the empty poset.
Both of them admit a unique orientation.
Let $P$ be an oriented graded poset.
Both the unique map from $\fun{U}P$ to the terminal poset and the unique map from the initial poset trivially preserve boundaries, so they lift to maps of oriented graded posets.

Let $\imath_1\colon Q \incl P_1$ and $\imath_2\colon Q \incl P_2$ be inclusions of oriented graded posets.
Computing their pushout in $\poscat$ determines two order-preserving maps 
\[
    j_1\colon \fun{U}P_1 \to \fun{U}P_1 \cup \fun{U}P_2, \quad \quad j_2\colon \fun{U}P_2 \to \fun{U}P_1 \cup \fun{U}P_2.
\]
Since $\fun{U}\imath_1$ and $\fun{U}\imath_2$ are closed embeddings, it is an exercise to show that $j_1$ and $j_2$ are also closed embeddings, and deduce that $\fun{U}P_1 \cup \fun{U}P_2$ is a graded poset. 
Since $j_1$ and $j_2$ preserve the covering relation and are jointly surjective, we can put a unique orientation on $\fun{U}P_1 \cup \fun{U}P_2$ in such a way that $j_1$ and $j_2$ both preserve orientations; overlaps are resolved by the fact that $(\fun{U}\imath_1);j_1 = (\fun{U}\imath_2);j_2$ and $\imath_1$ and $\imath_2$ preserve orientations.
This choice of orientation determines a unique lift of the pushout to $\ogpos$.
\end{proof}

\begin{lem} \label{lem:maximal_in_boundary}
Let $U$ be a closed subset of an oriented graded poset, $n \in \mathbb{N}$, and $\alpha \in \set{ +, - }$.
Then
\begin{enumerate}
    \item $\grade{n}{(\bound{n}{\alpha}U)} = \faces{n}{\alpha}U$,
    \item $\grade{k}{( \maxel{(\bound{n}{\alpha}U)} )} = \grade{k}{(\maxel{U})}$ for all $k < n$.
\end{enumerate}
\end{lem}
\begin{proof}
Let $x \in \bound{n}{\alpha}U$.
Then by definition there exists $y$ such that $x \leq y$ and either $y \in \faces{n}{\alpha}U$ or $y \in \grade{k}{(\maxel{U})}$ for some $k < n$.
If $x$ is maximal, necessarily $x = y$, and we obtain one inclusion.
The converse inclusions are evident.
\end{proof}

\begin{lem} \label{lem:maximal_vs_faces}
Let $U$ be a closed subset in an oriented graded poset, $n \in \mathbb{N}$, and $\alpha \in \set{ +, - }$.
Then 
\begin{enumerate}
    \item $\grade{n}{(\maxel{U})} = \faces{n}{+} U \cap \faces{n}{-} U$,
    \item if $n = \dim{U}$ then $\grade{n}{(\maxel{U})} = \faces{n}{\alpha} U = \grade{n}{U}$.
\end{enumerate}
\end{lem}
\begin{proof}
Let $x \in U$, $\dim{x} = n$.
Then $x$ is maximal if and only if it has no cofaces in $U$, if and only if $\cofaces{}{-\alpha}x \cap U = \cofaces{}{\alpha}x \cap U = \varnothing$, if and only if $x \in \faces{n}{+} U \cap \faces{n}{-} U$.
If $n = \dim{U}$, then every element of $\grade{n}{U}$ is maximal in $U$, so 
\[\grade{n}{U} = \grade{n}{(\maxel{U})} \subseteq \faces{n}{\alpha} U \subseteq \grade{n}{U}\]
using the first part of the proof, and we conclude that they are all equal.
\end{proof}

\begin{lem} \label{lem:faces_of_union}
Let $U, V$ be closed subsets of an oriented graded poset, $n \in \mathbb{N}$, and $\alpha \in \set{ +, - }$.
Then
\begin{enumerate}
    \item $\maxel{(U \cup V)} = (\maxel{U} \cap \maxel{V}) + (\maxel{U} \setminus V) +$ $(\maxel{V} \setminus U)$,
    \item $\faces{n}{\alpha}(U \cup V) = (\faces{n}{\alpha}U \cap \faces{n}{\alpha}V) + (\faces{n}{\alpha}U \setminus V) + (\faces{n}{\alpha}V \setminus U)$.
\end{enumerate}
\end{lem}
\begin{proof}
Follows straightforwardly from the definitions using the decomposition $U \cup V = (U \cap V) + (U \setminus V) + (V \setminus U)$.
\end{proof}

\begin{proof}[Proof of Lemma \ref{lem:boundaries_of_rewrite}]
Identifying $U$ and $V$ with their isomorphic images, we will prove that $\bound{}{-}(U \celto^\varphi V) = U$ and $\bound{}{+}(U \celto^\varphi V) = V$.
Let $n \eqdef \dim{U} = \dim{V}$.
By construction, we have $\faces{n}{-}(U \celto^\varphi V) = \grade{n}{U}$ and $\faces{n}{+}(U \celto^\varphi V) = \grade{n}{V}$.

For all $k < n$, the set $\grade{k}{ ( \maxel{(U \celto^\varphi V)} ) }$ is equal to $\grade{k}{ ( \maxel{(U \cup V)} ) }$.
We claim that this is equal to both $\grade{k}{ ( \maxel{U} ) }$ and $\grade{k}{ ( \maxel{V} ) }$.
For $k < n-1$,
\begin{equation*}
    \grade{k}{ ( \maxel{U} ) } = \grade{k}{ ( \maxel{\bound{}{\alpha}U} ) } = \grade{k}{ ( \maxel{\bound{}{\alpha}V} ) } = \grade{k}{ ( \maxel{V} ) }
\end{equation*}
by Lemma \ref{lem:maximal_in_boundary}.
For $k = n-1$, by Lemma \ref{lem:maximal_vs_faces}
\begin{equation*}
    \grade{n-1}{ ( \maxel{U} ) } \!=\! \faces{}{-}U \cap \faces{}{+}U \!=\! \faces{}{-}V \cap \faces{}{+}V \!=\! \grade{n-1}{ ( \maxel{V} ) }.
\end{equation*}
We then conclude by Lemma \ref{lem:faces_of_union}.
\end{proof}

\begin{proof}[Proof of Lemma \ref{lem:round_is_pure}]
We will prove the contrapositive.
Suppose that $U$ is not pure.
Then there exists a maximal element $x$ in $U$ with $k \eqdef \dim{x} < \dim{U}$.

Since $\grade{k}{(\maxel{U})} = \faces{k}{-}{U} \cap \faces{k}{+}{U}$, we have $x \in \bound{k}{-}U \cap \bound{k}{+}U$.
Then $\bound{k}{-}U \cap \bound{k}{+}U$ is $k$\nbd dimensional and cannot be equal to $\bound{k-1}{}U$, which is $(k-1)$\nbd dimensional.
It follows that $U$ is not round.
\end{proof}

\begin{proof}[Proof of Proposition \ref{prop:molecule_properties}]
The first point follows from \cite[Proposition 1.38]{hadzihasanovic2020diagrammatic}, the third and fourth point from \cite[Proposition 1.23]{hadzihasanovic2020diagrammatic}.
The second point follows from the first and the third for the pasting construction, and \cite[Lemma 2.2]{hadzihasanovic2020diagrammatic} for the rewrite construction.
The fifth point is an immediate consequence of the definition of roundness combined with globularity.
\end{proof}

\begin{proof}[Proof of Propositions \ref{prop:associativity_of_pasting}, \ref{prop:unitality_of_pasting}, \ref{prop:interchange_of_pasting}]
These all follow from \cite[Proposition 1.23]{hadzihasanovic2020diagrammatic} in conjunction with uniqueness of isomorphisms of regular molecules.
\end{proof}

\begin{proof}[Proof of Proposition \ref{prop:atom_properties}]
We prove the first point by induction.
If $U$ was produced by (\textit{Point}), then $U$ is the terminal oriented graded poset, which trivially has a greatest element.
If $U$ was produced by (\textit{Paste}), then $U$ splits into a union $V \cup W$, where $V \cap W = \bound{k}{+}V = \bound{k}{-}W$ and $k < \max\set{ \dim{V}, \dim{W} }$.
Then there exist elements $x_1 \in V$ and $x_2 \in W$ such that
\begin{enumerate}
    \item $x_1$ is maximal in $V$ and $x_2$ is maximal in $W$,
    \item $\dim{x_1} > k$ and $\dim{x_2} > k$.
\end{enumerate}
We have $\dim{(V \cap W)} \leq k$, so neither $x_1$ nor $x_2$ are contained in $V \cap W$.
It follows that $x_1$ and $x_2$ are distinct maximal elements of $U$, so $U$ does not have a greatest element.
If $U$ was produced by (\textit{Atom}), then $U$ splits into $(U_{-} \cup U_{+}) + \set{ \top }$, where $U_{-}$ and $U_{+}$ are round regular molecules of dimension $n$, and $\faces{}{\alpha}\top = \grade{n}{(U_\alpha)}$ for each $\alpha \in \set{ +, - }$.
By Lemma \ref{lem:round_is_pure}, we have $U_\alpha = \clos{\grade{n}{(U_\alpha)}}$, so $U_\alpha = \bound{}{\alpha}\top \subseteq \clset{ \top }$.
It follows that all elements of $U$ are in the closure of $x$, that is, $x$ is the greatest element of $U$.

We prove the second point also by induction.
If $U$ was produced by (\textit{Point}), then $x$ must be the unique element of $U$ whose closure is $U$ itself.
If $U$ was produced by (\textit{Paste}), it splits into $V \cup W$, and $x \in V$ or $x \in W$; the inductive hypothesis applies.
If $U$ was produced by (\textit{Atom}), it is equal to $(V \cup W) + \set{ \top }$, and either $x \in V$ or $x \in W$, in which case the inductive hypothesis applies, or $x = \top$, and $\clset{ x } = U$ is an atom by definition.

For the third point, if $U$ was produced by (\textit{Point}), it is trivially round.
If it was produced by (\textit{Atom}), it is of the form $V \celto W$ where $\bound{}{-}U$ is isomorphic to $V$ and $\bound{}{+}U$ to $W$, and by definition of the rewrite construction their intersection is uniquely isomorphic to $\bound{}{}V$ and $\bound{}{}W$; we conclude by globularity.
If $\dim{U} > 0$, then $U$ was produced by (\textit{Atom}), so it is of the form $V \celto W$, with $\bound{}{-}U$ isomorphic to $V$ and $\bound{}{+}U$ to $W$.
By uniqueness of these isomorphisms, $U$ is isomorphic to $\bound{}{-}U \celto \bound{}{+}U$.
\end{proof}


\subsection{Proofs for Section \ref{sec:matching}}

\begin{proof}[Proof of Lemma \ref{lem:submolecule_properties}]
For the first point, $\bound{n}{\alpha} U$ is a regular molecule by Proposition \ref{prop:molecule_properties}.
By Proposition \ref{prop:unitality_of_pasting}, the pastings $U \cp{n} \bound{n}{+}U$ and $\bound{n}{-}U \cp{n} U$ are both defined and uniquely isomorphic to $U$.
The inclusion of $\bound{n}{-}U$ into $U$ factors as the inclusion $\bound{n}{-}U \incl (\bound{n}{-}U \cp{n} U)$ followed by an isomorphism, and the inclusion of $\bound{n}{+}U$ factors as the inclusion $\bound{n}{+}U \incl (U \cp{n} \bound{n}{+}U)$ followed by an isomorphism.

For the second point, $\clset{ x }$ is an atom by Proposition \ref{prop:atom_properties}.
We proceed by induction on the construction of $U$.
If $U$ was produced by (\textit{Point}), then $x$ must be the unique element of $U$, so $\clset{ x } = U$.
If $U$ was produced by (\textit{Paste}), it splits into $V \cup W$ with $V, W \submol U$, and $x \in V$ or $x \in W$.
By the inductive hypothesis, $\clset{ x } \submol V$ or $\clset{ x } \submol W$.
If $U$ was produced by (\textit{Atom}), it splits into $(V \cup W) + \set{ \top }$ with $V, W \submol U$, and either $x \in V$ or $x \in W$, in which case the inductive hypothesis applies since $V, W \submol U$ by Lemma \ref{lem:submolecule_properties}, or $x = \top$, and $\clset{ x } = U$.
\end{proof}

\begin{lem} \label{lem:submolecule_rewrite}
Let $V$ be a regular molecule, $n < \dim{V}$, $\alpha \in \set{ +, - }$.
Consider a pushout diagram of the form
\[\begin{tikzcd}
	{\bound{n}{\alpha} V} && V \\
	U && {V \cup U}
	\arrow[hook, from=1-1, to=1-3]
	\arrow["\imath", hook', from=1-1, to=2-1]
	\arrow["{j_U}", hook, from=2-1, to=2-3]
	\arrow["{j_V}", hook', from=1-3, to=2-3]
	\arrow["\lrcorner"{anchor=center, pos=0.125, rotate=180}, draw=none, from=2-3, to=1-1]
\end{tikzcd}\]
in $\ogpos$.
If $\dim{U} = n$ and $\imath$ is a submolecule inclusion, then
\begin{enumerate}
    \item $V \cup U$ is a regular molecule,
    \item $j_U$ maps $U$ onto $\bound{n}{\alpha}(V \cup U)$,
    \item $j_V(V) \submol V \cup U$ and $j_V(\bound{n}{-\alpha}V) \submol \bound{n}{-\alpha}(V \cup U)$.
\end{enumerate}
\end{lem}
\begin{proof}
By induction on the construction of $\imath$.
If $\imath$ is an isomorphism, then $j_V$ is also an isomorphism, and all the statements are trivially satisfied.

Suppose $U$ is of the form $\bound{n}{\alpha}V \cp{k} W$ for some $k \in \mathbb{N}$, and $\imath$ is the inclusion of $\bound{n}{\alpha}V$ into the pasting.
Since $\dim{U} = n$, necessarily $\dim{W} \leq n$, so $\bound{n}{\alpha}W = W$.
If $k \geq n$, then also $k \geq \dim{W}$, and in this case $\imath$ and $j_V$ are again isomorphisms.
Suppose that $k < n$.
Identifying $V$ with its image through $j_V$, $V \cup U$ splits into $V \cup W$ with
\begin{equation*}
    V \cap W = \bound{n}{\alpha}V \cap W = \bound{k}{-}W = \bound{k}{+}(\bound{n}{\alpha}V) = \bound{k}{+} V
\end{equation*}
where the final equation uses globularity of $V$.
This exhibits $V \cup U$ as $V \cp{k} W$, with $j_V$ the inclusion of $V$ into the pasting, and $j_U$ maps $\bound{n}{\alpha}V \cp{k} W$ onto $\bound{n}{\alpha}(V \cp{k} W)$ by \cite[Proposition 1.23]{hadzihasanovic2020diagrammatic} and the axioms of strict $\omega$\nbd categories.
Similarly, $\bound{n}{-\alpha}V \submol \bound{n}{-\alpha}(V \cp{k} W)$.
The case where $U$ is of the form $W \cp{k} \bound{n}{\alpha}V$ is dual.

By the pasting law for pushout squares, if the statement is true of two submolecule inclusions, it is also true of their composite.
\end{proof}

\begin{lem} \label{lem:substitution_direct_pushout}
Let $U$ be an oriented graded poset, $V, W$ round regular molecules, and $\imath\colon V \incl U$ an inclusion such that $\subs{U}{W}{\imath(V)}$ is defined.
Let $\varphi\colon \bound{}{}V \incliso \bound{}{}W$ be the isomorphism used in the construction of $V \celto W$.
Then $\subs{U}{W}{\imath(V)}$ can be constructed as the pushout
\begin{equation} \label{eq:substitution_direct_pushout}
\begin{tikzcd}
	\bound{}{}V & \bound{}{}W & W \\
	U \setminus (\imath(V) \setminus \imath(\bound{}{} V)) && \subs{U}{W}{\imath(V)}.
	\arrow["\varphi", hook, from=1-1, to=1-2]
	\arrow[hook, from=1-2, to=1-3]
	\arrow[hook', from=1-1, to=2-1]
	\arrow[hook, from=2-1, to=2-3]
	\arrow["j", hook', from=1-3, to=2-3]
	\arrow["\lrcorner"{anchor=center, pos=0.125, rotate=180}, draw=none, from=2-3, to=1-1]
\end{tikzcd}
\end{equation}
\end{lem}
\begin{proof}
We can safely identify $V$ with its image through $\imath$, and treat it as a closed subset of $U$.
First of all, observe that $U \setminus (V \setminus \bound{}{}V)$ is the complement of the complement of a closed subset in a closed subset, so it is closed in $U$, and well-defined as an oriented graded poset.

Let $n \eqdef \dim U$, so $\dim (U \cup (V \celto W))$ and $\dim (V \celto W)$ are both equal to $n + 1$.
Then
\begin{equation*}
    \faces{n}{+}(V \celto W) = \grade{n}{W}, \quad \quad \faces{n}{+}U = \grade{n}{U}
\end{equation*}
and since $U \cap (V \celto W) = V$, by Lemma \ref{lem:faces_of_union}
\begin{equation*}
    \faces{n}{+}(U \cup (V \celto W)) = \grade{n}{W} + (\grade{n}{U} \setminus \grade{n}{V}) = \grade{n}{W} + \grade{n}{(U \setminus (V \setminus \bound{}{}V))},
\end{equation*}
while for all $k < n$
\begin{align*}
    \grade{k}{(\maxel{(U \cup (V \celto W))})} & = \grade{k}{(\maxel{U})} = \\
    & = \grade{k}{(\maxel{(U \setminus (V \setminus \bound{}{}V)}))}
\end{align*}
because both $V$ and $V \celto W$ are round, hence pure, and do not contain any maximal elements of dimension $k$.

It follows that $\bound{}{+}(U \cup (V \celto W))$ is the union of $W$ and $(U \setminus (V \setminus \bound{}{}V))$, with intersection $\bound{}{}W = \bound{}{}V$.
\end{proof}

\begin{lem} \label{lem:revert_substitution}
Let $U$ be an oriented graded poset, $V, W$ round regular molecules, and $\imath\colon V \incl U$ an inclusion such that $\subs{U}{W}{\imath(V)}$ is defined.
Let $j\colon W \incl \subs{U}{W}{\imath(V)}$ be the right side of (\ref{eq:substitution_direct_pushout}).
Then $\subs{(\subs{U}{W}{\imath(V)})}{V}{j(W)}$ is defined and isomorphic to $U$.
\end{lem}
\begin{proof}
Since $W \celto V$ is defined whenever $V \celto W$ is defined, it follows that $\subs{(\subs{U}{W}{\imath(V)})}{V}{j(W)}$ is defined.
The isomorphism with $U$ is straightforward algebra of closed subsets using Lemma \ref{lem:substitution_direct_pushout} twice.
\end{proof}

\begin{proof}[Proof of Proposition \ref{prop:round_submolecule_substitution}]
If $\imath$ is a submolecule inclusion, by Lemma \ref{lem:submolecule_rewrite} $U \cup (V \celto W)$ and its output boundary $U[W/\imath(V)]$ are regular molecules, and the inclusion of $W$ into $U[W/\imath(V)]$ is a submolecule inclusion.

If $V$ is a round regular molecule, then $\compos{V}$ is an atom, which is round by Proposition \ref{prop:atom_properties}, and has boundaries isomorphic to those of $V$ by Lemma \ref{lem:boundaries_of_rewrite}.
Then $V \celto \compos{V}$ is defined, so the fourth condition is a special case of the third one.

Finally, suppose $\subs{U}{\compos{V}}{\imath(V)}$ is a regular molecule.
By Lemma \ref{lem:submolecule_properties}, since $\compos{V}$ is an atom, its inclusion $j$ into $\subs{U}{\compos{V}}{\imath(V)}$ is a submolecule inclusion.

Using Lemma \ref{lem:submolecule_rewrite} as in the first part, we deduce that $\subs{(\subs{U}{\compos{V}}{\imath(V)})}{V}{j(\compos{V})}$ is a regular molecule, and the inclusion of $V$ into it is a submolecule inclusion.
By Lemma \ref{lem:revert_substitution}, $\subs{(\subs{U}{\compos{V}}{\imath(V)})}{V}{j(\compos{V})}$ is isomorphic to $U$, and $\imath$ factors as this submolecule inclusion followed by an isomorphism.
\end{proof}

\begin{proof}[Proof of Lemma \ref{lem:codimension_1_elements}]
This is \cite[Lemma 1.16]{hadzihasanovic2020diagrammatic}.
\end{proof}

\begin{proof}[Proof of Lemma \ref{lem:flow_under_inclusion}]
It follows from Lemma \ref{lem:properties_of_inclusions} that, for all $x \in V$ and $\alpha \in \set{ +,- }$, the set $\faces{k}{\alpha}x$ is isomorphic to $\faces{k}{\alpha}\imath(x)$.
It follows that, for all $x, y \in \bigcup_{i > k} \grade{i}{V}$, there is an edge between $x$ and $y$ in $\flow{k}{V}$ if and only if there is an edge between $\imath(x)$ and $\imath(y)$ in $\flow{k}{U}$.
\end{proof}

\begin{dfn}
The following proof uses results proved in the following sections; none of their proofs use it, so there is no circularity.
\end{dfn}

\begin{proof}[Proof of Proposition \ref{prop:round_molecule_connected_flowgraph}]
First of all, if $U$ is round, then it is pure, so the vertices of $\flow{n-1}{U}$ are the elements of $\grade{n}{U}$.
If $U$ is an atom, then $\flow{n-1}{U}$ consists of a single vertex and no edges, so it is trivially connected.
In particular this is true when $n = 0$ since $U$ is then the point, so we can proceed by induction on $n$.

Suppose $n > 0$ and $\size{\grade{n}{U}} > 1$, which by Lemma \ref{lem:layering_dimension_smaller_than_dimension} implies $\lydim{U} = n - 1$.
Assume by way of contradiction that $\flow{n-1}{U}$ is not connected.
Then there is a bipartition $\grade{n}{U} = A + B$ such that there are no edges in $\flow{n-1}{U}$ between vertices in $A$ and vertices in $B$.
By Lemma \ref{lem:codimension_1_elements}, no element of $\grade{n-1}{U}$ can be covered by two elements with the same orientation, so this implies that $\dim{(\clos{A} \cap \clos{B})} < n-1$.
Let 
\begin{align*}
    A' & \eqdef \set{x \in \faces{}{-}U \mid \cofaces{}{-}x \subseteq A}, \\
    B' & \eqdef \set{x \in \faces{}{-}U \mid \cofaces{}{-}x \subseteq B}.
\end{align*}
Then $A' + B'$ is a bipartition of $\faces{}{-}U$.
By Proposition \ref{prop:molecule_properties}, $\bound{}{-}U$ is round, so by the inductive hypothesis $\flow{n-2}{(\bound{}{-}U)}$ is connected.
It follows that there exist $\alpha \in \set{+, -}$, $x \in A'$, $y \in B'$, and $z \in \grade{n-2}{U}$ such that $z \in \faces{}{\alpha}x \cap \faces{}{-\alpha}y$.
Then $z$ has two distinct cofaces in $\bound{}{-}U$, so by Lemma \ref{lem:codimension_1_elements} $z \notin \bound{}{}(\bound{}{-}U) = \bound{n-2}{}U$.
We claim that $z \in \bound{}{+}U$, contradicting the roundness of $U$.

By Theorem \ref{thm:molecules_admit_layerings}, there exists an $(n-1)$\nbd layering $(\order{i}{U})_{i=1}^m$ of $U$; we will identify the $\order{i}{U}$ with their isomorphic images in $U$.
Let $V_0 \eqdef \bound{}{-}U$ and $V_i \eqdef \bound{}{+}\order{i}{U}$ for each $i \in \set{1, \ldots, m}$.
We will prove that, for all $i \in \set{0, \ldots, m}$,
\begin{enumerate}
    \item $z \in V_i$,
    \item there exist elements $x_i \in \clos{A}$ and $y_i \in \clos{B}$ such that $\cofaces{}{\alpha} z \cap V_i = \set{x_i}$ and $\cofaces{}{-\alpha} z \cap V_i = \set{y_i}$.
\end{enumerate}
For $i = 0$, we have already established this with $x_0 \eqdef x$, $y_0 \eqdef y$.
Let $i \geq 0$, and assume this holds for $i - 1$.
By Lemma \ref{lem:layering_basic_properties}, there is a single $n$\nbd dimensional element $\order{i}{x}$ in $\order{i}{U}$, and by Lemma \ref{lem:boundary_move}
\begin{equation*}
    V_i = \subs{\bound{}{-}\order{i}{U}}{\bound{}{+}\order{i}{x}}{\bound{}{-}\order{i}{x}} = \subs{V_{i-1}}{\bound{}{+}\order{i}{x}}{\bound{}{-}\order{i}{x}}.
\end{equation*}
Suppose $\order{i}{x} \in A$.
Then $y_{i-1} \notin \clset{\order{i}{x}}$, so $y_{i-1} \in V_i$, and we let $y_i \eqdef y_{i-1}$.
If $x_{i-1} \notin \clset{\order{i}{x}}$ then also $x_{i-1} \in V_i$, and we let $x_i \eqdef x_{i-1}$.
Otherwise, $x_{i-1}$ is the only coface of $z$ in $\bound{}{-}\order{i}{x}$, so by Lemma \ref{lem:codimension_1_elements} we have $z \in \bound{}{\alpha}(\bound{}{-}\order{i}{x}) = \bound{}{\alpha}(\bound{}{+}\order{i}{x})$.
It follows that $z \in V_i$ and there exists a unique $x_i$ such that $\cofaces{}{\alpha}z \cap \bound{}{+}\order{i}{x} = \set{x_i}$.
The case $\order{i}{x} \in B$ is analogous.

Since $V_m = \bound{}{+}U$, we have proved that $z \in \bound{}{+}U$, a contradiction.
\end{proof}


\subsection{Proofs for Section \ref{sec:rewritable}}

\begin{lem} \label{lem:layering_intersections}
Let $U$ be a regular molecule, $-1 \leq k < \dim{U}$, and $(\order{i}{U})_{i=1}^m$ a $k$\nbd layering of $U$.
For all $i < j \in \set{1, \ldots, m}$,
\begin{equation*}
    \order{i}{U} \cap \order{j}{U} = \bound{k}{+}\order{i}{U} \cap \bound{k}{-}\order{j}{U}.
\end{equation*}
\end{lem}
\begin{proof}
Let $i < j \in \set{1, \ldots, m}$, and
\begin{align*}
    V & \eqdef \order{1}{U} \cp{k} \ldots \cp{k} \order{i}{U}, \\
    W & \eqdef \bound{k}{+}\order{i}{U} \cp{k} \order{i+1}{U} \cp{k} \ldots \cp{k} \order{j-1}{U}, \\
    Z & \eqdef \order{j}{U} \cp{k} \ldots \cp{k} \order{m}{U}.
\end{align*}
Then $U$ splits into $V \cup (W \cp{k} Z)$ along the $k$\nbd boundary, so
\begin{equation*}
    \bound{k}{+}\order{i}{U} = \bound{k}{+}V = \bound{k}{-}(W \cp{k} Z) = V \cap (W \cp{k} Z).
\end{equation*}
Since $\order{i}{U} \subseteq V$ and $\order{j}{U} \subseteq (W \cp{k} Z)$, it follows that $\order{i}{U} \cap \order{j}{U} \subseteq \bound{k}{+}\order{i}{U}$.
Dually, from the fact that $U$ splits into $(V \cp{k} W) \cup Z$ along the $k$\nbd boundary, we derive $\order{i}{U} \cap \order{j}{U} \subseteq \bound{k}{-}\order{j}{U}$.
\end{proof}

\begin{proof}[Proof of Lemma \ref{lem:layering_basic_properties}]
For the first point, since $\dim{\order{i}{U}} > k$ each $\order{i}{U}$ contains at least one maximal element of dimension $> k$, and because
\begin{equation*}
    \dim{( \order{i}{U} \cap \order{j}{U} )} = \dim{(\bound{k}{+}\order{i}{U} \cap \bound{k}{-}\order{j}{U})} \leq k
\end{equation*}
by Lemma \ref{lem:layering_intersections}, no such maximal element is contained in two of them.
Since there are exactly $m$ maximal elements of dimension $> k$, it follows that each $\order{i}{U}$ contains exactly one of them.

For the second point, for all $i \in \set{1, \ldots, m}$, let
\begin{align*}
    \order{i}{V} \eqdef \bound{\ell}{+}\order{1}{U} & \cp{k} \ldots \cp{k} \bound{\ell}{+}\order{i-1}U \cp{k} \\
    & \cp{k} \order{i}{U} \cp{k} \bound{\ell}{-}\order{i+1}{U} \cp{k} \ldots \cp{k} \bound{\ell}{-}\order{m}U.
\end{align*}
By repeated applications of Proposition \ref{prop:interchange_of_pasting} followed by Proposition \ref{prop:unitality_of_pasting}, $U$ is isomorphic to
\begin{equation*}
    \order{1}{V} \cp{\ell} \ldots \cp{\ell} \order{m}{V}.
\end{equation*}
Restricting to the subsequence of $(\order{i}{V})_{i=1}^m$ on those $i$ such that $\dim{\order{i}{V}} > \ell$, which does not change the result by Proposition \ref{prop:unitality_of_pasting}, we obtain an $\ell$\nbd layering of $U$.
\end{proof}

\begin{lem} \label{lem:layering_dimension_smaller_than_dimension}
Let $U$ be a regular molecule, $n \eqdef \dim{U}$.
Then
\begin{enumerate}
    \item $\lydim{U} \leq n - 1$,
    \item $\lydim{U} = n - 1$ if and only if $\size{\grade{n}{U}} > 1$.
\end{enumerate}
\end{lem}
\begin{proof}
We have 
\begin{equation*}
    \size{\bigcup_{i > n} \grade{i}{(\maxel{U})}} = \size{\varnothing} = 0,
\end{equation*}
so $\lydim{U} \leq n - 1$, with equality if and only if
\begin{equation*}
    \size{\bigcup_{i > n-1} \grade{i}{(\maxel{U})}} = \size{\grade{n}{(\maxel{U})}} = \size{\grade{n}{U}} > 1,
\end{equation*}
where we used Lemma \ref{lem:maximal_vs_faces}.
\end{proof}

\begin{lem} \label{lem:layering_dimension_atom}
Let $U$ be a regular molecule.
Then $\lydim{U}$ is $-1$ if and only if $U$ is an atom.
\end{lem}
\begin{proof}
Suppose $\lydim{U} = -1$.
If $\size{\bigcup_{i > 0} \grade{i}{(\maxel{U})}} = 0$, then $\dim{U} = 0$ and $U$ is the point.

Otherwise, $1 = \size{\bigcup_{i > 0} \grade{i}{(\maxel{U})}} = \size{\maxel{U}}$ because a regular molecule which is not 0\nbd dimensional cannot have a 0\nbd dimensional maximal element.
In either case, $U$ has a greatest element.
Conversely, if $U$ has a greatest element, $\size{\bigcup_{i > 0} \grade{i}{(\maxel{U})}} \leq \size{\maxel{U}} = 1$.
\end{proof}

\begin{lem} \label{lem:layering_dimensions_pasting}
Let $U, V$ be regular molecules, and suppose $U \cp{k} V$ is defined for some $k < \min \set{ \dim{U}, \dim{V} }$.
Then
\begin{equation*}
    \lydim{(U \cp{k} V)} \geq \max \set{ \lydim{U}, \lydim{V}, k }.
\end{equation*}
\end{lem}
\begin{proof}
Identifying $U$ and $V$ with their isomorphic images, $U \cp{k} V$ splits into $U \cup V$ with $\dim{(U \cap V)} = \dim{\bound{k}{+}U} = k$.
By Lemma \ref{lem:faces_of_union}, for all $i > k$,
\begin{equation*}
    \grade{i}{ ( \maxel{(U \cp{k} V)} ) } = \grade{i}{(\maxel{U})} + \grade{i}{(\maxel{V})},
\end{equation*}
and since $k < \min \set{ \dim{U}, \dim{V} }$, both $U$ and $V$ have at least one maximal element of dimension strictly larger than $k$.
It follows that
\begin{align*}
    & \size{ \bigcup_{i > k} \grade{i}{ ( \maxel{ (U \cp{k} V) } ) } } = \\
    & \quad \quad = \size{ \bigcup_{i > k} \grade{i}{ ( \maxel{U} ) } } + \size{ \bigcup_{i > k} \grade{i}{ ( \maxel{V} ) } } \geq 2,
\end{align*}
so $k - 1 < \lydim{(U \cp{k} V)}$, that is, $k \leq \lydim{(U \cp{k} V)}$.
Furthermore, letting $n \eqdef \lydim{(U \cp{k} V)}$, since $n + 1 > k$,
\begin{align*}
    \size{ \bigcup_{i > n+1} \grade{i}{ ( \maxel{U} ) } } & + \size{ \bigcup_{i > n+1} \grade{i}{ ( \maxel{V} ) } } = \\
    & = \size{ \bigcup_{i > n+1} \grade{i}{ ( \maxel{ (U \cp{k} V) } ) } } \leq 1,
\end{align*}
which implies that 
\begin{equation*}
    \size{ \bigcup_{i > n+1} \grade{i}{ ( \maxel{U} ) } } \leq 1, \size{ \bigcup_{i > n+1} \grade{i}{ ( \maxel{V} ) } } \leq 1.
\end{equation*}
Then $\max \set{\lydim{U}, \lydim{V}} \leq \lydim{(U \cp{k} V)}$.
\end{proof}

\begin{lem} \label{lem:lydim_layering_properties}
Let $U$ be a regular molecule, $k \eqdef \lydim{U}$.
Suppose $k \geq 0$, and let $(\order{i}{U})_{i=1}^m$ be a $k$\nbd layering of $U$.
Then
\begin{enumerate}
    \item $m > 1$,
    \item for each $i \in \set{1, \ldots, m}$, $\lydim{\order{i}{U}} < k$,
    \item at most one of the $\order{i}{U}$ contains an element of dimension $> k + 1$.
\end{enumerate}
\end{lem}
\begin{proof}
By definition of $\lydim{U}$, if $k \geq 0$ and a $k$\nbd layering exists, then $m > 1$, for otherwise $k - 1 \leq \lydim{U}$, a contradiction.
Moreover, $U$ contains at most one element of dimension $> k + 1$, which can be contained at most in one of the $\order{i}{U}$.
Finally, by Lemma \ref{lem:layering_basic_properties}, we have $\size{ \bigcup_{j > k} \grade{j}{ ( \maxel{\order{i}{U}} ) } } = 1$, so $\lydim{\order{i}{U}} \leq k - 1 < k$.
\end{proof}

\begin{proof}[Proof of Theorem \ref{thm:molecules_admit_layerings}]
Let $k \eqdef \lydim{U}$.
If $k = -1$, then $U$ is an atom and admits the trivial layering $U = \order{1}{U}$.
If $k \geq 0$, by Lemma \ref{lem:layering_dimension_atom} $U$ is not an atom, so we can assume that $U$ was produced by (\textit{Paste}).
Then $U$ is equal to $V \cp{\ell} W$ for some regular molecules $V, W$ and $\ell < \min \set{ \dim{V}, \dim{W} }$.
By the inductive hypothesis, we have layerings
\begin{equation*}
    \order{1}{V} \cp{k_V} \ldots \cp{k_V} \order{m_V}{V}, \quad \order{1}{W} \cp{k_W} \ldots \cp{k_W} \order{m_W}{W}
\end{equation*}
of $V$ and $W$, where $k_V \eqdef \lydim{V}$ and $k_W \eqdef \lydim{W}$.
Furthermore, by Lemma \ref{lem:layering_dimensions_pasting}, $k \geq \max \set{ k_V, k_W, \ell }$.
Let
\begin{align*}
    n_V & \eqdef 
    \begin{cases}
        m_V & \text{if $k_V = k$,} \\
        1 & \text{if $k_V < k$ and $\dim{V} > k$,} \\
        0 & \text{if $k_V < \dim{V} < k$},
    \end{cases}
    \\
    n_W & \eqdef 
    \begin{cases}
        m_W & \text{if $k_W = k$,} \\
        1 & \text{if $k_W < k$ and $\dim{W} > k$,} \\
        0 & \text{if $k_W < \dim{W} < k$}.
    \end{cases}
\end{align*}
Notice that it can never be the case that $n_V = n_W = 0$.
We claim that we can decompose $V$ as
\begin{equation} \label{eq:padded_decomposition_1}
    \order{1}{\tilde{V}} \cp{k} \ldots \cp{k} \order{n_V}{\tilde{V}} \cp{k} 
    \underbrace{\bound{k}{+}V \cp{k} \ldots \cp{k} \bound{k}{+}V}_{\text{$n_W$ times}},
\end{equation}
where each $\order{i}{\tilde{V}}$ is a regular molecule containing exactly one maximal element of dimension $> k$.
If $k_V = k$, we let $\order{i}{\tilde{V}} \eqdef \order{i}{V}$ for all $i \in \set{1, \ldots, m_V}$.
If $k_V < k$, then $V$ contains at most one maximal element of dimension $> k_V + 1$, hence at most one maximal element of dimension $> k$.
If $\dim{V} > k$, it contains exactly one, and we let $\order{1}{\tilde{V}} \eqdef V$.
If $\dim{V} < k$, then $V = \bound{k}{+}V$.
By Proposition \ref{prop:unitality_of_pasting}, pasting copies of $\bound{k}{+}V$ does not change the result up to unique isomorphism.
Similarly, we can decompose $W$ as
\begin{equation} \label{eq:padded_decomposition_2}
    \underbrace{\bound{k}{-}W \cp{k} \ldots \cp{k} \bound{k}{-}W}_{\text{$n_V$ times}} \cp{k}
    \order{1}{\tilde{W}} \cp{k} \ldots \cp{k} \order{n_W}{\tilde{W}}
\end{equation}
where each $\order{i}{\tilde{W}}$ contains exactly one maximal element of dimension $> k$.

If $\ell = k$, since $\ell < \min \set{ \dim{V}, \dim{W} }$, we have $0 < \min \set{ n_V, n_W }$.
Then
\begin{equation*}
    \order{1}{\tilde{V}} \cp{k} \ldots \cp{k} \order{n_V}{\tilde{V}} \cp{k} 
    \order{1}{\tilde{W}} \cp{k} \ldots \cp{k} \order{n_W}{\tilde{W}}
\end{equation*}
is a $k$\nbd layering of $U$.
If $\ell < k$, let
\begin{equation*}
    \order{i}{U} \eqdef 
    \begin{cases}
        \order{i}{\tilde{V}} \cp{\ell} \bound{k}{-}W & \text{if $i \leq n_V$,} \\
        \bound{k}{+}V \cp{\ell} \order{i-n_V}{\tilde{W}} & \text{if $n_V < i \leq n_V+n_W$}.
    \end{cases}
\end{equation*}
Since $\dim{\bound{k}{-}V} = \dim{\bound{k}{+}W} = k$, each $\order{i}{U}$ still contains exactly one maximal element of dimension $> k$.
Plugging (\ref{eq:padded_decomposition_1}) and (\ref{eq:padded_decomposition_2}) in $V \cp{\ell} W$ and using Proposition \ref{prop:interchange_of_pasting} repeatedly, we deduce that $V \cp{\ell} W$ is isomorphic to
\begin{equation*}
    \order{1}{U} \cp{k} \ldots \cp{k} \order{n_V + n_W}{U},
\end{equation*}
which has the desired properties.
Necessarily, $n_V + n_W = m$.

This proves that $U$ has a $k$\nbd layering.
It remains to show that $\frdim{U} \leq k$.
Let $x, y$ be distinct maximal elements of $U$.
If $\min \set{\dim{x}, \dim{y}} \leq k$, then $\dim{(\clset{x} \cap \clset{y})} < k$.

Suppose that $k < \min \set{\dim{x}, \dim{y}}$, and let $(\order{i}{U})_{i=1}^m$ be a $k$\nbd layering of $U$.
By Lemma \ref{lem:layering_basic_properties} there exist $i \neq j$ such that $x \in \order{i}{U}$ and $y \in \order{j}{U}$.
By Lemma \ref{lem:layering_intersections}, there exists $\alpha \in \set{+, -}$ such that $\order{i}{U} \cap \order{j}{U} = \bound{k}{\alpha}\order{i}{U} \cap \bound{k}{-\alpha}\order{j}{U}$.
Then $\clset{x} \cap \clset{y} \subseteq \bound{k}{\alpha}\order{i}{U} \cap \bound{k}{-\alpha}\order{j}{U}$, which is at most $k$\nbd dimensional.
\end{proof}

\begin{proof}[Proof of Corollary \ref{cor:codimension_1_layering}]
Follows from Theorem \ref{thm:molecules_admit_layerings} together with Lemma \ref{lem:layering_dimension_smaller_than_dimension}.
\end{proof}

\begin{lem} \label{lem:maxflow_acyclic_pasting}
Let $U, V$ be regular molecules and suppose $U \cp{k} V$ is defined for some $k < \min \set {\dim{U}, \dim{V}}$.
If $\maxflow{k}{U}$ and $\maxflow{k}{V}$ are acyclic, then $\maxflow{k}{(U \cp{k} V)}$ is acyclic.
\end{lem}
\begin{proof}
Suppose that $\maxflow{k}{U}$ and $\maxflow{k}{V}$ are acyclic.
We may identify $U$ and $V$ with their images in $U \cp{k} V$.
By Lemma \ref{lem:faces_of_union}, since $\dim{(U \cap V)} = k$,
\begin{equation*}
    \bigcup_{i > k} \grade{i}{ ( \maxel{ (U \cp{k} V) } ) }  = 
    \bigcup_{i > k} \grade{i}{ ( \maxel{U} ) } + \bigcup_{i > k} \grade{i}{ ( \maxel{V} ) },
\end{equation*}
so $\maxflow{k}{U}$ and $\maxflow{k}{V}$ are isomorphic to the induced subgraphs of $\maxflow{k}{(U \cp{k} V)}$ on the vertices in $U$ and $V$, respectively.
It follows that a cycle in $\maxflow{k}{(U \cp{k} V)}$ cannot remain in $U$ or $V$, but has to visit vertices in both.
In particular, such a cycle has to go through an edge from $x \in V$ to $y \in U$, induced by the existence of $z \in \faces{k}{+}x \cap \faces{k}{-}y$.
But then $z \notin \bound{k}{-}V$ and $z \notin \bound{k}{+}U$, yet $z \in U \cap V$, a contradiction.
\end{proof}

\begin{proof}[Proof of Proposition \ref{prop:if_layering_then_ordering}]
Let $(\order{i}{U})_{i=1}^m$ be a $k$\nbd layering of $U$.
For each $i \in \set{1, \ldots, m}$, the graph $\maxflow{k}{\order{i}{U}}$ is trivially acyclic by Lemma \ref{lem:layering_basic_properties}.
We conclude by applying Lemma \ref{lem:maxflow_acyclic_pasting} repeatedly.
\end{proof}

\begin{proof}[Proof of Proposition \ref{prop:layerings_induce_orderings}]
The function $(\order{i}{U})_{i=1}^m \mapsto (\order{i}{x})_{i=1}^m$ is well-defined by Lemma \ref{lem:layering_basic_properties}.
Let $i, j \in \set{1, \ldots, m}$, and suppose that there is an edge from $\order{i}{x}$ to $\order{j}{x}$ in $\maxflow{k}{U}$, that is, there exists $z \in \faces{k}{+}\order{i}{x} \cap \faces{k}{-}\order{j}{x}$.
By Proposition \ref{prop:if_layering_then_ordering}, $\maxflow{k}{U}$ is acyclic, so necessarily $i \neq j$.
If $j < i$, then $\order{j}{U} \cap \order{i}{U} \subseteq \bound{k}{+}\order{j}{U} \cap \bound{k}{-}\order{i}{U}$ by Lemma \ref{lem:layering_intersections}, contradicting the existence of $z$.
It follows that $i < j$, so $(\order{i}{x})_{i=1}^m$ is a $k$\nbd ordering of $U$.

Let $(\order{i}{V})_{i=1}^m$ be another $k$\nbd layering, and suppose it determines the same $k$\nbd ordering as $(\order{i}{U})_{i=1}^m$.
Then the image of both $\order{1}{U}$ and $\order{1}{V}$ in $U$ is
\begin{equation*}
\clos\set{\order{1}{x}} \cup \bound{}{-}U,
\end{equation*}
so $\order{1}{U}$ is isomorphic to $\order{1}{V}$.
If $m = 1$ we are done.
Otherwise, $(\order{i}{U})_{i=2}^m$ and $(\order{i}{V})_{i=2}^m$ are $k$\nbd layerings inducing the same $k$\nbd ordering on their image.
By recursion, we conclude that they are layer-wise isomorphic.
\end{proof}

\begin{lem} \label{lem:substitution_preserves_boundaries}
Let $U, V, W$ be regular molecules, $k < \dim{U}$, $\alpha \in \set{+, -}$, and let $\imath\colon V \incl U$ be a submolecule inclusion such that $\subs{U}{W}{\imath(V)}$ is defined.
Then $\bound{k}{\alpha}U$ is isomorphic to $\bound{k}{\alpha}(\subs{U}{W}{\imath(V)})$.
\end{lem}
\begin{proof}
By Lemma \ref{lem:submolecule_rewrite}, $U \cup (V \celto W)$ is a regular molecule and $U$ is isomorphic to its input boundary.
By globularity, $\bound{k}{\alpha}U$ is isomorphic to 
\begin{equation*}
    \bound{k}{\alpha}(\bound{}{+}(U \cup (V \celto W))) = \bound{k}{\alpha}(\subs{U}{W}{\imath(V)}). \qedhere
\end{equation*}
\end{proof}

\begin{lem} \label{lem:pasting_after_substitution}
Let $U, V, W, U', U''$ be regular molecules, let $k < \dim{U}$, and let $\imath\colon V \incl U$ be a submolecule inclusion such that 
\begin{equation*}
    U \cp{k} U', U'' \cp{k} U, \subs{U}{W}{\imath(V)}
\end{equation*}
are defined.
Then 
\begin{enumerate}
    \item $\subs{U}{W}{\imath(V)} \cp{k} U'$ and $U'' \cp{k} \subs{U}{W}{\imath(V)}$ are defined,
    \item if $\dim{U'} \leq \dim{U}$, then $\subs{(U \cp{k} U')}{W}{\imath_U(\imath(V))}$ is defined and isomorphic to $\subs{U}{W}{\imath(V)} \cp{k} U'$,
    \item if $\dim{U''} \leq \dim{U}$, then $\subs{(U'' \cp{k} U)}{W}{\imath_U(\imath(V))}$ is defined and isomorphic to $U'' \cp{k} \subs{U}{W}{\imath(V)}$.
\end{enumerate}
\end{lem}
\begin{proof}
It follows from Lemma \ref{lem:substitution_preserves_boundaries} that $\subs{U}{W}{\imath(V)} \cp{k} U'$ and $U'' \cp{k} \subs{U}{W}{\imath(V)}$ are defined.

The substitution $\subs{(U \cp{k} U')}{W}{\imath_U(\imath(V))}$ is then defined if and only if $\dim{(U \cp{k} U')} = \dim{U}$, equivalently, if and only if $\dim{U'} \leq \dim{U}$.
Similarly, $\subs{(U'' \cp{k} U)}{W}{\imath_U(\imath(V))}$ is defined if and only if $\dim{U''} \leq \dim{U}$.
The isomorphisms follow straightforwardly from the definitions using the pasting law for pushout squares.
\end{proof}

\begin{dfn}
Theorem \ref{thm:molecules_admit_layerings} in conjunction with Lemma \ref{lem:lydim_layering_properties} and Lemma \ref{lem:layering_dimension_atom} allows us to prove properties of regular molecules \emph{by induction on their layering dimension}.
That is, to prove that a property holds of all regular molecules $U$, it suffices to
\begin{itemize}
    \item prove that it holds when $\lydim{U} = -1$, that is, when $U$ is an atom,
    \item prove that it holds when $k \eqdef \lydim{U} \geq 0$, assuming that it holds of all the $(\order{i}{U})_{i=1}^m$ in a $k$\nbd layering of $U$.
\end{itemize}
\end{dfn}

\begin{lem} \label{lem:boundary_move}
Let $U$ be a regular molecule, $k \in \mathbb{N}$, and suppose
\begin{equation*}
    \bigcup_{i > k} \grade{i}{ ( \maxel{U} ) } = \set{x}.
\end{equation*}
Then, for all $\alpha \in \set { +, - }$,
\begin{enumerate}
    \item $\bound{k}{\alpha}x \submol \bound{k}{\alpha}U$, 
    \item $\bound{k}{\alpha}U$ is isomorphic to $\subs{\bound{k}{-\alpha}U}{\bound{k}{\alpha}x}{\bound{k}{-\alpha}x}$.
\end{enumerate}
\end{lem}
\begin{proof}
We proceed by induction on $\lydim{U}$.

If $\lydim{U} = -1$, then $U$ is an atom and equal to $\clset{x}$.
It follows that $\bound{k}{\alpha}x = \bound{k}{\alpha}U$, which is trivially a submolecule, and is isomorphic to $\subs{\bound{k}{-\alpha}U}{\bound{k}{\alpha}x}{\bound{k}{-\alpha}{x}}$.

Suppose $\ell \eqdef \lydim{U} \geq 0$, and let $(\order{i}{U})_{i = 1}^m$ be an $\ell$\nbd layering of $U$.
Then $\ell \leq k-1 < k$ because $\size{\bigcup_{i > k} \grade{i}{ ( \maxel{U} ) }} = 1$.
By \cite[Proposition 1.23]{hadzihasanovic2020diagrammatic} and the axioms of strict $\omega$\nbd categories, $\bound{k}{\alpha}U$ is isomorphic to
\begin{equation*}
    \bound{k}{\alpha}\order{1}{U} \cp{\ell} \ldots \cp{\ell} \bound{k}{\alpha}\order{m}{U}.
\end{equation*}
Now $x$ is contained in a single $\order{i}{U}$.
By the inductive hypothesis, $\bound{k}{\alpha}x \submol \bound{k}{\alpha}\order{i}{U}$, and the latter is isomorphic to $\subs{\bound{k}{-\alpha}\order{i}{U}}{\bound{k}{\alpha}x}{\bound{k}{-\alpha}{x}}$.
We conclude by Lemma \ref{lem:pasting_after_substitution}.
\end{proof}

\begin{proof}[Proof of Proposition \ref{prop:layering_from_ordering}]
Suppose $(\order{i}{U})_{i=1}^m$ is a $k$\nbd layering.
Then, for all $i \in \set{1, \ldots, m}$, $\order{i}{U}$ is a regular molecule, and by Proposition \ref{lem:layering_basic_properties} $\order{i}{x}$ is the only element of dimension $> k$ in $\order{i}{U}$.
By Lemma \ref{lem:boundary_move}, $\bound{k}{-}\order{i}{x} \submol \bound{k}{-}\order{i}{U}$.

Conversely, it follows from Lemma \ref{lem:submolecule_rewrite} that for all $i$, if $\bound{k}{-}\order{i}{U}$ is a regular molecule and $\bound{k}{-}\order{i}{x}$ is its submolecule, then $\order{i}{U}$ is a regular molecule, hence $\bound{k}{+}\order{i}{U}$ is a regular molecule.
Moreover, since $(\order{i}{x})_{i=1}^m$ is a $k$\nbd ordering, it is straightforward to prove that $\order{i}{U} \cap \order{i+1}{U} = \bound{k}{+}\order{i}{U} = \bound{k}{-}\order{i+1}{U}$ for all $i \in \set{1, \ldots, m-1}$.
Since $\bound{}{-}\order{1}{U} = \bound{}{-}U$ is a regular molecule, it follows by induction, assuming condition \ref{cond:inputs_are_submolecules}, that $\order{i}{U}$ is a regular molecule for all $i \in \set{1, \ldots, m}$.
This proves that $(\order{i}{U})_{i=1}^m$ is a $k$\nbd layering of $U$.
\end{proof}

\begin{dfn}
In the following, we use the following explicit construction of $\mathscr{G}/(\restr{\mathscr{G}}{W})$.
Its set of vertices  is $(V_\mathscr{G} \setminus W) + \set{x_W}$, and for all pair of vertices $x, y$,
\begin{itemize}
    \item if $x, y \neq x_W$, there is an edge between $x$ and $y$ for each edge between $x$ and $y$ in $\mathscr{G}$,
    \item if $x = x_W$ and $y \neq x_W$, there is an edge from $x$ to $y$ for each pair of a vertex $z \in W$ and an edge from $x$ to $y$ in $\mathscr{G}$,
    \item if $x \neq x_W$ and $y = x_W$, there is an edge from $x$ to $y$ for each pair of a vertex $z \in W$ and an edge from $x$ to $z$ in $\mathscr{G}$,
    \item there are no edges from $x_W$ to $x_W$.
\end{itemize}
\end{dfn}

\begin{proof}[Proof of Lemma \ref{lem:connected_subgraph_conditions_path_induced}]
We prove the contrapositive of the implication from \ref{cond:path_induced} to \ref{cond:acyclic_contraction}.
Suppose $\mathscr{G}/(\restr{\mathscr{G}}{W})$ has a cycle.
If the cycle does not pass through $x_W$, then it lifts to a cycle in $\mathscr{G}$, contradicting the assumption that $\mathscr{G}$ is acyclic.
It follows that the cycle contains a segment of the form $x_W \to x_1 \to \ldots \to x_m \to x_W$, where $m > 0$ and $x_i \neq x_W$ for all $i \in \set{1, \ldots, m}$.
Then there exist $y, z \in W$ and a path $y \to x_1 \to \ldots \to x_m \to z$ in $\mathscr{G}$, so $\restr{\mathscr{G}}{W}$ is not path-induced.

Next, suppose that $\mathscr{G}/(\restr{\mathscr{G}}{W})$ is acyclic.
Then both the graphs $\mathscr{G}/(\restr{\mathscr{G}}{W})$ and $\restr{\mathscr{G}}{W}$ are acyclic, so they admit topological sorts $(\order{i}{x})_{i=1}^m$ and $(\order{j}{y})_{j=1}^p$, respectively.
For exactly one $q \in \set{1, \ldots, m}$, $\order{i}{x} = x_W$.
We claim that 
\begin{equation*}
    ((\order{i}{x})_{i=1}^{q-1}, (\order{j}{y})_{j=1}^p, (\order{i}{x})_{i=q+1}^m)
\end{equation*}
is a topological sort of $\mathscr{G}$.
Indeed, for all edges from $x$ to $x'$ in $\mathscr{G}$,
\begin{itemize}
    \item if $x, x' \notin W$, then $x = \order{i}{x}$, $x' = \order{i'}{x}$ for some $i, i' \in \set{1, \ldots, m} \setminus \set{q}$, and there is an edge from $x$ to $x'$ in $\mathscr{G}/(\restr{\mathscr{G}}{W})$, so $i < i'$;
    \item if $x, x' \in W$, then $x = \order{j}{y}$, $x' = \order{j'}{y}$ for some $j, j' \in \set{1, \ldots, p}$, and there is an edge from $x$ to $x'$ in $\restr{\mathscr{G}}{W}$, so $j < j'$;
    \item if $x \in W$, $x' \notin W$, then $x = \order{j}{y}$, $x' = \order{i}{x}$ for some $i \in \set{1, \ldots, m} \setminus \set{q}$, $j \in \set{1, \ldots, p}$, and there is an edge from $x_W$ to $x'$ in $\mathscr{G}/(\restr{\mathscr{G}}{W})$, so $q < i$;
    \item if $x \notin W$, $x' \in W$, then $x = \order{i}{x}$, $x' = \order{j}{y}$ for some $i \in \set{1, \ldots, m} \setminus \set{q}$, $j \in \set{1, \ldots, p}$, and there is an edge from $x$ to $x_W$ in $\mathscr{G}/(\restr{\mathscr{G}}{W})$, so $i < q$.
\end{itemize}
This proves the implication from \ref{cond:acyclic_contraction} to \ref{cond:consecutive_tsort}.
Moreover, it defines an injection from pairs of a topological sort of $\restr{\mathscr{G}}{W}$ and a topological sort of $\mathscr{G}/(\restr{\mathscr{G}}{W})$ to topological sorts of $\mathscr{G}$ in which the vertices of $W$ are consecutive.
This will prove to be a bijection as soon as we have proven the converse implication.

Finally, we prove the contrapositive of the implication from \ref{cond:consecutive_tsort} to \ref{cond:path_induced}.
Suppose $\restr{\mathscr{G}}{W}$ is not path-induced, that is, there is a path $x \to x_1 \to \ldots \to x_m \to y$ in $\mathscr{G}$ such that $m > 0$, $x, y \in W$, and $x_i \notin W$ for all $i \in \set{1, \ldots, m}$.
It follows that the $x_i$ must come between $x$ and $y$ in every topological sort of $\mathscr{G}$, so the vertices of $W$ can never be consecutive.
\end{proof}

\begin{proof}[Proof of Lemma \ref{lem:flow_of_substitution_is_contraction_of_flow}]
By Lemma \ref{lem:flow_under_inclusion} combined with Proposition \ref{prop:round_molecule_connected_flowgraph}, $\flow{n-1}{V}$ is a connected induced subgraph of $\flow{n-1}{U}$, so its contraction is well-defined.
Now, the vertices of the graph $\flow{n-1}{\subs{U}{\compos{V}}{\imath(V)}}$ are either
\begin{itemize}
    \item $x \in \grade{n}{U} \setminus \grade{n}{V}$, or
    \item $x_V$ such that the image of $\compos{V}$ in $\subs{U}{\compos{V}}{\imath(V)}$ is $\clset{x_V}$.
\end{itemize}
Let $x, y$ be two vertices of $\flow{n-1}{\subs{U}{\compos{V}}{\imath(V)}}$.
\begin{itemize}
\item If $x, y \in \grade{n}{U} \setminus \grade{n}{V}$, then $\faces{}{+}x \cap \faces{}{-}y$ is the same in $\subs{U}{\compos{V}}{\imath(V)}$ as in $U$, so there is an edge from $x$ to $y$ in $\flow{n-1}{U}$ if and only if there is an edge in $\flow{n-1}{\subs{U}{\compos{V}}{\imath(V)}}$.
\item If $x = x_V$ then $\faces{}{+}x_V \cap \faces{}{-}y$ is in bijection with $\faces{}{+}V \cap \faces{}{-}y$ in $U$.
For all $z \in \faces{}{+}V$, since $V$ is pure and $n$\nbd dimensional, there exists $w \in \cofaces{}{+}z$.
If $\faces{}{+}x_V \cap \faces{}{-}y$ is non-empty, it follows that $\faces{}{+}z \cap \faces{}{-}y$ is non-empty in $U$ for some $z \in \grade{n}{\imath(V)}$.
Thus there exist $z \in \grade{n}{\imath(V)}$ and an edge from $z$ to $y$ in $\flow{n-1}{U}$.
\item Dually, if $y = x_V$, there is an edge from $x$ to $y$ in $\flow{n-1}{\subs{U}{\compos{V}}{\imath(V)}}$ if and only if there exist $z \in \grade{n}{\imath(V)}$ and an edge from $x$ to $z$ in $\flow{n-1}{U}$.
\item Finally, $\faces{}{+}V \cap \faces{}{-}V = \varnothing$ because $V$ is pure, so $\faces{}{+}x_V \cap \faces{}{-}x_V$ and there is no edge from $x_V$ to $x_V$.
\end{itemize}
It is then straightforward to establish an isomorphism with the explicit description of $\flow{n-1}{U}/\flow{n-1}{V}$.
\end{proof}

\begin{proof}[Proof of Lemma \ref{lem:round_submolecules_from_layering}]
Identify $V$ with its isomorphic image through $\imath$, and suppose that $\imath$ is a submolecule inclusion.
Then $\tilde{U} \eqdef \subs{U}{\compos{V}}{V}$ is a regular molecule by Proposition \ref{prop:round_submolecule_substitution}, and admits an $(n-1)$\nbd layering $(\order{i}{\tilde{U}})_{i=1}^{m-p+1}$ by Theorem \ref{thm:molecules_admit_layerings}.
Let $\clset{x}$ be the image of $\compos{V}$ in $\tilde{U}$; then $x \in \order{q}{\tilde{U}}$ for exactly one $q \in \set{1, \ldots, m-p+1}$.
Then $W \eqdef \subs{\order{q}{\tilde{U}}}{V}{\clset{x}}$ is defined, and by Lemma \ref{lem:pasting_after_substitution} combined with Lemma \ref{lem:revert_substitution}, $U$ is isomorphic to
\begin{align*}
    \order{1}{\tilde{U}} & \cp{n-1} \ldots \cp{n-1} \order{q-1}{\tilde{U}} \cp{n-1} \\
    & \cp{n-1} W \cp{n-1} \order{q+1}{\tilde{U}} \cp{n-1} \ldots \cp{n-1} \order{m-p+1}{\tilde{U}}.
\end{align*}
By Lemma \ref{lem:boundary_move}, $\bound{}{-}x \submol \bound{}{-}\order{q}{\tilde{U}}$, so by Lemma \ref{lem:substitution_preserves_boundaries} $\bound{}{-}V \submol \bound{}{-}W$.
We can apply the criterion of Proposition \ref{prop:layering_from_ordering} to deduce that $(\order{i}{y})_{i=1}^p$ is an $(n-1)$\nbd ordering of $W$ induced by an $(n-1)$\nbd layering $(\order{i}{W})_{i=1}^p$.
Letting 
\begin{equation*}
    (\order{i}{U})_{i=1}^m \eqdef ((\order{i}{\tilde{U}})_{i=1}^{q-1}, (\order{i}{W})_{i=1}^p, (\order{i}{\tilde{U}})_{i=q+1}^{m-p+1}),
\end{equation*}
produces an $(n-1)$\nbd layering, hence an $(n-1)$\nbd ordering $(\order{i}{x})_{i=1}^m$ of $U$, with the property that $(\order{i}{x})_{i=q}^{p+q-1} = (\order{i}{y})_{i=1}^p$.

Conversely, let $(\order{i}{U})_{i=1}^m$ be an $(n-1)$\nbd layering of $U$ satisfying the properties in the statement, and let $W \submol U$ be the image of $\order{q}{U} \cp{n-1} \ldots \cp{n-1} \order{p+q-1}{U}$ in $U$.
Then $\grade{n}{W} = \grade{n}{V}$, so 
\begin{equation*}
    W = V \cup \bound{}{-}W.
\end{equation*}
Because $\bound{}{-}V \submol \bound{}{-}\order{q}{U}$ which is equal to $\bound{}{-}W$, by Lemma \ref{lem:submolecule_rewrite} $V \submol W \submol U$.
\end{proof}

\begin{proof}[Full proof of Theorem \ref{thm:rewritable_submolecule_criterion}]
Identify $V$ with its isomorphic image through $\imath$, and suppose that $\imath$ is a submolecule inclusion.
Then $\tilde{U} \eqdef \subs{U}{\compos{V}}{V}$ is a regular molecule by Proposition \ref{prop:round_submolecule_substitution}, so it admits an $(n-1)$\nbd layering $(\order{i}{\tilde{U}})_{i=1}^{m-p+1}$, which induces an $(n-1)$\nbd ordering.
By Lemma \ref{lem:flow_of_substitution_is_contraction_of_flow}, this $(n-1)$\nbd ordering can be identified with a topological sort $((\order{i}{x})_{i=1}^{q-1}, x_V, (\order{i}{x})_{i=q+1}^{m-p+1})$ of $\flow{n-1}{U}/\flow{n-1}{V}$.
By Lemma \ref{lem:boundary_move}, we have $\bound{}{-}x_V \submol \bound{}{-}\order{q}{\tilde{U}}$ and $\bound{}{-}\order{i}{x} \submol \bound{}{-}\order{i}{\tilde{U}}$ for $i \neq q$.
By Lemma \ref{lem:pasting_after_substitution} combined with Lemma \ref{lem:revert_substitution}, letting $W \eqdef \subs{\order{q}{\tilde{U}}}{V}{\clset{x_V}}$, $U$ is isomorphic to
\begin{align*}
    \order{1}{\tilde{U}} & \cp{n-1} \ldots \cp{n-1} \order{q-1}{\tilde{U}} \cp{n-1} \\
    & \cp{n-1} W \cp{n-1} \order{q+1}{\tilde{U}} \cp{n-1} \ldots \cp{n-1} \order{m-p+1}{\tilde{U}}.
\end{align*}
and $W$ is isomorphic to $\order{q}{U}$, while $\order{i}{\tilde{U}}$ is isomorphic to $\order{i}{U}$ for all $i \neq q$.
We conclude by Lemma \ref{lem:substitution_preserves_boundaries}.

The converse implication has already been fully proved.
\end{proof}

\begin{lem} \label{lem:lto_bijection_above_layering}
Let $U$ be a regular molecule, $\ell \geq -1$.
If $U$ has an $\ell$\nbd layering, then for all $k > \ell$ the function $\lto{k}{U}\colon \layerings{k}{U} \incl \orderings{k}{U}$ is a bijection.
\end{lem}
\begin{proof}
Let $(\order{i}{U})_{i=1}^m$ be an $\ell$\nbd layering of $U$, and let $(\order{i}{x})_{i=1}^m$ be its image through $\lto{\ell}{U}$.
For $k > \ell$, let $(\order{i}{y})_{i=1}^p$ be a $k$\nbd ordering of $U$.
Then there exists a unique injection $\fun{j}\colon \set{1, \ldots, p} \incl \set{1, \ldots, m}$ such that $\order{i}{y} = \order{\fun{j}(i)}{x}$ for all $i \in \set{1, \ldots, p}$.
Let
\begin{align*}
	\order{i}{V} & \eqdef \bound{k}{\alpha(i,1)}\order{1}{U} \cp{\ell} \ldots \cp{\ell} \order{\fun{j}(i)}{U} \cp{\ell} \ldots \cp{\ell} \bound{k}{\alpha(i,m)}\order{m}{U}, \\
	& \alpha(i,j) \eqdef \begin{cases} + & \text{if $j = \fun{j}(i')$ for some $i' < i$,} \\ 
	- & \text{otherwise}.
	\end{cases}
\end{align*}
Applying Proposition \ref{prop:unitality_of_pasting} and Proposition \ref{prop:interchange_of_pasting} repeatedly, we find that $(\order{i}{V})_{i=1}^p$ is a $k$\nbd layering of $U$ and $(\order{i}{y})_{i=1}^p$ is its image through $\lto{k}{U}$.
This proves that $\lto{k}{U}$ is surjective, and we conclude by Proposition \ref{prop:layerings_induce_orderings}.
\end{proof}

\begin{lem} \label{lem:frdim_layerings_bijection_with_orderings}
Let $U$ be a regular molecule.
Suppose that for all submolecules $V \submol U$, if $r \eqdef \frdim{V}$, then $V$ admits an $r$\nbd layering.
Then for all $k \geq \frdim{U}$ the function $\lto{k}{U}\colon \layerings{k}{U} \incl \orderings{k}{U}$ is a bijection.
\end{lem}
\begin{proof}
Let $r \eqdef \frdim{U}$.
By assumption, there exists an $r$\nbd layering of $U$, so by Lemma \ref{lem:lto_bijection_above_layering} it suffices to show that $\lto{r}{U}$ is a bijection.

Given two $r$\nbd orderings $(\order{i}{x})_{i=1}^m$ and $(\order{i}{y})_{i=1}^m$, there exists a unique permutation $\sigma$ such that $\order{i}{x} = \order{\sigma(i)}{y}$ for all $i \in \set{1, \ldots, m}$.
Let $d((\order{i}{x})_{i=1}^m, (\order{i}{y})_{i=1}^m)$ be the number of pairs $(j, j')$ such that $j < j'$ but $\sigma(j') < \sigma(j)$.
Under the assumption that $(\order{i}{x})_{i=1}^m$ is in the image of $\lto{r}{U}$, we will prove that $(\order{i}{y})_{i=1}^m$ is also in the image of $\lto{r}{U}$ by induction on $d((\order{i}{x})_{i=1}^m, (\order{i}{y})_{i=1}^m)$.
Since the image of $\lto{r}{U}$ is not empty, this will suffice to prove that $\lto{r}{U}$ is surjective, hence bijective by Proposition \ref{prop:layerings_induce_orderings}.

If $d((\order{i}{x})_{i=1}^m, (\order{i}{y})_{i=1}^m) = 0$, then $\order{i}{x} = \order{i}{y}$ for all $i \in \set{1, \ldots, m}$, and there is nothing left to prove.

Suppose $d((\order{i}{x})_{i=1}^m, (\order{i}{y})_{i=1}^m) > 0$.
Then there exists $j < m$ such that $\sigma(j+1) < \sigma(j)$.
Suppose $(\order{i}{x})_{i=1}^m$ is the image of the $r$\nbd layering $(\order{i}{U})_{i=1}^m$.
Let $V \submol U$ be the image of $\order{j}{U} \cp{r} \order{j+1}{U}$ in $U$, and let 
\begin{equation*}
    z_1 \eqdef \order{j}{x} = \order{\sigma(j)}{y}, \quad z_2 \eqdef \order{j+1}{x} = \order{\sigma(j+1)}{y}.
\end{equation*}
Because $z_1$ comes before $z_2$ in one $r$\nbd ordering, but after in another, there can be no edge between them in $\maxflow{r}{U}$, so
\begin{equation*}
    \dim{ (\clset{z_1} \cap \clset{z_2}) } < r.
\end{equation*}
Since $z_1, z_2$ are the only maximal elements of dimension $> r$ in $V$, we deduce that $\ell \eqdef \frdim{V} < r$.
By assumption, there exists an $\ell$\nbd layering of $V$.
In particular, there exist regular molecules $\order{1}{V}, \order{2}{V}$ such that
\begin{enumerate}
    \item $z_i$ is in the image of $\order{i}{V}$ for all $i \in \set{1, 2}$, and
    \item $V$ is isomorphic to $\order{1}{V} \cp{\ell} \order{2}{V}$ or to $\order{2}{V} \cp{\ell} \order{1}{V}$.
\end{enumerate}
Without loss of generality suppose that $V$ is isomorphic to $\order{1}{V} \cp{\ell} \order{2}{V}$.
By Proposition \ref{prop:unitality_of_pasting} and Proposition \ref{prop:interchange_of_pasting}, letting
\begin{align*}
    \order{j}{\tilde{U}} & \eqdef \bound{r}{-}\order{1}{V} \cp{\ell} \order{2}{V}, \\
    \order{j+1}{\tilde{U}} & \eqdef \order{1}{V} \cp{\ell} \bound{r}{+} \order{2}{V},
\end{align*}
we have that $V$ is isomorphic to $\order{j}{\tilde{U}} \cp{r} \order{j+1}{\tilde{U}}$.
Letting $\order{i}{\tilde{U}} \eqdef \order{i}{U}$ for $i \notin \set{j, j+1}$, we have that $(\order{i}{\tilde{U}})_{i=1}^m$ is an $r$\nbd layering of $U$, and
\begin{align*}
    \lto{r}{U}\colon (\order{i}{\tilde{U}})_{i=1}^m & \mapsto (\order{i}{\tilde{x}})_{i=1}^m = \\
    & = (\order{1}{x}, \ldots, \order{j+1}{x}, \order{j}{x}, \ldots, \order{m}{x}).
\end{align*}
Then $d((\order{i}{\tilde{x}})_{i=1}^m, (\order{i}{y})_{i=1}^m) < d((\order{i}{x})_{i=1}^m, (\order{i}{y})_{i=1}^m)$ and $(\order{i}{\tilde{x}})_{i=1}^m$ is in the image of $\lto{r}{U}$.
We conclude by the inductive hypothesis.
\end{proof}

\begin{lem} \label{lem:frame_acyclic_has_frame_layerings}
Let $U$ be a regular molecule, $r \eqdef \frdim{U}$.
If $U$ is frame-acyclic, then $U$ admits an $r$\nbd layering.
\end{lem}
\begin{proof}
Since submolecules of frame-acyclic regular molecules are frame-acyclic, we can proceed by induction on submolecules.
For all $x \in \grade{0}{U}$, we have $\frdim{\set{x}} = -1$, and $\set{x}$ admits a trivial $(-1)$\nbd layering, proving the base case.

We construct an ordered tree of submolecules $\order{j_1, \ldots, j_p}{U}$ of $U$, as follows:
\begin{itemize}
    \item the root is $\order{}{U} \eqdef U$;
    \item if $\lydim{\order{j_1, \ldots, j_p}{U}} \leq r$, then we let $\lydim{\order{j_1, \ldots, j_p}{U}}$ be a leaf;
    \item if $k \eqdef \lydim{\order{j_1, \ldots, j_p}{U}} > r$, then we pick a $k$\nbd layering $(\order{i}{V})_{i=1}^q$ of $\order{j_1, \ldots, j_p}{U}$, which is possible by Theorem \ref{thm:molecules_admit_layerings}, and for each $i \in \set{1, \ldots, q}$, we let the image of $\order{i}{V}$ be a child $\order{j_1, \ldots, j_p, i}{U}$ of $\order{j_1, \ldots, j_p}{U}$.
\end{itemize}
By Lemma \ref{lem:lydim_layering_properties}, the layering dimension of the children of a node is strictly smaller than that of the node, so the procedure terminates.

Fix an $r$\nbd ordering $(\order{i}{x})_{i=1}^m$ of $U$; this is possible because $\maxflow{r}{U}$ is acyclic.
Let $V \eqdef \order{j_1, \ldots, j_p}{U}$ be a node of the tree.
We have
\begin{align*}
    \bigcup_{j > r} \grade{j}{(\maxel{V})} & = 
    \sum_{i = 1}^m \bigcup_{j > r} 
    \left(\grade{j}{(\maxel{V})} \cap \clset{\order{i}{x}}\right) \\
    & \eqqcolon \sum_{i = 1}^m \order{i}{M};
\end{align*}
the $\order{i}{M}$ form a partition because $\frdim{U} = r$, so every element of dimension $> r$ is in the closure of $\order{i}{x}$ for a unique $i \in \set{1, \ldots, m}$.
We claim that $V$ is isomorphic to
\begin{equation*}
    \order{1}{V} \cp{r} \ldots \cp{r} \order{m}{V}
\end{equation*}
for some regular molecules $(\order{i}{V})_{i=1}^m$ with the following property: for each $i \in \set{1, \ldots, m}$, identifying $\order{i}{V}$ with its image in $V$, we have
\begin{equation*}
    \bigcup_{j > r} \grade{j}{(\maxel{\order{i}{V}})} = 
    \order{i}{M}.
\end{equation*}
We will prove this by backward induction on the tree $\order{j_1, \ldots, j_p}{U}$.

Suppose $V$ is a leaf, so $\lydim{V} \leq r$.
Then $V$ admits an $r$\nbd layering.
For each $i \in \set{1, \ldots, m}$, fix a topological sort $(\order{i, j}{y})_{j=1}^{p_i}$ of the induced subgraph $\restr{\maxflow{r}{V}}{\order{i}{M}}$.
We claim that $((\order{i,j}{y})_{j=1}^{p_i})_{i=1}^m$ is an $r$\nbd ordering of $V$.

Suppose there is an edge from $x$ to $x'$ in $\maxflow{r}{V}$.
Then $x \in \order{i}{M}$, $x' \in \order{i'}{M}$ for a unique pair $i, i' \in \set{1, \ldots, m}$.
If $i = i'$, then $x = \order{i, j}{y}$ and $x' = \order{i, j'}{y}$ for some $j, j' \in \set{1, \ldots, p_i}$, and $j < j'$ because $(\order{i, j}{y})_{j=1}^{p_i}$ is a topological sort of $\restr{\maxflow{r}{V}}{\order{i}{M}}$.
If $i \neq i'$, then there exists 
\begin{equation*}
    z \in \faces{r}{+}x \cap \faces{r}{-}x' \subseteq \clset{\order{i}{x}} \cap \clset{\order{i'}{x}}.
\end{equation*}
Since $\bound{r}{\alpha}\order{i}{x}$ and $\bound{r}{\alpha}\order{i'}{x}$ is pure and $r$\nbd dimensional for all $\alpha \in \set{+, -}$, by \cite[Proposition 6.4]{steiner1993algebra}
\begin{equation*}
    z \in (\faces{r}{+}\order{i}{x} \cap \faces{r}{-}\order{i'}{x}) \cup (\faces{r}{-}\order{i}{x} \cap \faces{r}{+}\order{i'}{x}),
\end{equation*}
and $\faces{r}{-}\order{i}{x} \cap \clset{x} \subseteq \faces{r}{-}x$ which is disjoint from $\faces{r}{+}x$, so $z \in \faces{r}{+}\order{i}{x} \cap \faces{r}{-}\order{i'}{x}$.
It follows that there is an edge from $\order{i}{x}$ to $\order{i'}{x}$ in $\maxflow{r}{U}$, so $i < i'$ because $(\order{i}{x})_{i=1}^m$ is a topological sort of $\maxflow{r}{U}$.
This proves that $((\order{i,j}{y})_{j=1}^{p_i})_{i=1}^m$ is an $r$\nbd ordering of $V$.

Let $W \submol V$, $\ell \eqdef \frdim{W}$.
If $V \neq U$ or $W \neq U$, then $W$ admits an $\ell$\nbd layering by the inductive hypothesis on proper submolecules of $U$.
If $W = V = U$ then $\ell = r$ and $W$ admits an $\ell$\nbd layering by Theorem \ref{thm:molecules_admit_layerings}.
In either case, $V$ satisfies the conditions of Lemma \ref{lem:frdim_layerings_bijection_with_orderings}, and since $r \geq \lydim{V} \geq \frdim{V}$, every $r$\nbd ordering of $V$ comes from an $r$\nbd layering of $V$.

It follows that $((\order{i,j}{y})_{j=1}^{p_i})_{i=1}^m$ comes from an $r$\nbd layering $((\order{i,j}{W})_{j=1}^{p_i})_{i=1}^m$, and we can define
\begin{equation*}
    \order{i}{V} \eqdef \order{i,1}{W} \cp{r} \ldots \cp{r} \order{i, p_i}{W}
\end{equation*}
for each $i \in \set{1, \ldots, m}$, satisfying the desired condition.
    
Now, suppose that $V$ is not a leaf, so $k \eqdef \lydim{V} > r$, and $V$ has children $(\order{j}{W})_{j=1}^q$ forming a $k$\nbd layering of $V$.
By the inductive hypothesis, each of the $\order{j}{W}$ has a decomposition
\begin{equation*}
    \order{j, 1}{W} \cp{r} \ldots \cp{r} \order{j, m}{W}
\end{equation*}
such that the maximal elements of dimension $> r$ in the image of $\order{j, i}{W}$ are contained in $\clset{\order{i}{x}}$.
Then, for each $i \in \set{1, \ldots, m}$ and $j, j' \in \set{1, \ldots, q}$,
\begin{equation*}
    \order{j, i}{W} \cap \order{j'}{W} \subseteq \order{j', i}{W},
\end{equation*}
so $\order{i}{V} \eqdef \order{1, i}{W} \cp{k} \ldots \cp{k} \order{q, i}{W}$ is defined.
Using Proposition \ref{prop:interchange_of_pasting} repeatedly, we conclude that $V$ is isomorphic to $\order{1}{V} \cp{r} \ldots \cp{r} \order{m}{V}$.

This concludes the induction on $\order{j_1, \ldots, j_p}{U}$.
In particular, for the root $\order{}{U} = U$, the decomposition $\order{1}{U} \cp{r} \ldots \cp{r} \order{m}{U}$ satisfies
\begin{equation*}
    \bigcup_{j > r} \grade{j}{(\maxel{\order{i}{U}})} = 
    \set{\order{i}{x}},
\end{equation*}
that is, $(\order{i}{U})_{i=1}^m$ is an $r$\nbd layering of $U$.
\end{proof}

\begin{proof}[Proof of Theorem \ref{thm:frame_acyclicity_equivalent_conditions}]
The implication from \ref{cond:frame_acyclic} to \ref{cond:frdim_layerings} is a consequence of Lemma \ref{lem:frame_acyclic_has_frame_layerings} together Lemma \ref{lem:layering_basic_properties}.
The implication from \ref{cond:frdim_layerings} to \ref{cond:frdim_bijection} is Lemma \ref{lem:frdim_layerings_bijection_with_orderings}.
Finally, the implication from \ref{cond:frdim_bijection} to \ref{cond:frame_acyclic} follows from Proposition \ref{prop:if_layering_then_ordering}.
\end{proof}

\begin{proof}[Proof of Lemma \ref{lem:acyclicity_implications}]
Suppose $U$ is acyclic.
Let $k \in \mathbb{N}$ and suppose there is a cycle $x_0 \to x_1 \to \ldots \to x_m = x_0$ in $\flow{k}{U}$.
By definition, for all $i \in \set{1, \ldots, m}$ there exists $y_i \in \faces{k}{+}x_{i-1} \cap \faces{k}{-}x_i$.
By \cite[Lemma 1.37]{hadzihasanovic2020diagrammatic}, there exist paths from $x_{i-1}$ to $y_i$ and then from $y_i$ to $x_i$ in $\hasseo{U}$.
Concatenating all these paths, we obtain a cycle in $\hasseo{U}$.

Suppose $U$ is dimension-wise acyclic, and let $V \submol U$ be a submolecule inclusion, $r \eqdef \frdim{V}$.
Then $\flow{r}{U}$ is acyclic, hence so are its induced subgraphs $\flow{r}{V}$ and $\maxflow{r}{V}$.
\end{proof}


\subsection{Proofs for Section \ref{sec:lowdim}}

\begin{proof}[Proof of Lemma \ref{lem:only_1_molecules}]
We will show that $U$ is isomorphic to 
\begin{equation*}
    \underbrace{\thearrow{} \cp{0} \ldots \cp{0} \thearrow{}}_{\text{$m$ times}}.
\end{equation*}
By Lemma \ref{lem:layering_dimension_smaller_than_dimension}, either $\lydim{U} = -1$ or $\lydim{U} = 0$.
In the first case, $U$ is an atom by Lemma \ref{lem:layering_dimension_atom}.
Because the point is the only 0\nbd dimensional molecule up to isomorphism, the arrow is the only 1\nbd dimensional atom, so $U$ is isomorphic to $\thearrow{}$.
In the second case, $U$ admits a 0\nbd layering $(\order{i}{U})_{i=1}^m$ by Theorem \ref{thm:molecules_admit_layerings}, and by Lemma \ref{lem:lydim_layering_properties}, for each $i \in \set{1, \ldots, m}$, necessarily $\lydim{\order{i}{U}} = -1$.
By the first part, $\order{i}{U}$ is isomorphic to $\thearrow{}$.
\end{proof}

\begin{proof}[Proof of Proposition \ref{prop:dim2_acyclic}]
This is \cite[Proposition 5]{hadzihasanovic2021smash}.
\end{proof}

\begin{dfn}[Horizontal and vertical order]
Let $U$ be a regular molecule, $\dim{U} \leq 2$.
The \emph{horizontal order} $\prech$ and the \emph{vertical order} $\precv$ on the set $\grade{1}{U}$ of 1\nbd dimensional elements of $U$ are defined by
\begin{itemize}
    \item $x \prech y$ if and only if there is a path from $x$ to $y$ in $\hasseo{U}$ only passing through elements of dimension $0$ and $1$,
    \item $x \precv y$ if and only if there is a path from $x$ to $y$ in $\hasseo{U}$ only passing through elements of dimension $1$ and $2$.
\end{itemize}
\end{dfn}

\begin{lem} \label{lem:dim2_horizontal_or_vertical_path}
Let $U$ be a regular molecule, $\dim{U} \leq 2$.
Then
\begin{enumerate}
    \item the union of $\prech$ and $\precv$ is a linear order on $\grade{1}{U}$,
    \item the intersection of $\prech$ and $\precv$ is the identity relation on $\grade{1}{U}$.
\end{enumerate}
\end{lem}
\begin{proof}
If $\dim{U} < 2$, then $\precv$ is trivially the identity relation, and $\prech$ is a linear order by \ref{lem:only_1_molecules}.
If $U$ is a 2\nbd dimensional atom, then $\grade{1}{U} = \faces{}{-}U + \faces{}{+}U$, and that $\prech$ is a linear order on $\faces{}{\alpha}U$ for each $\alpha \in \set{+, -}$ separately, so we have $x \precv y$ for all $x \in \faces{}{-}U$ and $y \in \faces{}{+}U$, and no other relations exist.

Otherwise, we proceed exactly as in the proof of Proposition \ref{prop:dim2_acyclic}, defining a 1\nbd ordering $(\order{i}{x})_{i=1}^m$ and a sequence $(\order{i}{V})_{i=1}^m$ of increasing submolecules of $U$.
We let $\order{i}{\prech}$ and $\order{i}{\precv}$ be the orders determined by paths in $\hasseo{\order{i}{V}}$, which are increasing in $i$, and proceed by induction.
Since $\dim{\order{0}{V}} = 1$, we have already proved the base case.

Let $i > 0$, assume that the statement holds of the orders $\order{i-1}{\prech}$ and $\order{i-1}{\precv}$ on $\grade{1}{(\order{i-1}{V})}$.
Let $x, y \in \grade{1}{(\order{i}{V})}$; we will show that $x$ and $y$ are comparable via $\order{i}{\prech}$ or $\order{i}{\precv}$.
If $x, y \in \order{i-1}{V}$ or $x, y \in \clset{\order{i}{x}}$, we can apply the 
inductive hypothesis or the atom case, so it suffices to consider the case 
\begin{equation*}
    x \in \order{i-1}{V} \setminus \faces{}{-}\order{i}{x}, \quad y \in \faces{}{+}\order{i}{x}.
\end{equation*}
Let $(\order{j}{z})_{j=1}^p$ be the unique linear 0\nbd ordering on $\bound{}{-}\order{i}{x}$, so $\order{j}{z} \prech \order{j'}{z}$ if $j \leq j'$.
For all $j \in \set{1, \ldots, p}$, we have $x \neq \order{j}{z}$, and by the inductive hypothesis $x$ and $\order{j}{z}$ are comparable via $\order{i-1}{\prech}$ or they are comparable via $\order{i-1}{\precv}$.
\begin{itemize}
    \item Suppose there exists $j$ such that $x$ and $\order{j}{z}$ are comparable via $\order{i-1}{\precv}$.
    Then necessarily $x \order{i-1}{\precv} \order{j}{z}$, because $\faces{}{-}\order{i}{x} \subseteq \faces{}{+}\order{i-1}{V}$.
    Since $\order{j}{z} \order{i}{\precv} y$, we have $x \order{i}{\precv} y$.
    \item Otherwise, $x$ and $\order{j}{z}$ are comparable via $\order{i-1}{\prech}$ for all $j \in \set{1, \ldots, p}$.
    Suppose that $x \order{i-1}{\prech} \order{1}{z}$, in which case $x \order{i-1}{\prech} \order{j}{z}$ for all $j \in \set{1, \ldots, p}$.
    Then the path from $x$ to $\order{1}{z}$ through elements of dimension 0 and 1 must enter $\order{1}{z}$ from $\bound{}{-}\order{1}{z} = \bound{0}{-}\order{i}{x}$.
    Since there is a path in $\hasseo{\grade{i}{V}}$ from $\bound{0}{-}\order{i}{x}$ to $y$, we have $x \order{i}{\prech} y$.
    \item Otherwise, there is a greatest $j$ such that $\order{j}{z} \order{i-1}{\prech} x$.
    If $j < p$, then $\order{j}{z} \order{i-1}{\prech} x \order{i-1}{\prech} \order{j+1}{z}$. Because all three are distinct, letting $\bound{}{+}\order{j}{z} = \bound{}{-}\order{j+1}{z} = \set{w}$, there is a non-trivial cycle in $\hasseo{\order{i-1}{V}}$ from $w$ to $x$ and back to $w$, a contradiction to Proposition \ref{prop:dim2_acyclic}.
    It follows that $\order{p}{z} \order{i-1}{\prech} x$, and the path between the two must leave $\order{p}{z}$ through $\bound{0}{+}\order{i}{x}$, so $y \order{i}{\prech} x$.
\end{itemize}
This proves that the union of $\order{i}{\prech}$ and $\order{i}{\precv}$ is a linear order on $\grade{1}{(\order{i}{V})}$.

Suppose that $x \order{i}{\prech} y$ and $x \order{i}{\precv} y$; we will prove that $x = y$.
If $x \in \clset{\order{i}{x}}$, then $x \order{i}{\precv} y$ implies that $y \in \clset{\order{i}{x}}$, and any path from $x$ to $y$ in $\hasseo{\order{i}{V}}$ is entirely contained in $\clset{\order{i}{x}}$, so we can apply the atom case.
Suppose that $x \notin \clset{\order{i}{x}}$.

If $y \in \faces{}{+}\order{i}{x}$, then any path from $x$ to $y$ through elements of dimension 1 and 2 consists of a path contained in $\order{i-1}{V}$ to some $z \in \faces{}{-}\order{i}{x}$ followed by the path $z \to \order{i}{x} \to y$; while any path through elements of dimension 0 and 1 consists of a path contained in $\order{i-1}{V}$ to $\bound{0}{-}\order{i}{x}$ followed by a path contained in $\bound{}{+}\order{i}{x}$.
Since there is a path from $\bound{0}{-}\order{i}{x}$ to any $z \in \faces{}{-}\order{i}{x}$ through elements of dimension 0 and 1 in $\order{i-1}{V}$, we have that $x \order{i-1}{\prech} z$ and $x \order{i-1}{\precv} z$ for some $z \in \faces{}{-}\order{i}{x}$.
By the inductive hypothesis, $x = z$, a contradiction since $z \in \clset{\order{i}{x}}$.

It follows that $y \in \order{i-1}{V}$.
Then any path from $x$ to $y$ through elements of dimension 1 and 2 is entirely contained in $\order{i-1}{V}$, so $x \order{i-1}{\precv} y$; while a path through elements of dimension 0 and 1 is either contained in $\order{i-1}{V}$, or it enters $\bound{}{+}\order{i}{x}$ through $\bound{0}{-}\order{i}{x}$, traverses it in its entirety, and leaves from $\bound{0}{+}\order{i}{x}$.
Such a path segment can be replaced with the one that traverses $\bound{}{-}\order{i}{x}$ in its entirety, so in either case $x \order{i-1}{\prech} y$, and we conclude that $x = y$ by the inductive hypothesis.
This concludes the proof of the statement for $\order{i}{V}$.
Since $\order{m}{V} = U$, we conclude.
\end{proof}

\begin{proof}[Proof of Lemma \ref{lem:dim2_inclusion_is_path_induced}]
Both $U$ and $\imath(V)$ are regular molecules of dimension $\leq 2$.
Let $\prech$, $\precv$ be the horizontal and vertical order on $U$, and $\prech^V$, $\precv^V$ those on $\imath(V)$, which are subsets of those on $U$.

Suppose by way of contradiction that $\flow{1}{V}$ is not path-induced. 
Then there exists a path $x_0 \to \ldots \to x_m$ in $\flow{1}{U}$ such that $m > 1$, $x_0, x_m \in \imath(V)$, and $x_i \notin \imath(V)$ for all $i \in \set{1, \ldots, m-1}$.
By definition, there exist 1\nbd dimensional elements $y_i \in \faces{}{+}x_{i-1} \cap \faces{}{-}x_{i}$ for all $i \in \set{1, \ldots, m}$.
Then $y_1 \precv y_m$ and $y_1 \neq y_m$.
By Lemma \ref{lem:codimension_1_elements}, $x_{i-1}$ and $x_i$ are the only cofaces of $y_i$ for all $i \in \set{1, \ldots, m}$.
Necessarily, then, $y_1 \in \faces{}{+}\imath(V)$ and $y_m \in \faces{}{-}\imath(V)$, 
so it is not possible that $y_1 \precv^V y_m$.
Then by Lemma \ref{lem:dim2_horizontal_or_vertical_path} applied to $\imath(V)$, one of $y_1 \prech^V y_m$, $y_m \prech^V y_1$, or $y_m \precv^V y_1$ must hold.
Combined with $y_1 \precv y_m$, each one of these implies $y_1 = y_m$ by Lemma \ref{lem:dim2_horizontal_or_vertical_path} applied to $U$, a contradiction.
\end{proof}

\begin{proof}[Proof of Theorem \ref{thm:dim3_frame_acyclic}]
By \cite[Theorem 1]{hadzihasanovic2021smash} $U$ is frame-acyclic.
Since for all $V \submol U$ and all rewritable $W \submol V$, $\subs{V}{\compos{W}}{W}$ is still a regular molecule of dimension $\leq 3$, it is frame-acyclic.
Hence $U$ is stably frame-acyclic.
\end{proof}

\end{document}